\documentclass[11pt]{article}
\usepackage{geometry}
\RequirePackage{amsthm,amsmath,amsfonts,amssymb}
\RequirePackage[numbers]{natbib}
\RequirePackage[colorlinks,citecolor=blue,urlcolor=blue]{hyperref}%% uncomment this for coloring bibliography citations and linked URLs
\RequirePackage{graphicx}%% uncomment this for including figures

\usepackage[shortlabels]{enumitem}
\usepackage{bbm,bm,nccmath}

\numberwithin{equation}{section}
\theoremstyle{plain}
\newtheorem{prop}{Proposition}[section]
\newtheorem{corollary}{Corollary}[section]
\newtheorem{theorem}{Theorem}[section]
\newtheorem*{theorem*}{Theorem}

\newtheorem{lemma}{Lemma}[section]

\theoremstyle{remark}
\newtheorem{defn}{Definition}[section]
\newtheorem{remark}{Remark}[section]
\newtheorem{example}{Example}[section]
\newtheorem{assumption}{Assumption}
\newenvironment{taggedassump}[1]
 {\renewcommand\theassumption{#1}\assumption}
 {\endassumption}

\newcommand{\beq}{\begin{equation}}
\newcommand{\eeq}{\end{equation}}
\def\beqs#1\eeqs{%
    \begin{equation}\begin{split}%
    #1%
    \end{split}\end{equation}%
}
\def\beqsn#1\eeqsn{%
    \begin{equation*}\begin{split}%
    #1%
    \end{split}\end{equation*}%
}
\def\beqsj #1\eeqsj
{
    \begin{equation}
    \setlength{\jot}{10pt}
    \begin{split}
    #1
    \end{split}\end{equation}
}
\newcommand{\Wbar}{V_0}
\DeclareMathOperator*{\argmin}{argmin}

\newcommand{\les}{\lesssim}

\newcommand{\A}{{\bf A}}
\newcommand{\rtri}{r_3}

\newcommand{\mpi}{\bar m}
\newcommand{\fbar }{f_0}

\newcommand{\Sigpi}{\Sigma}

\renewcommand{\c}[1]{a_{#1}}
\newcommand{\Var}{{\mathrm{Var}}}

\renewcommand{\l}{\left}
\renewcommand{\r}{\right}
\renewcommand{\H}{{\bf H}}
\newcommand{\mhat}{m^*}

\newcommand{\op}{{\text{op}}}
\newcommand{\barW}{\Wbar  }
\renewcommand{\S}{{\bf S}}
\newcommand{\KL}[2]{\mathrm{KL}(\,#1\;||\;#2)}
\newcommand{\e}{\,\mathrm{exp}}

\newcommand{\U}{{\mathcal U}}
\newcommand{\la}{\langle}
\newcommand{\ra}{\rangle}
\newcommand{\lla}{\l\la}
\newcommand{\rra}{\r\ra}
\newcommand{\R}{\mathbb R}
\newcommand{\E}{\mathbb E\,}

\newcommand{\TV}{\mathrm{TV}}

\newcommand{\mvi}[1]{\hat m}
\newcommand{\Svi}[1]{\hat S}

\newcommand{\mvigen}{\hat m}
\newcommand{\Svigen}{\hat S}

\newcommand{\cP}{\mathcal{P}}

\DeclareRobustCommand\Vbind{V}
\newcommand{\klopt}[1]{{\renewcommand\Vbind{{#1}}\eqref{kl-opt-intro}}}

\begin{document}

\author{
 Anya Katsevich\\
  \texttt{akatsevi@mit.edu}
  \and
  Philippe Rigollet\\
  \texttt{rigollet@math.mit.edu}
}
\title{On the Approximation Accuracy of \\Gaussian Variational Inference}

\maketitle
\begin{abstract}
The main computational challenge in Bayesian inference is to compute integrals against a high-dimensional posterior distribution. In the past decades, variational inference (VI) has emerged as a tractable approximation to these integrals, and a viable alternative to the more established paradigm of Markov Chain Monte Carlo. However, little is known about the approximation accuracy of VI. In this work, we bound the TV error and the mean and covariance approximation error of Gaussian VI in terms of dimension and sample size. %Our results indicate that Gaussian VI outperforms significantly the classical Gaussian approximation obtained from the ubiquitous Laplace method. 
Our error analysis relies on a Hermite series expansion of the log posterior whose first terms are precisely cancelled out by the first  order optimality conditions associated to the Gaussian VI optimization problem.
\end{abstract}

\section{Introduction}

A central challenge in Bayesian inference is to compute integrals against a posterior distribution $\pi$ on $\R^d$. The classical approach is to \emph{sample} the posterior using Markov Chain Monte Carlo (MCMC), in which a Markov chain designed to converge to $\pi$ is simulated for sufficiently long time. However, MCMC can be expensive, and it is notoriously difficult to identify clear-cut stopping criteria for the algorithm~\cite{mcmc-convergence}. An alternative, often computationally cheaper, approach is \emph{variational inference} (VI)~\cite{blei2017variational}. The idea of VI is to find, among all measures in a certain parameterized family $\cP$, the closest measure to $\pi$. While various notions of proximity have been proposed since the introduction of VI~\cite{daudeldouc2021mixtureopt,daudeldoucportier2021alphadiv}, we employ here  KL divergence, which is by far the most common choice. Typically, the family $\cP$ is selected so that integrals against measures in the family are either readily available or else easily computable. In this work, we consider the family of Gaussian distributions. We define \beq\label{VI-Gauss}\hat\pi =\mathcal N(\hat m,\hat S)\in\argmin_{p\in\mathcal P_{\text{Gauss}}}\KL{p}{\pi},\eeq where $\mathcal P_{\text{Gauss}}$ denotes the family of non-degenerate Gaussian distributions on $\R^d$. With the Gaussian approximation $\hat\pi$ in hand, we can now approximate $\int gd\pi$ using $\int fd\hat\pi$. For certain $g$ (e.g. polynomials), the latter integral can be computed in closed form in terms of $\hat m$ and $\hat S$. Otherwise, Gaussian sampling can be employed, which is much cheaper than MCMC for the original measure. 

%A fundamental problem in VI is to understand the tradeoff between accuracy and computational complexity. 

A key difference between MCMC and VI is that unbiased MCMC algorithms yield arbitrarily accurate samples from $\pi$ if they are run for long enough. On the other hand, the output of a perfect VI algorithm is $\hat\pi$, which is itself only an approximation to $\pi$. Therefore, a fundamental question in VI is to understand the quality of the approximation $\hat\pi\approx\pi$. In this work, we bound the approximation error $\int gd\pi - \int gd\hat\pi$ for a wide class of functions $g$, as well as several particular metrics of interest: the TV error $\TV(\pi,\hat\pi)$ and the mean and covariance errors $\|\hat m - \mpi\|$ and $\|\hat S - \Sigpi\|$, where $\bar m,\Sigma$ are the true mean and covariance of $\pi$, respectively.

Of course, we cannot expect an arbitrary, potentially multimodal $\pi$ to be well-approximated by a Gaussian distribution. In the setting of Bayesian inference, however, the Bernstein-von Mises theorem (BvM) guarantees that under certain regularity conditions, a posterior distribution converges to a Gaussian density in the limit of large sample size~\cite[Chapter 10]{van2000asymptotic}. To understand why this is the case, consider a generic posterior $\pi=\pi_n$ with density of the form 
\beq
\label{posterior}
\pi_n(\theta\mid x_{1:n})\propto \nu(\theta)\prod_{i=1}^n p_\theta(x_i)
\eeq 
Here, $\nu$ is the prior, $p_\theta$ is the model density, and $x_{1:n}=x_1,\dots, x_n$ are i.i.d observations. Provided $\nu$ and $p_\theta$ are everywhere positive, we can write $\pi_n$ as  \beq\label{logpost}\pi_n(\theta)\propto e^{-nv_n(\theta)},\qquad v_n(\theta) := -\frac1n\sum_{i=1}^n\log p_\theta(x_i) - \frac1n\log\nu(\theta).\eeq As $n$ increases, we expect that $v_n$ approaches $v_\infty$, the $n$-independent negative population log likelihood, a.k.a. the cross-entropy. Therefore, if $n$ is large and $v_\infty$ has a unique global minimum, then $\pi_n$ will place most of its mass in a neighborhood of this point. In other words, $\pi_n$ is effectively unimodal, and hence a Gaussian approximation is reasonable in this case.  This reasoning drives a second strategy to approximate posterior, whereby $\pi_n$ is a Gaussian centered at the global minimizer of $v_n$. While the accuracy of the Laplace approximation has been the subject of many works, including recent ones (see below), no comparable results exist for Gaussian VI. The main goal of this paper is to establish such approximation error bounds. 
%In particular, therefore, the mode $m_*$ can also serve as an approximation to the true mean $\mpi_n$. However, as we discuss below, Gaussian VI yields a more accurate estimate of $\mpi_n$.

 \paragraph*{Main contributions} In our above heuristic justification of a Gaussian approximation to $\pi=\pi_n$, we considered a sequence of functions $v_n$ converging to some $v_\infty$ as $n\to\infty$. However, our results do not require this setup. Instead, we simply consider a target measure of the form $\pi\propto e^{-nv}$ with $v\in C^4(\R^d)$, and we impose a requirement on the triple $(v,d,n)$. Namely, we assume $a_3(v)d/\sqrt n$ is smaller than a certain constant and $a_3(v)d/\sqrt n + a_4(v)d^2/ n\leq 1$, where $a_3(v)$ and $a_4(v)$ bound the size of the third and fourth derivatives of $v$ (see Assumption~\ref{assume:glob} for the precise definition). Our error bounds are in line with this condition on $(v,d,n)$, in the sense that the errors are small provided $a_3(v)d/\sqrt n + a_4(v)d^2/n$ is small. Up to this restriction, $v$ is free to depend on both $d$ and $n$.
 
Assuming also that $v$ has a unique global minimizer $\mhat$, that $\|\nabla^3v\|$ and $\|\nabla^4v\|$ grow at most polynomially away from $\mhat$, and that $v$ grows at least linearly away from $\mhat$, we show the following error bound for a wide class of functions $g$:
\beq\label{ferr}
\l|\int gd\pi - \int gd\hat\pi\r|\lesssim_{v}\sqrt{1\vee\Var_{\hat\pi}(g)} \begin{cases}
d/\sqrt n,\\
(d/\sqrt n)^2,\quad g\text{ even about }\mvi\pi,\\
(d/\sqrt n)^3,\quad g\text{ linear.}\end{cases}\eeq We have hidden the $v$-dependent factors $a_3(v),a_4(v)$ within the $\lesssim_v$ symbol for maximal clarity, though these factors may in principle affect the order of magnitude of the upper bound. See Theorem~\ref{thm:Vgen} for the precise statement of the result and observe that the mean a posteriori, i.e. Bayes estimator, is approximated at the fastest of these three rates. The second bound in~\eqref{ferr} is enabled by a leading order expansion of the error. Namely, we show that
\beq
\int gd\pi - \int gd\hat\pi = \int gQd\hat\pi+ \text{Rem}_g, \qquad |\text{Rem}_g|\lesssim_v\sqrt{1\vee\Var_{\hat\pi}(g)}\l(\frac{d}{\sqrt n}\r)^2,
\eeq
where $Q$ is a certain cubic polynomial which is odd about $\mvi\pi$. Thus if $g$ is even about $\mvi\pi$, then the leading order term $\int gQd\hat\pi$ vanishes. Similarly, the third bound in~\eqref{ferr} is enabled by a higher order expansion of the error. We note that having derived the leading order term not only provides valuable insight into the error, but also opens the door to \emph{correcting} the Gaussian measure $\hat\pi$ to get the more accurate approximation $(1+Q)\hat\pi$. However, we do not delve into this possibility here and postpone it for future research. The benefit of improving a Gaussian posterior approximation via a multiplicative correction was already studied in~\cite{katskew}, but for a different, \emph{Laplace} approximation to the posterior (defined in~\eqref{laplace}). %dpr{Is there a reference to your paper on this?}

This main result~\eqref{ferr} gives us a bound on $\int gd\pi-\int gd\hat\pi$ which is tailored to the observable $g$. We can also apply the result to derive error bounds for common Bayesian quantities of interest. For example, the following TV bound quantifies the error in using $\hat\pi$ to construct approximate credible sets for $\pi$:
\beq\label{TVintro}
\TV(\pi,\hat\pi)\lesssim_v\frac{d}{\sqrt n}.\eeq We also bound the mean and covariance approximation errors.  The posterior mean holds particular significance in Bayesian inference, since it is the Bayes-optimal estimator of the unknown parameter with respect to the squared error risk. The posterior covariance is another measure of the uncertainty in the posterior which can be used in addition to credible sets. To account for the vanishing variance (on the order of $n^{-1}$) of $\pi\propto e^{-nv}$, we rescale the below approximation errors appropriately:
\beq\label{meancovintro}
\sqrt n\|\hat m - \mpi\| \lesssim_v\l(\frac{d}{\sqrt n}\r)^{3},\qquad n\|\hat S - \Sigma\| \lesssim_v \l(\frac{d}{\sqrt n}\r)^2.\eeq

Together, the bounds~\eqref{ferr},~\eqref{TVintro}, and~\eqref{meancovintro} give us a comprehensive understanding of the error induced by using $\hat\pi$ as a proxy for $\pi$ in computing all of the relevant quantities of interest in Bayesian inference. In Section~\ref{sec:logreg}, we apply the general theory to formulate high-probability bounds on the VI approximation error for a posterior arising from logistic regression with Gaussian design. 

%

%he theorem shows that both the mean and covariance VI estimates, and especially the mean estimate $\hat m_{n}$, are remarkably accurate approximations to the true mean and covariance.  As such, it is a compelling endorsement of Gaussian VI for estimating the posterior mean and covariance in the finite sample regime. Although the condition $n\geq d^3$ is restrictive when $d$ is very large, we believe that it is unavoidable without further assumptions and note that it also appears in existing bounds for the Laplace method \cite{spokoiny}.
\paragraph*{Laplace approximation as benchmark for comparison}
As mentioned above, the Laplace method is another Gaussian approximation to $\pi$ that is widespread in practice for its computational simplicity. We use it as a benchmark to put the above error bounds into context. The Laplace approximation to $\pi\propto e^{-nv}$ is given by
\beq\label{laplace}\pi \approx \mathcal N(\mhat,\, (n\nabla^2v(\mhat))^{-1}),\eeq where $\mhat$ is the global minimizer of $v$. This approximation simply replaces $v$ by its second order Taylor expansion around $\mhat$. In recent work~\cite{katskew}, the first author shows that the TV error and covariance error of the Laplace approximation are of the same order as that of VI, in terms of both the $d$ and $n$ dependence. However, the VI mean error is two orders of magnitude smaller than that of Laplace: $(d/\sqrt n)^3$ for VI vs. $d/\sqrt n$ for Laplace. See Section~\ref{discuss} for a more detailed comparison of the approximation accuracy of the two methods. 

We emphasize that our aim in the present work is \emph{not} to make a value judgment about Gaussian VI in comparison to the Laplace approximation. Indeed, the accuracy of a method is not the only axis on which to judge its overall performance. An overall comparison of the two methods is a delicate undertaking that involves computational considerations and varies according to downstream statistical tasks. It is beyond the scope of the present work to produce a full comparison analysis but the approximation results should certainly be useful to consider in certain scenarios.

Instead, the main contribution of our work is to supply precise statements of Gaussian VI's approximation accuracy, an understanding of which has been virtually nonexistent, lagging far behind algorithmic developments in the literature. We compare VI to Laplace simply because the Laplace approximation is another popular approach to ease posterior computation via a Gaussian fit. 

Another way in which the Laplace approximation is relevant to VI is that the proof techniques developed by the first author in~\cite{katsTDD} to bound the Laplace approximation error also play a crucial role in the present work. Specifically, notice that the bounds~\eqref{ferr},~\eqref{TVintro}, and~\eqref{meancovintro} are all small provided $d^2\ll n$. This dependence on dimension is an improvement over the condition $d^3\ll n$ required in a long line of work on the BvM and the Laplace approximation accuracy~\cite{ghosal2000,spokoiny2013bernstein,lu2017bernstein, dehaene2019deterministic, helin2022non,spokoiny2023laplace}. Since then, the recent works~\cite{broderick-posterior} and~\cite{katsTDD} have developed new proof techniques which reduced the dependence to $d^2\ll n$, though only the latter work's proof techniques are flexible enough to derive leading order asymptotics in terms of $d/\sqrt n$.

%The recent works~\cite{spokoiny} and~\cite{broderick-posterior} derive error bounds for the Laplace approximation. Spokoiny~\cite{spokoiny} shows that $\sqrt n\|m_* - m_{\pi_n}\| \lesssim (d^3/n)^{1/2}$ assuming $v$ is strongly convex and $C^3$, and~\cite{broderick-posterior} similarly shows that $\sqrt n\|m_* - m_{\pi_n}\|\lesssim 1/\sqrt n$ with implicit dependence on $d$, assuming $v$ has sufficient growth at infinity. For the covariance approximation,~\cite{broderick-posterior} shows that $n\|(n\nabla^2v(m_*))^{-1}- S_{\pi_n}\| \lesssim 1/\sqrt n$. ~\cite{spokoiny} shows that this bound can be tightened to $n\|(n\nabla^2v(m_*))^{-1}- S_{\pi_n}\| \lesssim d^3/n$ if one assumes $v\in C^4$. 
 
We employ these same techniques in the present work. In addition, our proof is based on crucial new insights into why minimization of the KL objective~\eqref{VI-Gauss} leads to good approximation properties. We now briefly summarize these insights.
 \paragraph*{First-order optimality conditions and Hermite series expansion}
 %The improvement of Gaussian VI over the Laplace method to estimate the mean a posteriori 
 Our analysis of the approximation accuracy of Gaussian VI rests on a remarkable interaction between first-order optimality conditions and a Hermite series expansion of the potential $v$. 

Hereafter, we replace $\theta$ by $x$ and let $V=nv$, $\pi\propto e^{-V}$. The focal point of this work are the first order optimality equations for the minimization~\eqref{VI-Gauss}: $$\nabla_{m,S}\KL{\mathcal N(m,S)}{\pi}\big\vert_{(m,S)=(\mvigen,\Svigen)}=0.$$ This is also equivalent to setting the Bures-Wasserstein gradient of $\KL{p}{\pi}$ to zero at $p=\mathcal N(\hat m,\hat S)$ as  in~\cite{gauss-VI}. Explicitly, we obtain that $(m,S)=(\mvigen,\Svigen)$ is a solution to
\beq\label{kl-opt-intro}%\tag{$E_{\Vbind}$}
\E[\nabla V(m+S^{1/2} Z)]=0,\qquad \E[\nabla^2V(m+S^{1/2} Z)] = S^{-1},
\eeq
where $Z\sim\mathcal N(0, I_d)$ and $S^{1/2}$ is the positive definite symmetric square root of~$S$; see~\cite{gauss-VI} for this calculation. In some sense, the fact that $\hat\pi=\mathcal N(\mvigen, \Svigen)$ minimizes the KL divergence to $\pi$ does not explain why $\hat\pi$ approximates $\pi$ well. Rather, the true reason has to do with properties of solutions to the fixed point equations~\eqref{kl-opt-intro}. %can be derived by setting $$\nabla_{m,S}\KL{\mathcal N(m,S)}{\pi}\big\vert_{(m,S)=(\hat m,\hat S)}=0$$ or equivalently, by setting the $\mathrm{BW}(\R^d)$ gradient of $\KL{p}{\pi}$ to zero at $p=\mathcal N(\hat m,\hat S)$. 

To see why, consider the function $\Wbar(x) =  V(\mvigen+\Svigen^{1/2}x)$. If $\pi\propto e^{-V}$ is close to the density of $\mathcal N(\mvigen, \Svigen)$, then $\rho\propto e^{-\Wbar}$ should be close to the density of $\mathcal N(0, I_d)$. In other words, we should have that $\Wbar(x)\approx\mathrm{const.}+\|x\|^2/2$. This is ensured by the first order optimality equations~\eqref{kl-opt-intro}. Indeed, note that~\eqref{kl-opt-intro} can be written in terms of $\Wbar$ as
\beq\label{kl-opt-intro-rescale}\E[\nabla\Wbar(Z)]=0,\qquad \E[\nabla^2\Wbar(Z)]=I_d.\eeq As we explain in Section~\ref{subsec:outline}, the equations~\eqref{kl-opt-intro-rescale} set the first and second order coefficients in the Hermite series expansion of $\Wbar$ to $0$ and $I_d$, respectively. As a result, $\Wbar(x) -\|x\|^2/2= \mathrm{const.}+r_3(x)$, where $r_3$ is a Hermite series containing only third and higher order Hermite polynomials. The accuracy of the Gaussian VI approximation for integrating linear and quadratic functions stems from the fact that the Hermite remainder $r_3$ is of order $r_3\sim 1/\sqrt n$, and the fact that $r_3$ is \emph{orthogonal to linear and quadratic functions with respect to the Gaussian measure}. See Section~\ref{subsec:outline} for a high-level summary of this Hermite series based error analysis.

 \paragraph*{Related Work}
The literature on VI can be roughly divided into statistical and algorithmic works. Works on the statistical side have focused on the contraction of variational posteriors around a ground truth parameter in the large $n$ (sample size) regime. (We use ``variational posterior" as an abbreviation for variational approximation to the posterior.) For example,~\cite{wang2019frequentist} prove an analogue of the Bernstein-von Mises theorem for the variational posterior,~\cite{zhang2020convergence} study the contraction rate of the variational posterior around the ground truth in a nonparametric setting, and~\cite{alquier2020concentration} study the contraction rate of variational approximations to tempered posteriors, in high dimensions. 

A key difference between these works and ours is that here, we determine how well the statistics of the variational posterior match those of the posterior itself, rather than those of a limiting ($n\to\infty$) distribution. %The work of~\cite{han2019statistical} is related to ours in that they s
We are only aware of one other work studying the problem of quantifying posterior approximation accuracy of VI. In \cite{han2019statistical}, the authors consider using the mean $\hat m$ of the mean-field variational posterior as a proxy for the maximum likelihood estimate in a Bayesian latent variable model. They show that $\sqrt n\|\hat m - \mathrm{MLE}\|\lesssim 1/n^{1/4}$.

On the algorithmic side of Gaussian VI, we refer the reader to the works~\cite{gauss-VI,diao2023forwardbackward} and references therein. These works complement our analysis in that they provides rigorous convergence guarantees for several algorithms that solves the optimization problem~\eqref{VI-Gauss}. See also the discussion in Section~\ref{discuss}.

Next, we mention a closely related approximation method called expectation propagation (EP). In Gaussian EP, one attempts to minimize the KL divergence $\KL{\pi}{p}$ over all Gaussian measures $p$. (Note that the order of $\pi$ and $p$ has been switched relative to~\eqref{VI-Gauss}.) Minimizing this objective amounts to matching the first and second moments of $p$ to those of $\pi$. It is not possible to minimize this objective directly without resorting to computing integrals against $\pi$. So instead, EP works by iteratively constructing hybrids between the Gaussian distribution and $\pi$, and matching the moments of the hybrid (which are easier to compute) to those of $\pi$. The works~\cite{dehaene2018expectation} and~\cite{dehaene2015bounding} studied the mean and covariance approximation accuracy of EP with respect to $n$. Interestingly enough, the authors showed EP has the exact same mean and covariance accuracy as shown here for traditional Gaussian VI: $\|\mvi\pi-\mpi\|\les n^{-2}$ and $\|\Svi\pi- \Sigma\|\les n^{-2}$ (compare to~\eqref{meancovintro}). Whether this is a coincidence or a result of a deeper connection between the two methods is an interesting avenue to explore.

%\blue{To do: cite ``Asymptotic normality and valid inference for Gaussian variational approximation", ``Gaussian Variational Approximate Inference for Generalized Linear Mixed Models", and ``Theory of Gaussian variational approximation for aPoisson mixed model".}
Finally, we note that~\cite{weng2010bayesian} also uses Hermite polynomials to provide an approximation to the posterior. The approximation, given by an Edgeworth series, is effectively a perturbation of a Gaussian measure. The coefficients of the series depend on the cumulants of the posterior distribution. Thus writing down the series requires computing cumulants, which come from integrals against the posterior. Therefore, the expansion cannot itself be directly used as a tool to help approximate integrals against the posterior, which is the goal in our work.   

 %Edgeworth expansions also use Hermite polynomials another way to quantify the deviation of a posterior distribution from Gaussian

%are another kind of Gaussian approximation  

\subsection*{Organization of the paper}The rest of the paper is organized as follows. In Section~\ref{sec:main} we state our assumptions and main results on the Gaussian VI approximation accuracy. In Section~\ref{sec:logreg}, we apply the general theory to a logistic regression posterior. In Section~\ref{sec:proof}, we outline the proof of the main theorems. Finally in Section~\ref{sec:m-sig-exist}, we prove the existence and uniqueness of solutions to the first-order optimality conditions~\eqref{kl-opt-intro}. Omitted proofs can be found in the Appendix.
\subsection*{Notation} The notation $a\les b$ means there is constant $C>0$ such that $a\leq Cb$. Unless indicated otherwise, the suppressed constant is absolute. We let $\gamma$ denote the density of the standard Gaussian distribution $\mathcal N(0, I_d)$ in $d$ dimensions, and we write either $\int fd\gamma$ or $\E[f(Z)]$ or $\gamma(f)$ for the expectation of $f$ under $\gamma$. We write $Z$ to denote a standard multivariate Gaussian random variable $Z\sim\gamma$ in $\R^d$. We let $\Var_{\hat\pi}(f)=\int (f-\hat\pi(f))^2d\hat\pi$, and $\|f\|_2 = (\int f^2d\gamma)^{1/2}$. 

A tensor $T$ of order $k$ is an array $T=(T_{i_1i_2\dots i_k})_{i_1,\dots,i_k=1}^d$. For two order $k$ tensors $T$ and $S$ we let $\la T, S\ra$ be the entrywise inner product. We say $T$ is symmetric if $T_{i_1\dots i_k}= T_{j_1\dots j_k}$, for all permutations $j_1\dots j_k$ of $i_1\dots i_k$. 

Let $H$ be a symmetric positive definite matrix. For a vector $x\in\R^d$, we let $\|x\|_H$ denote $\|x\|_H = \sqrt{x^THx}$.  For an order $k\geq2$ symmetric tensor $T$, we define the $H$-weighted operator norm of $T$ to be
\beq\label{TH}\|T\|_H:=\sup_{\|u\|_H=1}\la T, u^{\otimes k}\ra .\eeq%= \sup_{\|u\|_H=1}\sum_{i_1,\dots,i_k=1}^dT_{i_1i_2\dots i_k}u_{i_1}u_{i_2}\dots u_{i_k}\eeq  
By Theorem 2.1 of~\cite{symmtens}, for symmetric tensors, the definition~\eqref{TH} coincides with the standard definition of operator norm:
$$\sup_{\|u_1\|_H=\dots=\|u_k\|_H=1} \la T, u_1\otimes\dots\otimes u_k\ra = \|T\|_H=\sup_{\|u\|_H=1}\la T, u^{\otimes k}\ra.$$ 
When $H=I_d$, the norm $\|T\|_{I_d}$ is the regular operator norm, and in this case we omit the subscript. %For a symmetric, order 3 tensor $T$ and a symmetric matrix $M$, we let $\la T, M\ra\in\R^d$ be the vector with coordinates
%\beq
%\la T, M\ra_i=\sum_{j,k=1}^dT_{ijk}M_{jk},\quad i=1,\dots,d.\eeq

\section{Statement of Main Results}\label{sec:main}
In light of the centrality of the fixed point equations~\eqref{kl-opt-intro}, we begin the section by redefining $(\mvigen, \Svigen)$ as solutions to~\eqref{kl-opt-intro} rather than minimizers of the KL divergence objective~\eqref{VI-Gauss}. These definitions diverge only in the case that $V$ is not strongly convex. Indeed, if $V$ is strongly convex then $\KL{\cdot}{\pi}$ is strongly geodesically convex in the submanifold of Gaussian distributions; see, e.g., \cite{gauss-VI}. Therefore, in this case, $\KL{\cdot}{\pi}$ admits a unique minimizer $\hat \pi$ over this submanifold, corresponding to a unique solution $(\mvigen,\Svigen)\in\R^d\times\S^d_{++}$ to~\eqref{kl-opt-intro}. In general, however, there spurious solutions $(m,S)$ o the first order equations~\eqref{kl-opt-intro} that may have poor approximation properties. To see this, consider the equations in the following form, recalling that $v=V/n$:
\beq\label{kl-opt-intro-ii}
\E[\nabla v(m+S^{1/2} Z)]=0,\qquad S\,\E[\nabla^2v(m+S^{1/2} Z)]= \frac 1nI_d.
\eeq
Let $x\neq m_*$ be a critical point of $v$, that is, $\nabla v(x)=0$, and consider the pair $(m,S)=(x, 0)$. For this $(m,S)$ we have $$\E[\nabla v(m+S^{1/2} Z)]=\nabla v(x)=0, \qquad S\,\E[\nabla^2v(m+S^{1/2} Z)]=0\approx \frac1nI_d.$$ Thus $(x,0)$ is an approximate solution to~\eqref{kl-opt-intro-ii}, and by continuity, we expect that there is an exact solution nearby. In other words, to each critical point $x$ of $v$ is associated a solution $(m,S)\approx (x,0)$ of~\eqref{kl-opt-intro-ii}. The solution $(m,S)$ of~\eqref{kl-opt-intro-ii} which we are interested in, then, is the one near $(m_*,0)$. Lemma~\ref{lma:intro:exist:V} below formalizes this intuition; we show there is a unique solution $(m,S)$ to~\eqref{kl-opt-intro} in the set
%\beq\label{Rpi}\mathcal R_V= \l\{(m,\sigma)\in\R^d\times \S^{d}_{++} \; :\; \sigma^2 \preceq \frac{2}{n}H^{-1},\; \|\sqrt H\sigma\|^2 + \|\sqrt H(m-m_*)\|^2\leq \frac{8}{n}\r\}\ee
\beqs\mathcal R_V= \bigg\{(m,S)\in\R^d\times &\S^{d}_{++} \; :\; S \preceq 2H_V^{-1},\, \| H_V^{1/2}S^{1/2}\|^2 + \| H_V^{1/2}(m-m_*)\|^2\leq 8\bigg\},\eeqs 
where $H_V=\nabla^2V(m_*)=n\nabla^2v(m_*)$. Note that due to the scaling of $H_V$ with $n$, the set $\mathcal R_V$ is a small neighborhood of $(m_*, 0)$. We call this unique solution $(m, S)$ in $\mathcal R_V$ the ``canonical" solution of~\eqref{kl-opt-intro}. We expect the Gaussian distribution corresponding to this canonical solution to be the minimizer of~\eqref{VI-Gauss}, although we have not proved this. Regardless of whether it is true, we will redefine $(\mvigen, \Svigen)$ to denote the canonical solution. Indeed, whether or not $\hat\pi=\mathcal N(\mvigen, \Svigen)$ actually minimizes the KL divergence or is only a local minimizer is immaterial for the purpose of approximating $\pi$. 

In the rest of this section, we state our assumptions on $v$, a lemma guaranteeing a canonical solution $\mvigen, \Svigen$ to~\eqref{kl-opt-intro}, and our main results on the approximation accuracy of Gaussian VI .

\subsection{Assumptions on the potential}\label{assumptions}
\begin{taggedassump}{1}[Regularity and unique strict minimum]\label{assume:1}
The potential $v\in C^4$, with unique global minimizer $x=\mhat $, and $H_v=\nabla^2v(\mhat )\succ 0$.
\end{taggedassump}
 \begin{taggedassump}{2}[Polynomial growth of $\|\nabla^kv\|_{H_v}$, $k=3,4$.]\label{assume:glob} There is some $q,a_3,a_4>0$ such that
\begin{align}
&\|\nabla^kv(x)\|_{H_v}\leq a_k\l(1\vee\sqrt {n/d}\|x-\mhat\|_{H_v}\r)^q\quad \forall x\in\R^d,\quad k=3,4,\label{globnabla}\\
%\eeq  Furthermore,
%\beqs
%&a_3d/\sqrt n\leq C(q) <1,\qquad a_3d/\sqrt n+a_4d^2/n\leq 1\label{cconds}
&a_3\frac{d}{\sqrt n}\leq C(q),\qquad a_3\frac{d}{\sqrt n}+a_4\frac{d^2}{ n}\leq 1\label{cconds}
\end{align} for a certain constant $C(q)<1$ depending only on $q$.
\end{taggedassump} 
The $\sqrt{n/d}$ factor in~\eqref{globnabla} is there to ``zoom in'' on the $\sqrt{d/n}$ scale, which is the scale at which the measure $\pi\propto e^{-nv}$ concentrates. 
\begin{taggedassump}{3}[Lower bound on growth of $v$]\label{assume:c0} There exists $c_0>0$ such that 
\beq\label{c0cond}
c_0\sqrt {d/n}\|x-\mhat\|_{H_v} \leq v(x) - v(\mhat) \qquad \forall \, x\,:\, \|x-\mhat\|_{H_v}\geq \frac12\sqrt{d/n}.\eeq  \end{taggedassump}
See Section~\ref{discuss} for a discussion of the assumptions.
\subsection{Main Results}
We are now in a position to state our main results. In this section, we use the notation $\epsilon=d/\sqrt n$. First, we characterize the Gaussian VI parameters $(\mvi\pi, \Svi\pi)$:
\begin{lemma}\label{lma:intro:exist:V}
Let Assumptions~\ref{assume:1} and~\ref{assume:glob} be satisfied and let $H_V = \nabla^2V(m_*)=n\nabla^2v(m_*)$. Then there exists a unique $(m,S)=(\mvi\pi,\, \Svi\pi)$ in the set $$\mathcal R_V= \l\{(m,S)\in\R^d\times \S^{d}_{++} \; :\; S\preceq 2H_V^{-1},\; \| {H_V}^{1/2} S^{1/2}\|^2 + \|{H_V}^{1/2}(m-m_*)\|^2\leq 8\r\}$$ which solves 
$$\E[\nabla V(m+S^{1/2}Z)]=0,\qquad\E[\nabla^2V(m+S^{1/2}Z)]=S^{-1}.$$
%~\eqref{kl-opt-intro}. 
%both $\E[\nabla v(\sigma Z+m)]=0$ and $\E[\nabla^2v(\sigma Z+m)] = n^{-1}\sigma^{-2}$.
Moreover, $\Svi\pi$ satisfies
 \beq\label{Svi-bds}\frac23H_V^{-1}\preceq \Svi\pi \preceq 2H_V^{-1}.\eeq
\end{lemma} For the proof of this lemma, see Section~\ref{sec:m-sig-exist}. We now state our bounds on the error in using $\hat\pi=\mathcal N(\mvi\pi,\Svi\pi)$ to approximate expectations under the true measure $\pi$. The main theorem is the following.
%Recall the parameters $R_0,c_0$ from Assumption~\ref{assume:c0}.
\begin{theorem}\label{thm:Vgen}
Suppose $v$ satisfies Assumptions~\ref{assume:1},~\ref{assume:glob}, and~\ref{assume:c0}, and let $\mvi\pi,\Svi\pi$ be as in Lemma~\ref{lma:intro:exist:V}. Let $g$ be a function satisfying the following inequality for some $R_g>0$:
 \beq\label{gcond}|g(x)-\int gd\hat\pi|\leq \e\l(c_0\sqrt d\|x-\mvi\pi\|_{\Svi\pi^{-1}}/4\r),\quad\forall\, x\,:\, \|x-\mvi\pi\|_{\Svi\pi^{-1}}\geq R_g\sqrt d.\eeq %Let $\Delta_g=\int gd\pi-\int gd\hat\pi$. 
 Then
\begin{align}
\l|\medint\int gd\pi - \medint\int gd\hat\pi\r|\les \l(1+\Var_{\hat\pi}(g)^{\frac12}\r)\l(a_3\epsilon + a_4\epsilon^2\r),\label{overallbd1V}%\\
\end{align} If additionally, $g$ is even about $\mvi\pi$, then
\beq\label{overallbd3V}
\l|\medint\int gd\pi - \medint\int gd\hat\pi\r|\les  \l(1+\Var_{\hat\pi}(g)^{\frac12}\r)(a_3^2+  a_4)\epsilon^2.
\eeq Finally, if $g$ is linear, then
\beq\label{overallbd4V} \l|\medint\int gd\pi - \medint\int gd\hat\pi\r|\les \l(1+\Var_{\hat\pi}(g)^{\frac12}\r)\l(\l(a_3^3+ a_3a_4\r)\epsilon^3 +  a_4^2\epsilon^4\r).\eeq In all the above bounds, the suppressed constant is an increasing function of $q$, $c_0^{-1}$, and $R_g$.
\end{theorem}
The improved bound~\eqref{overallbd3V} actually holds for a somewhat wider class of functions $g$; see Remark~\ref{rem:moreg}. For a discussion of the condition~\eqref{gcond}, see Section~\ref{discuss}. \begin{theorem}[Leading order term in VI approximation error]\label{thm:corr}
Suppose $v$ satisfies Assumptions~\ref{assume:1},~\ref{assume:glob}, and~\ref{assume:c0}. Define the function
\beq\label{Qdef}Q(x) =\lla  \E_{X\sim\hat\pi}[\nabla^3V(X)], \;\;\frac12(x-\mvi\pi)\otimes\Svi\pi-\frac16(x-\mvi\pi)^{\otimes3}\rra.\eeq If $g$ satisfies~\eqref{gcond}, then
\beq\label{overallbd2V}
\l|\int gd\pi -\int gd\hat\pi - \int gQd\hat\pi\r|\les \l(1+\Var_{\hat\pi}(g)^{\frac12}\r)( \c3^2+  \c4)\epsilon^2.
\eeq  The suppressed constant is an increasing function of $q$, $c_0^{-1}$, and $R_g$.
\end{theorem}
See Section~\ref{discuss} for further insight into the leading order term $\int gQd\hat\pi$.
\begin{remark}\label{rem:moreg}
Note that the bound~\eqref{overallbd3V} from Theorem~\ref{thm:Vgen} is an immediate corollary of Theorem~\ref{thm:corr}, since $\int gQd\hat\pi=0$ if $g$ is even about $\mvi\pi$.
\end{remark}
Next, we bound the TV error. Let $\mathcal B(\R^d)$ be the set of all Borel sets of $\R^d$, and $\mathcal S_{\mvi\pi}(\R^d)$ be the set of Borel sets symmetric about $\mvi\pi$.    
\begin{corollary}[TV error]\label{corr:TV}Suppose $v$ satisfies Assumptions~\ref{assume:1},~\ref{assume:glob}, and~\ref{assume:c0}. Then
\beqs
\sup_{A\in \mathcal B(\R^d)}\l|\pi(A)-\hat\pi(A)\r|&\lesssim  a_3\epsilon +  a_4\epsilon^2,\\
\sup_{A\in \mathcal S_{\mvi\pi}(\R^d)}\l|\pi(A)-\hat\pi(A)\r|&\lesssim \l( a_3^2+ a_4\r)\epsilon^2.
\eeqs The suppressed constant is an increasing function of $q$ and $c_0^{-1}$.
\end{corollary}  
\begin{corollary}[Mean and covariance error]\label{corr:meancov}Let $\mpi = \int xd\pi(x)$ be the mean of $\pi$, and $\Sigpi = \int (x-\mpi)(x-\mpi)^Td\pi(x)$ be the covariance of $\pi$. Suppose $v$ satisfies Assumptions~\ref{assume:1},~\ref{assume:glob}, and~\ref{assume:c0}. Then
\beqs
\|\Svi\pi^{-1/2}(\mpi - \mvi\pi)\| &\les \l(a_3^3+ a_3a_4\r)\epsilon^3 +  a_4^2\epsilon^4,\\
\|\Svi\pi^{-1/2}(\Sigpi -\Svi\pi)\Svi\pi^{-1/2}\|&\les \l( a_3^2+ a_4\r)\epsilon^2.
\eeqs
The suppressed constant is an increasing function of $q$ and $c_0^{-1}$.
\end{corollary} 
\begin{remark}
Recall from~\eqref{Svi-bds} that $\Svi\pi$ is equivalent to $H^{-1}$ with respect to the positive definite matrix ordering. Therefore, given a lower bound $n\alpha\preceq H$, we can translate the bounds in Corollary~\ref{corr:meancov} into analogous, unweighted bounds:
\beqsn
\|\mpi-\mvi\pi\|&\les \frac{1}{\sqrt{n\alpha}}\l[\l( \c3^3+ \c3 \c4\r)\frac{d^3}{n\sqrt n} +  \c4^2\frac{d^4}{n^2}\r],\\
\|\Sigpi-\Svi\pi\|&\les \frac{1}{n\alpha}\l[\l( \c3^2+ \c4\r)\frac{d^2}{n}\r]
\eeqsn
\end{remark}
\subsection{Discussion}\label{discuss}
\paragraph*{Assumptions on $v$ and $g$}Recall that in Bayesian inference, $v$ is the normalized negative log posterior, taking the form~\eqref{logpost}, where the $x_i$ are i.i.d. In typical situations, we expect that when $n$ is large, the prior has negligible effect and $v(\theta)\approx-\E\log p_\theta(X)$, in which case $v$ depends only on $d$. However, our results do not only apply to this typical regime. The function $v$ is allowed to depend on both $n$  and $d$, so that the constants from the assumptions --- $a_3,a_4,q,c_0$ --- may also depend on both $n$ and $d$. Our results hold as long as these quantities and $d$ and $n$ are constrained by~\eqref{cconds}. However, for the logistic regression example, we will see that $a_3,a_4,q, c_0$ can all be chosen to be absolute constants. 

%In Assumption~\ref{assume:glob}, there is flexibility in how to choose $a_3$, $a_4$, and $q$. For example, if $\sup_{x\in\R^d}\|\nabla^kv(x)\|_{H_v}=a_{k,\infty}<\infty$ for $k=3,4$, then we can of course take $q=0$ and $a_k=a_{k,\infty}$. However, if e.g. $\|\nabla^3v(x)\|_{H_v}$ is smaller when $x$ is near $\mhat$ and larger when $x$ is farther away from $\mhat$, it may be possible to satisfy~\eqref{globnabla} with $q=1$ and a value of $a_3$ that is smaller than $a_{3,\infty}$.

When $v$ is convex, Assumption~\ref{assume:c0} automatically follows from~\eqref{cconds}. Indeed, we have the following lemma.
\begin{lemma}\label{lma:c0convex}
Let $v$ be convex and suppose $a_3(1+(1/2)^q)\sqrt{d/n}\leq 3$. Then $v$ satisfies Assumption~\ref{assume:c0} with $c_0=1/8$.
\end{lemma} See Appendix~\ref{app:sec:main} for the proof. If $v$ is not convex, note that the condition~\eqref{c0cond} is a nontrivial requirement only for larger $x$. To see this, fix the largest possible $r$ such that $a_3r(1+r^q)\sqrt{d/n}\leq 3/2$. Then a straightforward Taylor expansion shows that
\beqsn
v(x)-v(\mhat)\geq \frac14\|x-\mhat\|_{H_v}^2\geq \frac18\sqrt{\frac dn}\|x-\mhat\|_{H_v}\eeqsn for all $x$ such that $(1/2)\sqrt{d/n}\leq\|x-\mhat\|_{H_v}\leq r\sqrt{d/n}$. Therefore,~\eqref{c0cond} need only be checked for $\|x-\mhat\|_{H_v}\geq r\sqrt{d/n}$.\\
 
Next, let us discuss the condition~\eqref{gcond} on the function $g$ in Theorem~\ref{thm:Vgen}. The condition is easiest to check when $g$ is of the form $g(x)=f(\Svi\pi^{-1/2}(x-\mvi\pi))$. In this case, we have the following sufficient condition.
\begin{lemma}\label{lma:gsuff}Let $g(x)=f(\Svi\pi^{-1/2}(x-\mvi\pi))$ and suppose $f$ satisfies 
\beq\label{fcondeasy}|f(y)-f(0)|\leq e^{C_f\sqrt d\|y\|^\alpha}\quad\forall y\in\R^d,\eeq for some $0\leq\alpha<1$ and $C_f>0$. Then for some function $R_g=R_g(C_f,\alpha,c_0^{-1})$ which is increasing in each of the variables, the condition~\eqref{gcond} is satisfied.
\end{lemma} See Appendix~\ref{supp:aux} for the proof.

\paragraph*{Leading order term: VI vs Laplace} We can gain insight into the difference between the Laplace and VI approximations by considering the respective leading order terms (LOTs) of the difference $\int gd\pi - \int gd\hat\pi$, where $\hat\pi$ is either the Laplace or VI Gaussian approximation. To make this comparison, we first make note of an alternative representation of the VI leading order term $\int gQd\hat\pi$. Let \beq\label{WbarVI}\Wbar(x)=V(\mvi\pi+\Svi\pi^{1/2}x),\qquad f(x)=g(\mvi\pi+\Svi\pi^{1/2}x),\eeq and let $\H_3$ be the tensor of third order Hermite polynomials. See Appendix~\ref{hermite-tail} for a primer on Hermite polynomials in dimension $d\geq1$. Let $\A_3(f)=\E[f(Z)\H_3(Z)]$ and $\A_3(\Wbar)=\E[\Wbar(Z)\H_3(Z)]=\E[\nabla^3\Wbar(Z)]$, where the second representation is by Gaussian integration by parts. (We don't write $\A_3(f)$ this way because $f$ need not be $C^3$.) It can be shown (see Theorem~\ref{thm:W} and Remark~\ref{rk:A3fV}) that 
$$\int gQd\hat\pi = -\frac{1}{3!}\lla \A_3(\Wbar),\;\;\A_3(f)\rra = -\frac{1}{6}\lla\E\l[\nabla^3\Wbar(Z)\r], \E\l[f(Z)\H_3(Z)\r]\rra.$$% Therefore, if $f$ is orthogonal to all third order Hermite polynomials, then $\A_3(f)=0$ and the leading order term in $\Delta_g$ vanishes. This happens if $f$ is even, or if $f$ is a linear or quadratic polynomial (equivalently, if $g$ is even about $\mvi\pi$, or $g$ is linear or quadratic).

In contrast, it is shown in~\cite{katskew} that the leading order term of the difference $\int gd\pi - \int gd\hat\pi$ for the Laplace approximation $\hat\pi=\mathcal N(\mhat, H_V^{-1})$ takes the form $ -\frac{1}{6}\la\nabla^3\Wbar(0),\E\l[f(Z)Z^{\otimes3}\r]\ra$ where now $\Wbar,f$ are defined as \beq\label{WbarLap}\Wbar(x)=V(\mhat+H_V^{-1/2}x),\qquad f(x)=g(\mhat+H_V^{-1/2}x).\eeq To summarize, we have the following errors, informally:
\begin{align*}
\int gd\pi - \int gd\hat\pi\; &=  -\frac{1}{6}\lla\E\l[\nabla^3\Wbar(Z)\r],\; \E\l[f(Z)\H_3(Z)\r]\rra + \mathcal O\l(\frac{d^2}{n}\r)\tag{VI}\\
% \Wbar&=W\circ T_{\mathrm{VI}},\;f =g\circ T_{\mathrm{VI}}\\
\int gd\pi - \int gd\hat\pi\; &=  -\frac{1}{6}\lla\nabla^3\Wbar(0),\;\E\l[f(Z)Z^{\otimes3}\r]\rra+ \mathcal O\l(\frac{d^2}{n}\r)\tag{Laplace}\\
%\Wbar&=W\circ T_{\mathrm{L}},\; f=g\circ T_{\mathrm{L}}
\end{align*}
Here, $\Wbar, f$ are given by~\eqref{WbarVI} in the first line, and by~\eqref{WbarLap} in the second line. The remainder term also depends on $v$ and $g$, but we have written $\mathcal O(d^2/n)$ for brevity. We now compare the two approximation errors for various functions $f$. The most natural way to do so is to compare the two LOTs for the same function $f$, keeping in mind that this $f$ corresponds to two different functions $g$ for the two methods. 

Both of the two LOTs can be shown to have order $d/\sqrt n$, generically. In particular, if $f$ is an indicator, then neither LOT vanishes. This explains why $\TV(\pi,\hat\pi)\les d/\sqrt n$ for both VI and Laplace. Both the VI and Laplace LOT vanish when $f$ is even. In particular, this explains why the Laplace and VI covariance errors are both bounded by $d^2/n$. 

Finally, the key difference between Laplace and VI is that for functions $f$ which are orthogonal to all third order Hermite polynomials but which are \emph{not} even, the VI LOT vanishes while the Laplace LOT does not. Probably the most important example of such a function is $f(x)=x$. This explains VI's superior mean approximation accuracy. (The fact that the error is $\mathcal O((d/\sqrt n)^3)$ rather than $\mathcal O((d/\sqrt n)^2)$ stems from a cancellation in the next order term.) Another example of such an $f$ is a fifth order Hermite polynomial.

We remark that having the LOT allows one to prove lower bounds in addition to upper bounds. In~\cite{katskew}, we showed that for a posterior stemming from logistic regression, the Laplace mean approximation error is \emph{lower}-bounded by $d/\sqrt n$, giving definitive proof that the VI mean approximation error is better in general. See Section 2.2 of~\cite{katskew} for the precise statements of the Laplace error bounds summarized above. 

\paragraph*{Solution to first-order optimality conditions~\eqref{kl-opt-intro} vs solution to~\eqref{VI-Gauss}}Recall that in our results, $(\mvi\pi,\Svi\pi)$ is a solution to~\eqref{kl-opt-intro} for which $\mvi\pi$ is in a neighborhood of $\mhat$, the mode of the posterior $\pi$. In the case that the posterior is multimodal, $(\mvi\pi,\Svi\pi)$ may not be the global minimizer of the KL objective~\eqref{VI-Gauss}. This may occur if the posterior landscape has a single, ``correct" dominant mode as well as a ``false" second mode with a larger density value, but which is nevertheless a low probability region. Such a low probability mode could arise due to outliers in the data, for example. In this setting, the solution to~\eqref{kl-opt-intro} would lie near the false mode, while the solution to~\eqref{VI-Gauss} would lie near the correct mode. It would therefore be useful to study the global minimizer of the KL objective in this regime, rather than the critical point near the highest mode, as we do in the present work. Moreover, this is potentially a case in which VI may be preferable over the Laplace approximation. Intuitively, we expect that VI is stable to low probability perturbations, while Laplace is not.

We leave the question of VI's performance under this kind of multimodality to future work.
 %Our assumptions likely ensure that the largest mode is also the highest-probability mode, and hence that the solution to~\eqref{kl-opt-intro} coincide 
Note that our assumptions do allow for multimodality, as long as the negative log posterior grows linearly away from the minimum. Our results show that in this case, the solution to~\eqref{kl-opt-intro} near the mode is a good approximation to the posterior, whether or not this solution coincides with the global minimizer of the KL objective.

%Since we have shown that the TV distance between $\pi$ and $\hat\pi$ is small under our assumptions, we conjecture that our assumptions imply that (1) the critical point $(\mvi\pi,\Svi\pi)$ corresponds to the highest-probability mode of the posterior, and (2) this critical point coincides with the global minimizer of the KL objective.

\paragraph*{Computation}
Although the computational side of Gaussian VI is not our focus, for the sake of completeness we summarize the known computational guarantees on Gaussian VI algorithms. The typical algorithm iteratively updates the mean and covariance matrix according to some optimization procedure. For example,~\cite{gauss-VI} studies the discretization of an ODE evolving $(m_t,\Sigma_t)$, which corresponds to a Wasserstein gradient flow on $\KL{\cdot}{\pi}$ constrained to the submanifold of normal distributions. The authors show that if $V$ is strongly convex, then the gradient flow (as well as its discretization) converges exponentially fast in Wasserstein distance to $\hat \pi$ of~\eqref{VI-Gauss}, in this case the unique minimizer of $\KL{\cdot}{\pi}$ over Gaussians. In the nonconvex case,~\cite{diao2023forwardbackward} obtains a stationarity point guarantee for a ``forward-backward" algorithm, involving a gradient step and a proximal step. The update step in~\cite{gauss-VI} involves inverting a $d\times d$ matrix, while that of~\cite{diao2023forwardbackward} involves computing the square root of a $d\times d$ matrix. The update step of both algorithms also involve computing a Gaussian expectation or else drawing a single Gaussian sample, in which case the algorithm is stochastic. 

A cheaper algorithm whose complexity scales linearly with dimension is \emph{mean-field} Gaussian VI, in which the objective~\eqref{VI-Gauss} is minimized only over Gaussians with diagonal covariance matrices. See~\cite{challis13GKL} for a discussion of the computational complexity of a whole range of Gaussian VI variants, from diagonal to full covariance matrix. However, the accuracy of Gaussian mean-field VI is not currently known.

% This work studies a well-known ODE in Bayesian optimal filtering and smoothing, which evolves forward a mean and covariance matrix $(m_t, \Sigma_t)$ in order to iteratively estimate the mean and covariance of a measure $\pi\propto e^{-V}$. The authors show that this dynamics can be cast as a Wasserstein gradient flow on $\KL{\cdot}{\pi}$ constrained to the submanifold of normal distributions. As a result, if $V$ is strongly convex, then the gradient flow converges exponentially fast to $\hat \pi$ of~\eqref{VI-Gauss}, in this case the unique minimizer of $\KL{\cdot}{\pi}$ over Gaussians. 

\section{Logistic regression with Gaussian design}\label{sec:logreg}
In logistic regression, we generate $n$ feature vectors $X_i\in\R^d$ and observe their corresponding labels $Y_i\in\{0,1\}$. The distribution of the labels given the features is modeled as %are assumed to generated randomly from the feature vectors and a parameter $\theta\in\R^d$ via
\beq\label{bernp}p(Y_i\mid X_i, \theta) = \exp\l(Y_iX_i^T\theta-\psi(X_i^T\theta)\r),\eeq where $\psi(t)=\log(1+e^t)$ and $\theta$ is the unknown coefficient vector. Note that~\eqref{bernp} is just the probability distribution for $\mathrm{Bern}(\sigma(X_i^T\theta))$, where $\sigma(t)=\psi'(t)=(1+e^{-t})^{-1}$ is the sigmoid.  We assume the model is well-specified, so that there is a true $\theta_0\in\R^d$ such that 
$$Y_i\sim\mathrm{Bern}(\sigma(X_i^T\theta_0)).$$ We assume the features are generated from a standard Gaussian distribution: 
$$X_i\stackrel{\mathrm{i.i.d.}}{\sim}\mathcal N(0, I_d),\quad i=1,\dots,n.$$ We consider a Gaussian prior $\theta\sim\mathcal N(0, \Sigma)$ on the coefficient vector $\theta$. The posterior distribution of $\theta$ given the labels $Y_i$ is then $\pi(\theta)\propto e^{-nv(\theta)}$, where
\beqs\label{LL}
v(\theta) &= \ell(\theta)+ \frac1{2n}\theta^T\Sigma^{-1}\theta,\\
\ell(\theta)&:=-\frac1n\sum_{i=1}^n\log p(Y_i\mid X_i, \theta) = -\theta^T\l[\frac1n\sum_{i=1}^nY_iX_i\r]+\frac1n\sum_{i=1}^n\psi(X_i^T\theta) 
%-\frac1n\log\pi_0(\theta) \\&= -\theta^T\l[\frac1n\sum_{i=1}^ny_ix_i\r]+\frac1n\sum_{i=1}^n\psi(x_i^T\theta) + \frac1{2n}\theta^T\Sigma^{-1}\theta.
\eeqs  Here, $\ell$ is the negative normalized log likelihood, and $\theta^T\Sigma^{-1}\theta/2n$ is the contribution from the log prior. Formally, $\Sigma^{-1}=0$ corresponds to the case of a flat prior. Note that the distribution of the $X_i$ does not enter into the posterior. 

\subsection{Checking the assumptions}\label{sec:assump:lr} To verify Assumption~\ref{assume:1}, we need to prove there is a unique global minimizer $\theta^*$ of $v$. %Moreover, to check Assumption~\ref{assume:glob}, we will also need to bound $\nabla^2v(\theta^*)$ from below. To do this, we will need to know that $\|\theta^*\|$ is 
To do so, we first show there is a unique global minimizer $\theta^*_\ell$ of $\ell$. It is straightforward to check that $\ell$ is strictly convex provided the features $X_i$ span $\R^d$, and this occurs with probability 1. Therefore, if $\nabla\ell(\theta^*_\ell)=0$ for some $\theta^*_\ell$, then this point must be the unique global minimizer of $\ell$. %Similarly, we then show there is a point $\theta^*$ such that $\nabla v(\theta^*)=0$, to conclude $\theta^*$ is the global minimizer of $v$. 
To show $\ell$ has a critical point, we will use the following lemma with $f=\ell$ and $\theta=\theta_0$.
\begin{lemma}[Corollary F.6 in~\cite{katsBVM} with $A=I_d$] \label{IVT-LL}
Let $f\in C^2(\R^d)$. Suppose there is a point $\theta$ such that for some $\lambda,s>0$, the following holds:%and suppose there is define the event $E_{\lambda,s}$ as the intersection of the following three events:
\beqs\label{freq}
&\nabla^2f(\theta)\succeq\lambda I_d,\\
&\|\nabla f(\theta)\|\leq 2\lambda s,\\
&\|\nabla^2f(\theta') - \nabla^2f(\theta)\|\leq \frac\lambda4\quad\forall \|\theta'-\theta\|\leq s.
\eeqs Then there exists $\theta_{f}^*$ such that $\nabla f(\theta_{f}^*)=0$ and $\|\theta_{f}^*-\theta\|\leq s$.
\end{lemma}
Of course, $\ell$ is a random function, so~\eqref{freq} cannot hold deterministically for $f=\ell$. Instead, we bound from below the probability that~\eqref{freq} holds.
\begin{lemma}\label{lma:grad}Suppose $d/n\leq 1/10$ and $n\geq13$. Then the event 
$$E_1=\{\|\nabla\ell(\theta_0)\|\leq 8(d/\sqrt n)^{1/2}\}$$ has probability at least $1-e^{-n/4}-n^{-d/4}$ under the ground truth distribution: $X_i\sim\mathcal N(0, I_d), Y_i\mid X_i\sim\mathrm{Bern}((\sigma(X_i^T\theta_0))$, and $(X_i,Y_i)$, $i=1,\dots,n$ are i.i.d. 
\end{lemma}
The proof of the lemma is similar to that of Lemma 4.2 in~\cite{katsBVM}, but we include the proof in Appendix~\ref{app:log} for the sake of completeness.
\begin{lemma}[Adapted from Lemma 7, Chapter 3,~\cite{pragyathesis}]\label{lma:hess}
Suppose $d<n$. Then the event
\beq\label{E2pdef}
E_2=\l\{\inf_{\|\theta\|\leq c}\lambda_{\min}\l(\nabla^2\ell(\theta)\r)\geq \tau(c)\quad\forall c\geq0\r\}\eeq has probability at least $1-4e^{-Cn}$ for an absolute constant $C$. Here, $\tau:[0,\infty)\to(0,\infty)$ is a universal non-increasing function. \end{lemma}
We note that the function $\ell$ actually considered in Lemma 7 in~\cite{pragyathesis} relates to our $\ell$ by a $\sqrt n$ rescaling. We have translated the statement of the result in terms of our function $\ell$.
\begin{lemma}\label{lma:34deriv}
Suppose $d\leq n$. Then there exist absolute constants $C_3,C_4, C$ such that the event
\beq
E_3=\l\{\sup_{\theta\in\R^d}\|\nabla^k\ell(\theta)\|\leq C_k, \quad k=3,4\r\}\eeq has probability at least $1-\e(-C(nd)^{1/9})$.
\end{lemma}See Appendix~\ref{app:log} for the proof, which follows almost immediately by an application of Theorem 4.2 in~\cite{adamczak2010quantitative}. We can now apply Lemma~\ref{IVT-LL} to prove $\ell$ has a unique global minimizer $\theta^*_\ell$. This then also implies $v$ has a unique global minimizer nearby, by another application of Lemma~\ref{IVT-LL}. Lemmas~\ref{lma:grad},~\ref{lma:hess},~\ref{lma:34deriv} also give us the necessary ingredients to check the other assumptions.
\begin{corollary}\label{corr:log}
Suppose $\|\theta_0\|\leq C$, $\|\Sigma^{-1}\|\leq C$ and that $d/\sqrt n$ and $n^{-1}$ are smaller than a sufficiently small absolute constant. Then with probability at least $1-\e(-C(nd)^{1/9})-n^{-d/4}-5e^{-Cn}$, the following statements all hold: 
\begin{enumerate}
\itemsep=5pt 
\item The functions $\ell$ and $v$ have unique global minimizers $\theta^*_\ell$ and $\theta^*$, respectively, with $\|\theta^*_\ell-\theta_0\|\les\sqrt{d/n}$ and $\|\theta^* - \theta^*_\ell\|\les n^{-1}$. In particular, Assumption~\ref{assume:1} is satisfied.
\item We have $\|\theta^*\|\leq C$ and $\lambda_{\min}(H_v)=\lambda_{\min}(\nabla^2v(\theta^*))\geq \tau(C)$.
\item We have $\sup_{\theta\in\R^d}\|\nabla^kv(\theta)\|_{H_v}\leq C$ for $k=3,4$. In particular,~\eqref{globnabla} of Assumption~\ref{assume:glob} is satisfied with $q=0$ and $a_3=a_4=C$. Furthermore,~\eqref{cconds} of Assumption~\ref{assume:glob} is also satisfied provided $d/\sqrt n$ is smaller than some fixed deterministic absolute constant.
\item Assumption~\ref{assume:c0} is satisfied with $c_0=1/8$.
\end{enumerate}
\end{corollary}
See Appendix~\ref{app:log} for the proof.

\subsection{Application of results in Section~\ref{sec:main} to logistic regression}
We first compute the function $Q$ in the leading order term. Now, one can check that
$$\nabla^3v(\theta)=\frac1n\sum_{i=1}^n\psi'''(\theta^TX_i)X_i^{\otimes 3},$$ and recall that $V=nv$. Therefore, 
$$\E_{\theta\sim\hat\pi}[\nabla^3V(\theta)] = \sum_{i=1}^nb_i(\theta)X_i^{\otimes3},$$ where \beq\label{bidef} b_i(\theta) =\E_{\theta\sim\hat\pi}\l[\psi'''(\theta^TX_i)\r].\eeq Using the definition~\eqref{Qdef} of $Q$ and the above expression for $\E_{\theta\sim\hat\pi}[\nabla^3V(\theta)]$, we get
\beqs\label{Qlog}
Q(\theta)%&=\lla \sum_{i=1}^nb_i(\theta)X_i^{\otimes3}, \frac12(\theta-\mvi\pi)\otimes\Svi\pi- \frac16(\theta-\mvi\pi)^{\otimes3}\rra\\
= \sum_{i=1}^nb_i(\theta)\l[\frac12X_i^T(\theta-\mvi\pi)X_i^T\Svi\pi X_i -\frac16(X_i^T(\theta-\mvi\pi))^3\r]
\eeqs
We now summarize the results from Section~\ref{sec:main} applied to logistic regression. We assume the set-up from the beginning of the section: well-specified model, i.i.d. Gaussian design, bounded ground truth $\|\theta_0\|\leq C$, and Gaussian or flat prior $\mathcal N(0,\Sigma)$, with $\|\Sigma^{-1}\|\leq C$. In this setting, Corollary~\ref{corr:log} shows that Assumptions~\ref{assume:1},~\ref{assume:glob},~\ref{assume:c0} are satisfied, and that $q=0,a_3=a_4=C$.
\begin{theorem}\label{thm:log}Suppose $d/\sqrt n$ and $n^{-1}$ are smaller than certain absolute constants. On an event of probability at least $1-\e(-C(nd)^{1/9})-5e^{-Cn}-n^{-d/4}$, the following results hold.

There exists a unique solution $\mvi\pi,\Svi\pi$ to~\eqref{kl-opt-intro} in the region $\mathcal R_V$. Furthermore, if $g$ satisfies~\eqref{gcond} with $c_0=1/8$, then
\beqs
 \l|\medint\int gd\pi - \medint\int gd\hat\pi\r|&\les_{R_g} (1+\Var_{\hat\pi}(g)^{\frac12})\frac{d}{\sqrt n},\\
\l|\medint\int gd\pi - \medint\int gd\hat\pi - \int gQd\hat\pi\r|&\les_{R_g}  (1+\Var_{\hat\pi}(g)^{\frac12})\l(\frac{d}{\sqrt n}\r)^2,
\eeqs where $Q$ is defined in~\eqref{Qlog} and~\eqref{bidef}. If $g$ is even about $\mvi\pi$, then
\beq
 \l|\medint\int gd\pi - \medint\int gd\hat\pi\r|\les_{R_g} (1+\Var_{\hat\pi}(g)^{\frac12})\l(\frac{d}{\sqrt n}\r)^2.
\eeq If $g$ is linear, then
\beq\l|\medint\int gd\pi - \medint\int gd\hat\pi\r|\les_{R_g}(1+\Var_{\hat\pi}(g)^{\frac12})\l(\frac{d}{\sqrt n}\r)^3.\eeq
Next, we have
\beqs\label{TVlog}
\sup_{A\in \mathcal B(\R^d)}\l|\pi(A)-\hat\pi(A)\r|&\lesssim \frac{d}{\sqrt n},\\
\sup_{A\in \mathcal B_{s,\mvi\pi}(\R^d)}\l|\pi(A)-\hat\pi(A)\r|&\lesssim \l(\frac{d}{\sqrt n}\r)^2.
\eeqs Finally, let $\bar\theta = \int \theta d\pi(\theta)$ be the mean of $\pi$ and $\Sigpi = \int (\theta-\bar\theta)(\theta-\bar\theta)^Td\pi(\theta)$ be the covariance of $\pi$. Then
\beqs\label{meancovlog}
\sqrt n\|\mpi - \mvi\pi\| &\les \l(\frac{d}{\sqrt n}\r)^3,\\
n\|\Sigpi-\Svi\pi\|&\les \l(\frac{d}{\sqrt n}\r)^2
\eeqs In~\eqref{TVlog} and~\eqref{meancovlog}, all suppressed constants are absolute.
\end{theorem} The proof follows immediately from Corollary~\ref{corr:log} and the results in Section~\ref{sec:main}.

\subsection{Numerical Simulation}We confirm some of the theoretical results ---specifically, the bounds~\eqref{meancovlog} --- in a numerical simulation. The numerical results are displayed in Figure~\ref{fig:demo}. We take $d=2$ and $n=100,200,\dots, 1000$, and a flat prior on $\theta$. For each $n$, we draw ten sets of covariates $x_i, i=1,\dots, n$ from $\mathcal N(0, I_d)$, yielding ten posterior distributions $\pi_n(\cdot\mid x_{1:n})$. We then compute the Laplace and VI mean and covariance approximation errors for each $n$ and each of the ten posteriors at a given $n$. The solid lines in Figure~\ref{fig:demo} depict the average approximation errors over the ten distributions at each $n$. The shaded regions depict the spread of the middle six out of ten approximation errors. See Appendix~\ref{app:log} for details about the simulation.

  \begin{figure}[t]
 \centering
 \includegraphics[width=0.49\textwidth]{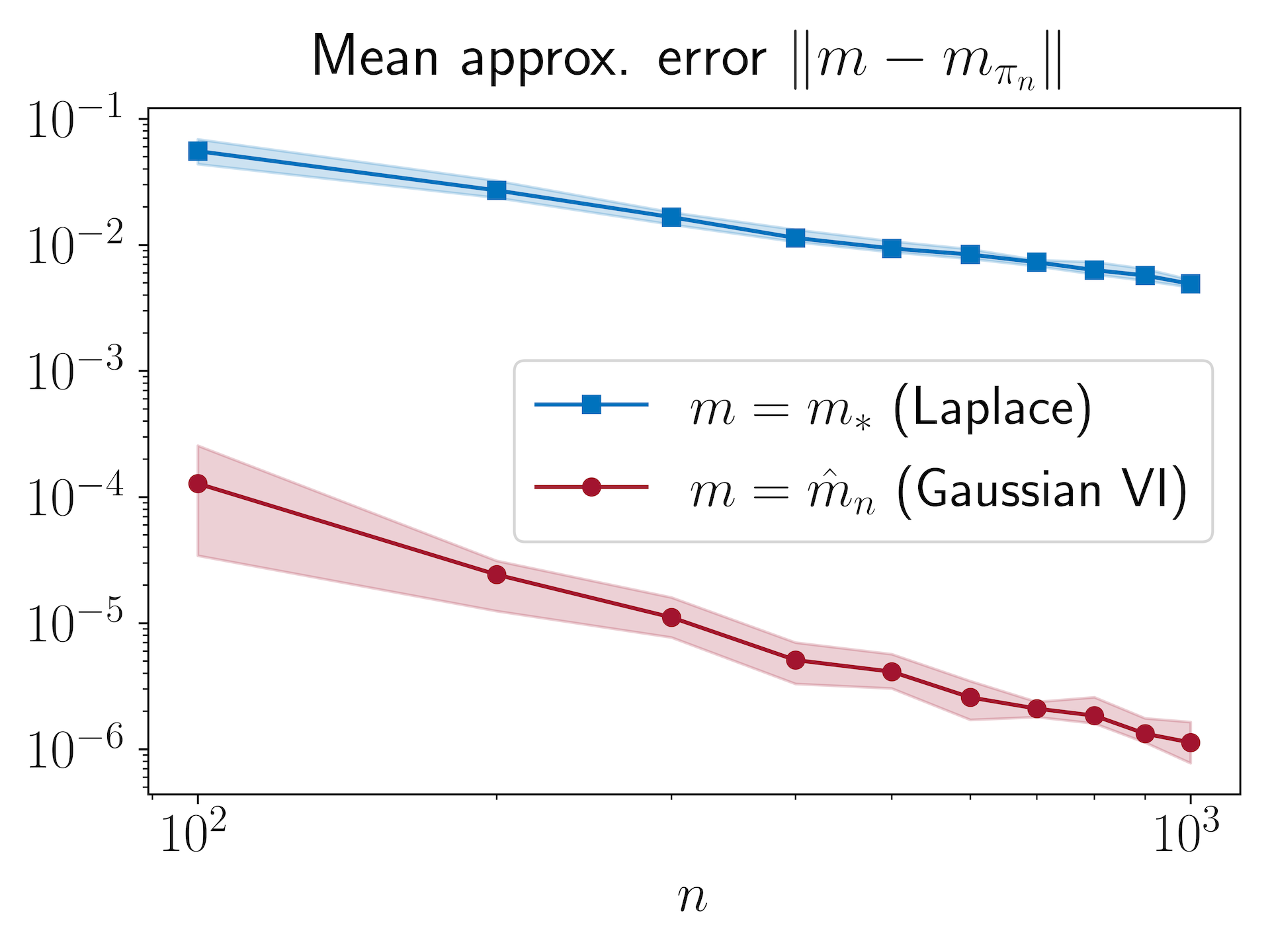}
 \includegraphics[width=0.49\textwidth]{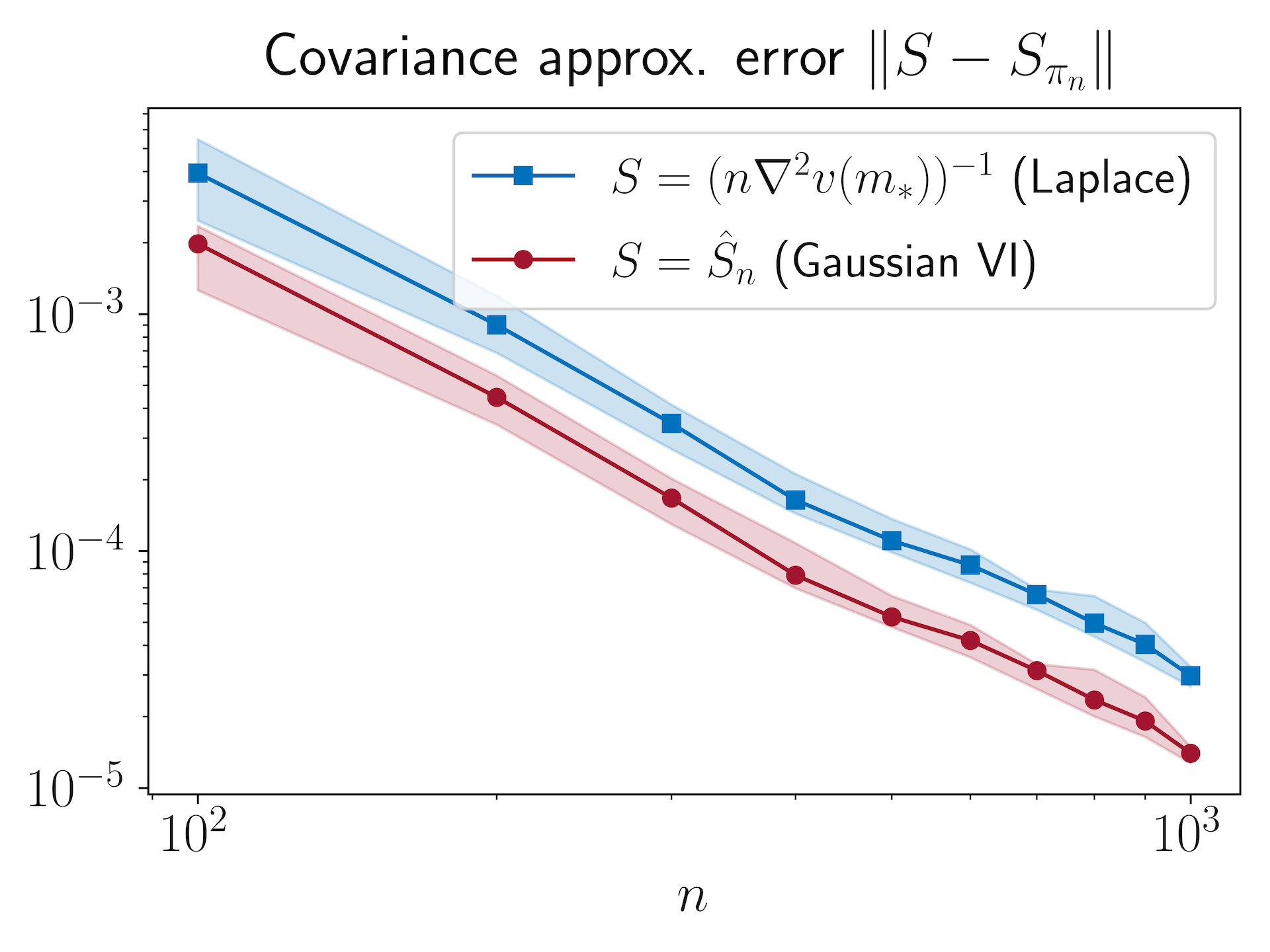}
 \vspace{-0.3cm}
 \caption{Gaussian VI yields a more accurate mean estimate than does Laplace, while the two covariance estimates are on the same order. Here, $\pi_n$ is the likelihood of logistic regression given $n$ observations in dimension $d=2$. For the left-hand plot, the slopes of the best-fit lines are $-1.04$ for the Laplace approximation and $-2.02$ for Gaussian VI. For covariance: the slopes of the best-fit lines are -2.09 for Laplace, -2.12 for VI.}
 \label{fig:demo}
 \end{figure}
 
In the left panel of Figure~\ref{fig:demo} depicting the mean error, the slopes of the best fit lines are $-1.04$ and $-2.02$ for Laplace and Gaussian VI, respectively. For the covariance error in the righthand panel, the slopes of the best fit lines are $-2.09$ and $-2.12$ for Laplace and Gaussian VI. This confirms that our bounds on VI, as well as the Laplace bounds, are tight in their $n$ dependence.

\section{Proof of Theorems~\ref{thm:Vgen} and~\ref{thm:corr}}\label{sec:proof} In this section we explain the main steps in the proof of Theorems~\ref{thm:Vgen} and~\ref{thm:corr}. We take Lemma~\ref{lma:intro:exist:V} as given; its proof is outlined in Section~\ref{sec:m-sig-exist}.
\subsection{Reduction to comparison with a standard Gaussian}Recall that $\hat\pi=\mathcal N(\mvi\pi,\Svi\pi)$, where $\mvi\pi$, $\Svi\pi$ are the solutions to~\eqref{kl-opt-intro} furnished by Lemma~\ref{lma:intro:exist:V}. We are interested in the difference $\int gd\pi-\int gd\hat\pi$, but we can always change variables via a bijection $T$, writing this difference in the form
\beq\label{gfequiv}\int gd\pi-\int gd\hat\pi = \int fdT_{\#}\pi - \int fdT_{\#}\hat\pi,\qquad f=g\circ T^{-1}.\eeq For convenience, we choose $T$ so that $T_{\#}\hat\pi=\gamma$ is the standard Gaussian distribution. Specifically, we let
\beq\label{Tdef}T(x)=\Svi\pi^{-1/2}(x-\mvi\pi),\eeq and we define $\rho=T_{\#}\pi\propto e^{-\Wbar}$, where
\beq\label{rhodef}\Wbar(x) :=  V(\mvi\pi + \Svi\pi^{1/2}x)=V(T^{-1}(x)).\eeq
Next, we reformulate Theorems~\ref{thm:Vgen} and~\ref{thm:corr}. To do so, define $\A_3 = \E[\nabla^3\Wbar(Z)]$ and
$$p_3(x)=\frac16\lla \A_3, \H_3(x)\rra = \lla \A_3, \;\frac16x^{\otimes3}-\frac12x\otimes I_d\rra.$$ Here, $\H_3(x)$ is the $d\times d\times d$ tensor of third order multivariate Hermite polynomials; see Section~\ref{subsec:outline} and Appendix~\ref{hermite-tail} for more details. 
\begin{theorem}\label{thm:W}
Let $f=g\circ T^{-1}$, where $T$ is as in~\eqref{Tdef}, and let $g$ satisfy~\eqref{gcond}. Also, let $\rho\propto e^{-\Wbar}$ as in~\eqref{rhodef}, and $\Delta_f = \int fd\rho - \int fd\gamma.$ Then under the conditions of Theorem~\ref{thm:Vgen}, it holds
\begin{align}
\l|\Delta_f\r|&\les (1+\|\fbar\|_2)\l(a_3\epsilon +  a_4\epsilon^2\r) ,\label{overallbd1}\\
\l|\Delta_f - \int f(-p_3)d\gamma\r|&\les (1+\|\fbar\|_2)(a_3^2+  a_4)\epsilon^2.\label{overallbd2}
\end{align} If $f$ is orthogonal to all third order Hermite polynomials (in particular, if $f$ is even), then
\beq\label{overallbd3}\l|\Delta_f\r|\les (1+\|\fbar\|_2)(a_3^2+  a_4)\epsilon^2.\eeq If $f$ is linear, then
\beq\label{overallbd4}\l|\Delta_f\r|\les (1+\|\fbar\|_2)\l(\l(a_3^3+ a_3a_4\r)\epsilon^3 +  a_4^2\epsilon^4\r).\eeq In all of the above bounds, the suppressed constant is an increasing function of $q, c_0^{-1}, R_g$. 
  \end{theorem}
\noindent The bounds~\eqref{overallbd1},~\eqref{overallbd3}, and~\eqref{overallbd4} correspond to ~\eqref{overallbd1V},~\eqref{overallbd3V}, and~\eqref{overallbd4V} of Theorem~\ref{thm:Vgen}, respectively. This immediately follows from~\eqref{gfequiv} and the fact that $\Var_{\hat\pi}(g)=\Var_\gamma(f)=\|\fbar\|_2^2$. The bound~\eqref{overallbd2} corresponds to~\eqref{overallbd2V} of Theorem~\ref{thm:corr}. This is because $\int f(-p_3)d\gamma = \int gQd\hat\pi$, as a short calculation in Appendix~\ref{supp:aux} shows.
\subsection{Proof Outline}\label{subsec:outline}We now outline the proof of the equivalent Theorem~\ref{thm:W}. The proof is based on several key observations. First, note that the optimality equations~\eqref{kl-opt-intro} and the definition~\eqref{rhodef} of $\barW$ imply 
\beq\label{V0-c1c2}
\E[\nabla \Wbar(Z)] = 0,\qquad \E[\nabla^2\Wbar(Z)] = I_d.
\eeq
The proof exploits~\eqref{V0-c1c2} through six further observations.\\

\noindent 1) The optimality conditions~\eqref{V0-c1c2} imply that the Hermite series expansion of $\Wbar $ is of the form $\Wbar (x) = \mathrm{const.}+\frac12\|x\|^2 + r_3(x)$, where \beq\label{r3-def}r_3(x)=\sum_{k\geq3}\frac{1}{k!}\la \A_k, \H_k(x)\ra.\eeq In particular, $d\rho\propto e^{-\Wbar}\propto e^{-r_3}d\gamma$.\\

%\item 
\noindent 2) Since $r_3$ contains only third and higher order Hermite polynomials, $r_3$ is \emph{orthogonal to all polynomials of degree 2 or lower}, with respect to the Gaussian measure. In other words, $\int fr_3d\gamma$ for all quadratic $f$. As a special case (taking $f=1$), we see that $\int r_3d\gamma=0$.\\

\noindent 3) Let $p_3 = \frac{1}{3!}\la\A_3,\H_3\ra$ be the third order Hermite polynomial contribution to $\barW$, and write $r_3=p_3+r_4$, where $r_4$ is the Hermite series remainder of order 4. The change of variables via the transformation $T$ ensures that $r_3\sim p_3\sim n^{-1/2}$ and $r_4\sim n^{-1}$.\\
%\item 

\noindent 4) Let $\fbar = f-\gamma(f)$, so that $\Delta_f=\int \fbar d\rho$. Using that $d\rho\propto e^{-r_3}d\gamma$, that $r_3\sim n^{-1/2}$, and that $\int \fbar d\gamma=\int r_3d\gamma=0$, we have 
\beqs\label{EfbarX-explained}
\Delta_f=\int \fbar d\rho= \frac{\int\fbar e^{-r_3}d\gamma}{\int e^{-r_3}d\gamma} &= \frac{\int\fbar\l(1-r_3+r_3^2/2+\mathcal O(n^{-3/2})\r)d\gamma}{\int\l(1-r_3+\mathcal O(n^{-1})\r)d\gamma}\\
&= \frac{\int\fbar\l(-r_3+r_3^2/2+\mathcal O(n^{-3/2})\r)d\gamma}{1+\mathcal O(n^{-1})}
\eeqs
%\item 
5) For a generic $f$, the order of $\Delta_f$ is that of the first remaining term in the numerator on the second line: $\Delta_f\sim \int\fbar r_3d\gamma\sim n^{-1/2}$. But if $f$ is linear or quadratic, then $\int\fbar r_3d\gamma =0$, since $r_3$ is orthogonal to quadratic polynomials. Therefore, $\Delta_f$ is at most of order $\int \fbar r_3^2d\gamma\sim n^{-1}$. \\

\noindent 6) By writing $r_3 = p_3 + r_4=\mathcal O(n^{-1/2})+\mathcal O(n^{-1})$ and therefore $r_3^2=p_3^2+\mathcal O(n^{-3/2})$, we can get even more refined estimates. Namely, we have
$$\Delta_f\sim \int \fbar\l(-r_3+\frac12r_3^2\r)d\gamma = \underbrace{\int \fbar(-p_3)d\gamma}_\text{$\mathcal O(n^{-1/2})$} + \underbrace{\int \fbar \l(-r_4+\frac12p_3^2\r)d\gamma}_\text{$\mathcal O(n^{-1})$} + \mathcal O(n^{-3/2}).$$ From this decomposition, we see that if $f$ is orthogonal to third order Hermite polynomials then $\int \fbar p_3d\gamma=0$, and hence $\Delta_f\sim\mathcal O(n^{-1})$ in this case. If $f(x)=a^Tx+b$ then $0=\int \fbar p_3d\gamma=\int \fbar r_4d\gamma$ automatically, but we also have $\int\fbar p_3^2d\gamma=0$ since $\fbar(x)=a^Tx$ is odd and $p_3^2$ is even. Hence both the $\mathcal O(n^{-1/2})$ term and the $\mathcal O(n^{-1})$ term vanish, so that $\Delta_f\sim n^{-3/2}$.\\

This is the essence of the proof. Now that we have given this overview, let us go into a few more details about the above points. To explain 1), recall that the Hermite series expansion of $\Wbar  $ is defined as
\beq
\Wbar  (x) = \sum_{k=0}^\infty\frac{1}{k!}\la \A_k(V_0), \H_k(x)\ra,\qquad \A_k(\Wbar):= \E\l[\Wbar(Z)\H_k(Z)\r]
\eeq
Here, the $\A_k(\Wbar)$ and $\H_k(x)$ are tensors in $(\R^d)^{\otimes k}$. Specifically, $\H_k(x)$ is the tensor of all order $k$ Hermite polynomials, enumerated as $H_k^{(\alpha)}$, $\alpha\in[d]^k$ with some entries repeating. For $k=0,1,2$, the Hermite tensors are given by
\beq\label{H012}\H_0(x) = 1, \quad \H_1(x)  = x, \quad \H_2(x) = xx^T - I_d.\eeq See Appendix~\ref{hermite-tail} for further details on Hermite series and polynomials. Distinct Hermite polynomials are orthogonal to each other with respect to the Gaussian weight. %In particular, if $f$ is an order $k$ polynomial and $\ell>k$ then $$\E[f(Z)H_\ell^{(\alpha)}(Z)]=0,\qquad\forall\alpha\in[d]^\ell.$$
Using the representation
\beq\label{H-via-e}\H_k(x)e^{-\|x\|^2/2} = (-1)^k\nabla^ke^{-\|x\|^2/2},\eeq we obtain a useful, ``Gaussian integration by parts" identity for a $k$-times differentiable function $f$. Namely,
\beq\label{GIP}%\int f\H_kd\gamma = \int\nabla^kfd\gamma.
\E[f(Z)\H_k(Z)] = \E[\nabla^kf(Z)].
\eeq This is a generalization of Stein's identity, %$\int x_if(x)d\gamma(x) = \int\partial_{x_i}f(x)d\gamma(x)$. 
$\E[Z_if(Z)]=\E[\partial_if(Z)]$. Since $\Wbar$ is at four times differentiable, we can use Gaussian integration by parts to write $\A_1,\A_2$ as 
\beqs
\A_1(\Wbar) &:= \E[\Wbar(Z)\H_1(Z)]= \E[\nabla\Wbar  (Z)]=0,\\ \A_2(\Wbar) &:= \E[\Wbar  (Z)\H_2(Z)] = \E[\nabla^2\Wbar  (Z)]=I_d,\eeqs where the last equality in each line comes from the optimality conditions~\eqref{V0-c1c2}. Therefore the Hermite series expansion of $\Wbar$ takes the form
\beqs\label{V0-expand}
\Wbar  (x)& =\la A_ 0(\Wbar), H_0\ra+ \la 0, \H_1(x)\ra + \frac12\la I_d, \H_2(x)\ra+ r_3(x)\\
&=\mathrm{const.}+\frac{\|x\|^2}{2}+ r_3(x),
\eeqs 
where $r_3$ is the third order remainder~\eqref{r3-def}. Here, we recall from~\eqref{H012} that $\H_2(x)=xx^T-I_d$. This explains point 1), and point 2) is a direct application of the orthogonality of distinct Hermite polynomials. 

For a heuristic proof of point 3), recall that $\Wbar(x)= V(\mvi\pi + \Svi\pi^{1/2}x)$, that $V=nv$, and that $\Svi\pi^{1/2}\sim n^{-1/2}$. Therefore by the chain rule, $\nabla^3\Wbar\sim n^{-1/2}$, and $\nabla^4\Wbar\sim n^{-1}$. Hence $$p_3=\la\A_3,\H_3\ra=\la\E[\nabla^3\Wbar(Z)],\H_3\ra\sim n^{-1/2}$$ and for $k\geq4$, we can apply ``partial" Gaussian integration by parts to express $\A_k$ as $$\A_k =\E[\H_k(Z)\Wbar  (Z)]=\E[\H_{k-4}(Z)\otimes \nabla^4\Wbar  (Z)]\sim n^{-1}.$$ Therefore $r_4\sim n^{-1}$. 

\begin{remark}\label{rk:A3fV}Recall that in the statement of Theorem~\ref{thm:W}, we defined $\A_3$ as $\A_3=\E[\nabla^3\Wbar(Z)]$. We now see that $\A_3=\A_3(\Wbar)$ is precisely the third order coefficient tensor in the Hermite series expansion of $\Wbar$. Furthermore, we can now explain the formulation of the leading order term given in Section~\ref{discuss}. Letting $\A_3(f)$ be the third order coefficient tensor in the Hermite series expansion of $f$, we can write
$$\int fp_3d\gamma = \frac{1}{3!}\lla \A_3(\Wbar), \int f\H_3 d\gamma \rra = \frac{1}{3!}\lla \A_3(\Wbar), \A_3(f)\rra.$$
\end{remark}
\subsubsection{Dimension dependence}
So far, we have only discussed the $n$ dependence of our bounds. We now also briefly discuss the $d$ dependence. As point 6) shows, in the typical case the leading order term is $\Delta_f\sim\int \fbar p_3d\gamma$, which by Cauchy-Schwarz is bounded as $\|\fbar\|_2\|p_3\|_2$. In~\cite{katsTDD}, we showed that for a symmetric $d\times d\times d$ tensor $T$, it holds
\beq\E\l[\la T, \H_3(Z)\ra^{p}\r]\leq C_k\E\l[\la T,\H_3(Z)\ra^2\r]^{p/2} \leq C_p\l(d\|T\|\r)^{p}\eeq for even $p$. Therefore, if $\|T\|\sim 1/\sqrt n$, then we get the desired dependence on $\epsilon=d/\sqrt n$. 

Another difficulty that arises is in bounding $\|r_4\|_p$. Intuitively, $r_4(x)$ is dominated by the fourth order polynomial $\la\A_4,\H_4(x)\ra$, so we expect that \beq\label{r4heur}\E[|r_4(Z)|^p]\sim \E\l[|\la\A_4,\H_4(Z)\ra|^p\r] \sim (d^2/n)^p,\eeq since $\A_4\sim n^{-1}$ and $\E[\|Z\|^{4p}]\sim d^{2p}$. However, with only a series representation of $\|r_4\|_p$, it is difficult to translate this intuition into a formal bound. We therefore derive the following explicit formula for the Hermite series remainder:
\begin{prop}\label{prop:intro:hermite-tail-ddim}
Assume $\Wbar  \in C^K$. Let $\Wbar  (x) = \sum_{j=0}^\infty\frac{1}{j!}\la \A_ j, \H_j(x)\ra$ be the Hermite series expansion of $\Wbar  $, and define 
\beq\label{rk-def}r_{k}(x) = \Wbar  (x) - \sum_{j=0}^{k-1}\frac{1}{j!}\la \A_ j, \H_j(x)\ra.\eeq Then for all $k\leq K$, it holds
\beqs\label{intro:rk-exact}
r_k(x) = \int_0^1\frac{(1-t)^{k-1}}{(k-1)!}\E\bigg[\bigg\la\nabla^{k}\Wbar  \l((1-t)Z+tx\r), \;\H_k(x)-Z\otimes\H_{k-1}(x)\bigg\ra\bigg]dt.\eeqs 
\end{prop} Note that~\eqref{intro:rk-exact} is analogous to the integral form of the remainder of a Taylor series. See Lemma~\ref{lma:rk-exact} in Appendix~\ref{app:tensor-hermite} for the proof of this proposition. The formula~\eqref{intro:rk-exact} is known in one dimension; see Section 4.15 in~\cite{lebedev}. However, we could not find the multidimensional version in the literature, so we have proved it here. Using this formula, we show that the heuristic argument~\eqref{r4heur} is essentially correct in that it gives the right scaling with $d$ and $n$.

%\begin{lemma}
%If $g$ satisfies~\eqref{gcond} and $f=g\circ T$, then $f$ satisfies
%\beq\label{fcond}
%|f(x)-\gamma(f)|\leq \e\l(\frac{c_0}{4}\sqrt d\|x\|\r)\quad\forall \|x\|\geq R_f\sqrt d,\qquad R_f=2R_g+4.
%\eeq
%\end{lemma}
%This follows from Lemma~\ref{} in Supplement~\ref{} (in which it is shown~\eqref{fcond} holds with a slightly tighter lower bound on $\|x\|$). 
%\begin{lemma}
%If $g$ satisfies~\eqref{gcond} and $f=g\circ T$, then $f$ satisfies
%\beq\label{fcond}
%|f(x)-\gamma(f)|\leq \e\l(\frac{c_0}{4}\sqrt d\|x\|\r)\quad\forall \|x\|\geq R_f\sqrt d,\qquad R_f=R_g.
%\eeq
%\end{lemma}The proof is immediate. 
\subsection{Key Lemmas}\label{subsec:key}Now that we have outlined the steps of the proof informally, we turn to the rigorous proof. We formulate a sequence of lemmas from which Theorem~\ref{thm:W} will follow. The omitted proofs can be found in Appendix~\ref{app:sec:proof}.

First, we define a few quantities. 
%\end{defn}
%Note that $d\rho\propto e^{-\rtri}d\gamma$, and
%\beqs\label{rtrix2}
%\rtri(x)+ \frac{\|x\|^2}{2} &= \Wbar  (x)-\E[\Wbar  (Z)]+\frac d2= \Wbar  (x)-\Wbar  (0) -\E\l[\Wbar  (Z)-\Wbar  (0)-\frac d2\r] \\
%&=\Wbar  (x)-\Wbar  (0)-\delta,\eeqs where $\delta = \E\l[\Wbar  (Z)-\Wbar  (0)-\frac d2\r]$.
\begin{defn}\label{LEdef}Let $\rtri$ be as defined in~\eqref{rk-def}, and let $E:=\int  \l(e^{-\rtri}-1+\rtri\r)d\gamma.$ Also, for a function $f\in L^2(\gamma)$, let
\beqsn
 L(f) &= \int \fbar  \l(-\rtri+\frac{\rtri^2}{2}\r)d\gamma,\quad E(f) =\int \fbar  \l(e^{-\rtri}-1+\rtri-\frac{\rtri^2}{2}\r)d\gamma.
\eeqsn 
\end{defn} In the next lemma, we write $\int fd\rho-\int fd\gamma$ suggestively as $L(f)$ plus a remainder term.% Since $\rtri\sim n^{-1/2}$, we see heuristically that $L(f)\sim n^{-1/2}$, $E\sim n^{-1}$, and $E(f)\sim n^{-3/2}$. The next lemma then suggests that $\int fd\rho - \int fd\gamma=L(f) + \mathcal O(n^{-3/2})$. 
\begin{lemma}\label{lma:LmEm}It holds 
\beqs\label{LmEm}
\l|\int fd\rho -\int fd\gamma - L(f)\r| \leq |E|\l|L(f)\r|+ |E(f)|.
\eeqs \end{lemma}
\begin{proof}
Recall from the proof outline that $\int fd\rho - \int fd\gamma = \int \fbar e^{-r_3}d\gamma/\int e^{-r_3}d\gamma$, and that $\int \fbar d\gamma=\int r_3d\gamma=0$. Therefore, we see that
\beq\label{LmEminter}\int fd\rho - \int fd\gamma = \frac{L(f)+E(f)}{1+E}.\eeq Furthermore, $1+E=\int e^{-r_3}d\gamma\geq1$ by Jensen's inequality, since $\int(-r_3)d\gamma=0$. Subtracting $L(f)$ from both sides of~\eqref{LmEminter} and bounding the resulting righthand side gives~\eqref{LmEm}.
\end{proof}
Consider the definition of $L(f), E(f)$, and $E$. Since $r_3\sim n^{-1/2}$, we expect that $L(f)\sim n^{-1/2}$, $E\sim n^{-1}$, and $E(f)\sim n^{-3/2}$. This is confirmed in the next lemma. %in which we bound $L(f)$, $E$, and $E(f)$.
\begin{lemma}\label{alldabounds}Let $f=g\circ T^{-1}$, as in Theorem~\ref{thm:W}. Then under the conditions of Theorem~\ref{thm:Vgen}, it holds 
\beqs
|L(f)| &\les \|\fbar\|_2\l(\c3\epsilon+ \c4\epsilon^2\r),\\
|E| &\les  \l(\c3\epsilon+ \c4\epsilon^2\r)^2, \\
|E(f)| &\les  (1+\|\fbar \|_2)\l(\c3\epsilon+ \c4\epsilon^2\r)^3,
\eeqs where the suppressed constant is an increasing function of $q, c_0^{-1}$, and $R_g$.\end{lemma}

%$g$ satisfy~\eqref{gcond}, so that $f=g\circ T$ satisfies
%\beq\label{fcond}
%|f(x)-\gamma(f)|\leq \e\l(\frac{c_0}{4}\sqrt d\|x\|\r)\quad\forall \|x\|\geq R_g\sqrt d.
%\eeq
% Then u
Lemmas~\ref{lma:LmEm} and~\ref{alldabounds} immediately give the following corollary, which shows that $\int fd\rho - \int fd\gamma = L(f) + \mathcal O(n^{-3/2})$. 
\begin{corollary}\label{corr:Ebd}In the same setting as Lemma~\ref{alldabounds}, it holds
\beq\label{mainEbd}
\l|\int fd\rho -\int fd\gamma - L(f)\r| \les (1+\|\fbar \|_2)\l(\c3\epsilon + \c4\epsilon^2\r)^3,\eeq where the suppressed constant is an increasing function of $q, c_0^{-1}$, and $R_g$.
\end{corollary}
Thus it remains to study the term $L(f)$. By writing $r_3 = p_3 + r_4$, we break down $L(f)$ into a term of order $\mathcal O(n^{-1/2})$, a term of order $\mathcal O(n^{-1})$, and a remainder of order $\mathcal O(n^{-3/2})$. %s iitself into a leading order contribution and a remainder. This is done in the following lemma.
\begin{lemma}\label{leading-leading}Let $p_3(x)=\frac{1}{3!}\la \A_ 3,\H_3(x)\ra$ and $r_4 = r_3 - p_3$. In the same setting as Lemma~\ref{alldabounds}, it holds
$$
L(f) = \int \fbar(-p_3)d\gamma + \int \fbar \l(-r_4+\frac12p_3^2\r)d\gamma + R, 
$$ where
\begin{align}
 \l|\int \fbar(-p_3)d\gamma\r|&\les_q \|\fbar\|_2a_3\epsilon,\label{Lpart2}\\
  \l|\int \fbar\l(-r_4+\frac12p_3^2\r)d\gamma\r|& \les_q \|\fbar\|_2(a_3^2+ a_4)\epsilon^2,\label{Lpart1}\\
|R|&\les_q \|\fbar \|_2\l(a_3a_4\epsilon^3 +  a_4^2\epsilon^4\r).\label{Ldecomp}
 \end{align}
%\l|L(f)+\int fp_3d\gamma + \int \fbar \l(r_4-\frac12p_3^2\r)d\gamma\r|&\les \|\fbar \|_2\l( \c3\c4\epsilon^3 +  \c4^2\epsilon^4+e^{-\lvl}\Cg3\r)
 \end{lemma} 
 Combining Lemma~\ref{leading-leading} with Corollary~\ref{corr:Ebd}, we can now prove Theorem~\ref{thm:W}.
\begin{proof}[Proof of Theorem~\ref{thm:W}]
To get~\eqref{overallbd1}, we add the bound~\eqref{mainEbd} to the bound on $|L(f)|$, which is given by the sum of~\eqref{Lpart2},~\eqref{Lpart1}, and~\eqref{Ldecomp}. To get~\eqref{overallbd2}, we add up~\eqref{mainEbd}, ~\eqref{Ldecomp} and~\eqref{Lpart1}, omitting the bound~\eqref{Lpart2} which gets incorporated into the lefthand side. Similarly, if $f$ is orthogonal to all third order Hermite polynomials then $\int \fbar p_3d\gamma=0$, so the bound~\eqref{overallbd3} stems from adding up~\eqref{mainEbd}, ~\eqref{Lpart1} and~\eqref{Ldecomp}, again omitting~\eqref{Lpart2}. Finally, if $f(x)=a^Tx+b$ then $\int \fbar p_3d\gamma=\int \fbar r_4d\gamma=0$. Furthermore, $\fbar(x)=a^Tx$, which is odd, so $\int \fbar p_3^2d\gamma=0$. So we get~\eqref{overallbd4} by adding ~\eqref{mainEbd} and~\eqref{Ldecomp}, omitting both~\eqref{Lpart2} and~\eqref{Lpart1}.
\end{proof}
\section{Existence of unique solution to stationarity conditions }\label{sec:m-sig-exist} As in the proof of Theorem~\ref{thm:Vgen}, our first step in proving Lemma~\ref{lma:intro:exist:V} will be to reformulate the lemma in a scale-free way. We will use the notation
\beqs\label{BrS-notation}
B_r(0,0) &= \{(m,\sigma)\in \R^{d}\times\R^{d\times d} \; : \; \|\sigma\|^2+\|m\|^2\leq r^2\},\\
B_r &= \{\sigma\in\R^{d\times d}\; : \; \|\sigma\|\leq r\},\\
S_{c_1, c_2} &= \{\sigma\in\S^{d}_{+}\; : \; c_1I_d\preceq \sigma\preceq c_2I_d\}.\eeqs
In particular, note that $S_{0, r}\subset B_r$. 
\subsection{Change of coordinates}Recall that $\mhat$ is the unique minimizer of $V=nv$, and that $H_V=\nabla^2V(\mhat)$. Define \beq\label{Wdef}W(x) = V(\mhat + H_V^{-1/2}x).\eeq 
\begin{remark}
Note that this change of variables is different from the one in Section~\ref{sec:proof}. The change of variables~\eqref{rhodef} in that section already presumes existence and uniqueness of $\mvi\pi$ and $\Svi\pi$, which is precisely what we will prove in the present section. The transformation $V\mapsto W$ in~\eqref{Wdef} simply recenters the minimum to be at zero, and rescales the Hessian at the minimum to be the identity matrix.
\end{remark}

%In this section, we use $m\in\R^d,\sigma\in\R^{d\times d}$ to denote generic arguments. Consider the equations~\eqref{kl-opt-scalefree}, which we rewrite in the following form:
%Note that these equations are well-defined for all $(m,\sigma)\in\R^{d\times d}\times\R^d$, although we can only expect uniqueness of solutions in a subset of $\S^d_{++}\times\R^d$; indeed, ~\eqref{kl-opt-1} and~\eqref{kl-opt-2} only depend on $\sigma$ through $S=\sigma\sigma^T$, which has multiple solutions $\sigma$.  
%We now restate Lemma~\ref{lma:intro:exist:W} using the following notation:
\begin{lemma}\label{lma:m-sig-soln}
Let $W$ be as in~\eqref{Wdef} for $V=nv$, where $v$ satisfies Assumptions~\ref{assume:1} and~\ref{assume:glob}. Then there exists a unique pair $(m,\sigma)\in B_{2\sqrt2}(0, 0)\cap \R^d\times S_{0,\sqrt2}$ satisfying 
\begin{align}
\E[\nabla W(m+\sigma Z)] &= 0,\label{kl-opt-1}\\
\E[\nabla^2 W(m+\sigma Z)] &= (\sigma\sigma^T)^{-1}\label{kl-opt-2}.
\end{align}and this pair is such that $\sigma\in S_{\sqrt{2/3},\sqrt2}$.
\end{lemma}
Using this lemma, a rescaling argument easily proves Lemma~\ref{lma:intro:exist:V}; see Appendix~\ref{supp:aux}.
\begin{remark}\label{rk:mVW}
The relationship between $(m,\sigma)$ and $(\mvi\pi,\Svi\pi)$ is as follows:
\beq\label{mVW}\mvi\pi = \mhat + H_V^{-1/2}m,\qquad \Svi\pi=H_V^{-1/2}\sigma\sigma^TH_V^{-1/2}.\eeq
\end{remark}
\subsection{Proof of Lemma~\ref{lma:m-sig-soln}}\label{sec:proof-exist}
We first sketch the proof of Lemma~\ref{lma:m-sig-soln}, and then formally state the lemmas from which the result will follow. The proofs of the lemmas can be found in Appendix~\ref{app:sec:m-sig-exist}.

Let $f:\R^d\times\R^{d\times d}\to \R^d$ be given by $f(m,\sigma)= \E[\nabla W(m+\sigma Z)].$ Note that $f(0, 0)=0$, so by the Implicit Function Theorem, there exists a map $m(\sigma)$ defined in a neighborhood of $\sigma =0$ such that $f(m(\sigma),\sigma)=0$. In Lemma~\ref{m-of-sig}, we make this statement quantitative, showing that for $r=2\sqrt2$ we have the following result: for every $\sigma\in B_{r/2}$ there is a unique $m=m(\sigma)$ such that $(m,\sigma)\in B_r(0, 0)$ and $f(m,\sigma)=0$. Since $S_{0,r/2}\subset B_{r/2}$, we have in particular that any solution $(m,\sigma)$ to~\eqref{kl-opt-1} in the region $B_r(0, 0)\cap \R^d\times S_{0,r/2}$ is of the form $(m(\sigma),\sigma)$. Thus it remains to prove there exists a unique solution $\sigma\in S_{0, r/2}$ to the equation $\E[\nabla^2 W(m(\sigma)+\sigma Z)] = (\sigma\sigma^T)^{-1}$.

To do so, we rewrite this equation as $F(\sigma)=\sigma$, where $$F(\sigma) = \E[\nabla^2W(m(\sigma)+\sigma Z)]^{-1/2}.$$ We show in Lemma~\ref{lma:contract} that $F$ is well-defined on $S_{0,r/2}$, a contraction, and satisfies $F(S_{0,r/2})\subset S_{c_1,c_2}\subset S_{0, r/2}$. Thus by the Contraction Mapping Theorem, there is a unique $\sigma\in S_{0,r/2}$ satisfying $F(\sigma)=\sigma$. But since $F$ maps $S_{0,r/2}$ to $S_{c_1,c_2}$, the fixed point $\sigma$ necessarily lies in $S_{c_1,c_2}$.
%in particular $F:S_{c_1,c_2}\to S_{c_1,c_2}$ is a contraction and hence there is a unique positive definite $\sigma\in S_{c_1,c_2}$ such that $F(\sigma)=\sigma$. 
This finishes the proof.\\

\noindent Using a quantitative statement of the Inverse Function Theorem given in~\cite{langanalysis}, the following lemma determines the size of the neighborhood in which the map $m(\sigma)$ is defined. 

\begin{lemma}\label{IVT}
Let $f = (f_1,\dots, f_d):\R^d\times \R^{d\times d} \to\R^d$ be $C^3$, where $\R^{d\times d}$ is the set of $d\times d$ matrices, endowed with the standard matrix operator norm. Suppose $f(0, 0)=0$, $\nabla_\sigma f(0, 0)=0$, $\nabla_mf(m,\sigma)$ is symmetric for all $m,\sigma$, and $\nabla_mf(0, 0)=I_d$. Let $r>0$ be such that
\beq\label{supsigmam}
\sup_{(m,\sigma)\in B_r(0, 0)}\|\nabla f(m,\sigma) - \nabla f(0,m_*)\|_\op \leq \frac{1}4.\eeq Then for each $\sigma\in \R^{d\times d}$ such that $\|\sigma\|\leq r/2$ 
there exists a unique $m=m(\sigma)\in\R^d$ such that $f(m(\sigma),\sigma)=0$ and $(m(\sigma),\sigma) \in B_r(0, 0)$. Furthermore, the map $\sigma\mapsto m(\sigma)$ is $C^2$, with
 \beqs\label{IVT-props}
\frac{1}{2}I_d\preceq\nabla_mf(m,\sigma)\big\vert_{m=m(\sigma)}&\preceq \frac{3}{2} I_d,\\ 
\|\nabla_\sigma m(\sigma)\|_\op &\leq 1.\eeqs
\end{lemma}
See Appendix~\ref{app:sec:m-sig-exist} for careful definitions of the norms appearing above, as well as the proof of the lemma. 
\begin{lemma}\label{m-of-sig} Let $f:\R^d\times\R^{d\times d}\to\R^d$ be given by $f(m,\sigma) = \E[\nabla W(\sigma Z+m)]$. Then all the conditions of Lemma~\ref{IVT} are satisfied; in particular,~\eqref{supsigmam} is satisfied with $r=2\sqrt 2$.
Thus the conclusions of Lemma~\ref{IVT} hold with this choice of $r$.
\end{lemma}
\begin{lemma}\label{lma:contract}
Let $r=2\sqrt2$ and $\sigma\in S_{0,r/2}\mapsto m(\sigma)\in\R^d$ be the restriction of the map furnished by Lemmas~\ref{IVT} and~\ref{m-of-sig} to symmetric nonnegative matrices. Then the function $F$ given by
$$F(\sigma) =\E[\nabla^2W(m(\sigma)+\sigma Z)]^{-1/2}$$% = \frac{1}{\sqrt n}\nabla_mf(m,\sigma(m))^{-1/2},$$ 
is well-defined and a strict contraction on $S_{0,r/2}$. Moreover, 
$$F(S_{0,r/2})\subseteq S_{c_1,c_2}\subseteq S_{0,r/2},\qquad \text{where }\;c_1 = \sqrt{2/3}, \; \; c_2 = \sqrt{2}=r/2.$$ \end{lemma}
This lemma concludes the proof of Lemma~\ref{lma:m-sig-soln} since by the Contraction Mapping Theorem there is a unique fixed point $\sigma\in S_{0,r/2}$ of $F$, and $F(\sigma)=\sigma$ is simply a reformulation of the second optimality equation~\eqref{kl-opt-2}. We know $\sigma$ must lie in $S_{c_1,c_2}$ since $F$ maps $S_{0,r/2}$ to this set. 

\section*{Acknowledgments}
A.~Katsevich is supported by NSF grant DMS-2202963. P.~Rigollet is supported by NSF grants IIS-1838071, DMS-2022448, and CCF-2106377.

\appendix

We review some notation used throughout the appendix. For a function $f:\R^d\to\R$ such that $\int |f|^pd\gamma<\infty$, we define
$$\|f\|_{p} =\l( \int |f|^pd\gamma\r)^{\frac1p}.$$ For any other measure $\lambda$, we indicate the dependence of the $p$-norm on $\lambda$ explicitly. In particular, we let
$$\|f\|_{L^p(\lambda)} =\l( \int |f|^pd\lambda\r)^{\frac1p},\quad \Var_{\lambda}(f)=\int (f-\lambda(f))^2d\lambda.$$

\section{Proofs from Section~\ref{sec:main}}\label{app:sec:main}
\begin{proof}[Proof of Corollary~\ref{corr:TV}]
 Let $g=\mathbbm{1}_A$. Note that $|g(x)-\hat\pi(g)|\leq1$ for all $x$, so $g$ satisfies~\eqref{gcond} with $R_g=0$. Also, $\Var_{\hat\pi}(g)\leq1$. We now apply~\eqref{overallbd1V} and~\eqref{overallbd3V} of Theorem~\ref{thm:Vgen} to conclude.
 \end{proof}
 \begin{proof}[Proof of Corollary~\ref{corr:meancov}]
Fix $\|u\|=1$, and let $g_u(x)=u^T\Svi\pi^{-1/2}(x-\mvi\pi)$. We have
\beqsn
\|\Svi\pi^{-1/2}(\mpi - \mvi\pi)\| & = \sup_{\|u\|=1}\int g_ud\pi - \int g_ud\hat\pi.\eeqsn Now, $g_u(x)=f_u(\Svi\pi^{-1/2}(x-\mvi\pi))$, where $f_u(y)=u^Ty$. We have
$$|f_u(y)-f(0)|\leq \|y\|\leq e^{2\|y\|^{1/2}}\quad\forall y\in\R^d,$$ so $f$ satisfies~\eqref{fcondeasy} with $\alpha=1/2$ and $C_{f_u}=2$. Therefore by Lemma~\ref{lma:gsuff}, we have that $g_u$ satisfies~\eqref{gcond} with $R_g=C(c_0^{-1})$ (note $C_{f_u}$ and $\alpha$ are absolute constants, so we don't include them). Finally, note that $\Var_{\hat\pi}(g_u) = \Var(u^TZ)=1$ for all $\|u\|=1$.  An application of~\eqref{overallbd4V} from Theorem~\ref{thm:Vgen} concludes the proof of the bound on the mean error. Since $R_g$ depends only on $c_0^{-1}$, the suppressed constant depends only on $c_0^{-1}$ and $q$.

To bound the covariance error, we consider the function $g_u^2$, which is even about $\mvi\pi$. Now, 
\beqs
\|\Svi\pi^{-1/2}&(\Sigpi -\Svi\pi)\Svi\pi^{-1/2}\|= \sup_{\|u\|=1}\l|\Var_{X\sim\pi}(g_u(X))-1\r|\\
&\leq\sup_{\|u\|=1}\l|\E_{X\sim\pi}\l[g_u(X)^2\r]-1\r| +  \sup_{\|u\|=1}\E[g_u(X)]^2\\
&=\sup_{\|u\|=1}\l|\int  g_u^2d\pi - \int  g_u^2d\hat\pi\r| +  \|\Svi\pi^{-1/2}(\mpi - \mvi\pi )\|^2
\eeqs For the bound on $\|\Svi\pi^{-1/2}(\mpi - \mvi\pi)\|^2$, we use the mean error we have just proved. For the bound on $\int g_u^2d\pi - \int g_u^2d\hat\pi$, we use that $g_u^2(x) = f_u(\Svi\pi^{-1/2}(x-\mvi\pi))$, where $f_u(y)=(u^Ty)^2$. We have
$$|f_u(y)-f(0)|\leq \|y\|^2\leq e^{4\|y\|^{1/4}}\quad\forall y\in\R^d,$$ so $f$ satisfies~\eqref{fcondeasy} with $\alpha=1/4$ and $C_{f_u}=4$. Therefore by Lemma~\ref{lma:gsuff}, we have that $g_u^2$ satisfies~\eqref{gcond} with $C_{g_u^2}=C(c_0^{-1})$ (note $C_{f_u}$ and $\alpha$ are absolute constants, so we don't include them). 
Finally, note that $\Var_{\hat\pi}(g_u^2)=\Var((u^TZ)^2)\leq 3$ for all $\|u\|=1$. The bound on $\int g_u^2d\pi - \int g_u^2d\hat\pi$ now follows from~\eqref{overallbd3V} of Theorem~\ref{thm:Vgen}. Since $C_{g_u^2}$ depends only on $c_0^{-1}$, the suppressed constant depends only on $c_0^{-1}$ and $q$.
\end{proof}
\begin{proof}[Proof of Lemma~\ref{lma:c0convex}]
Let $x$ be such that $\|x-\mhat\|_{H_v}=(1/2)\sqrt{d/n}$ for some $r>0$. A Taylor expansion of $v$ around $\mhat$ gives
\beqs
v(x)-v(\mhat) &= \frac12\|x-\mhat\|_{H_v}^2 + \frac{1}{3!}\la\nabla^3v(\xi),(x-\mhat)^{\otimes3}\ra \\
&\geq \frac12\|x-\mhat\|_{H_v}^2\l(1-\frac13\|\nabla^3v(\xi)\|_{H_v}\|x-\mhat\|_{H_v}\r)\\
&\geq \frac12\|x-\mhat\|_{H_v}^2\l(1 - \frac{a_3}{6}\sqrt{d/n}(1+(1/2)^q)\r)\geq\frac14\|x-\mhat\|_{H_v}^2
\eeqs using the assumption $a_3(1+(1/2)^q)\sqrt{d/n}\leq3$ to get the last inequality. Therefore,
\beq
\inf_{\|x-\mhat\|_{H_v}=(1/2)\sqrt{d/n}}\frac{v(x)-v(\mhat)}{\|x-\mhat\|_{H_v}} \geq\frac18\sqrt{d/n}
\eeq
Now, since $v$ is convex, we have
$$
\frac{v(y)-v(\mhat)}{\|y-\mhat\|_{H_v}}\geq \inf_{\|x-\mhat\|_{H_v}=(1/2)\sqrt{d/n}}\frac{v(y=x)-v(\mhat)}{\|x-\mhat\|_{H_v}}\geq\frac18\sqrt{d/n}$$ for all $\|y-\mhat\|_{H_v}\geq(1/2)\sqrt{d/n}$.
\end{proof}
For the proof of Lemma~\ref{lma:gsuff}, see Section~\ref{supp:aux}.

\section{Logistic Regression Example}\label{app:log}
\subsubsection*{Details of Numerical Simulation}
For the numerical simulation displayed in Figure~\ref{fig:demo}, we take $d=2$ and $n=100,200,\dots, 1000$. For each $n$, we draw ten sets of covariates $x_i, i=1,\dots, n$ from $\mathcal N(0, \lambda^2I_d)$ with $\lambda=\sqrt5$, yielding ten posterior distributions $\pi_n(\cdot\mid x_{1:n})$. For each $\pi_n$ we compute the ground truth mean and covariance by directly evaluating the integrals, using a regularly spaced grid (this is feasible in two dimensions). The mode $m_*$ of $\pi_n$ is found by a standard optimization procedure, and the Gaussian VI estimates $\hat m,\hat S$ are computed using the procedure described in~\cite{gauss-VI}. We used the authors' implementation of this algorithm, found at \url{https://github.com/marc-h-lambert/W-VI}. We then compute the Laplace and VI mean and covariance approximation errors for each $n$ and each of the ten posteriors at a given $n$. The solid lines in Figure~\ref{fig:demo} depict the average approximation errors over the ten distributions at each $n$. The shaded regions depict the spread of the middle eight out of ten approximation errors.

\subsubsection*{Verifying the Assumptions}
\begin{proof}[Proof of Lemma~\ref{lma:grad}]
Let $S=\frac1n\sum_{i=1}^n(Y_i-\E[Y_i\mid X_i])$ and $\mathcal N$ be a $1/2$-net of the sphere in $\R^d$, so that $\inf_{u\in\mathcal N}\|u-w\|\leq 1/2$ for all $\|w\|=1$. Standard arguments show we can take $\mathcal N$ to have at most $5^d$ elements. For some $s\geq0$, we have
\beq\label{P1}\mathbb P(\|S\| \geq 2s\sqrt{d/n})\leq \mathbb P(\sup_{u\in\mathcal N}u^TS\geq s\sqrt{d/n})\leq 5^d\sup_{\|u\|=1}\mathbb P(u^TS\geq s\sqrt{d/n})\eeq using a union bound and the fact that $|\mathcal N|\leq 5^d$. For a fixed unit vector $u$, we have
\beqs\label{P2}
\mathbb P(u^TS\geq s\sqrt{d/n}) &= \E\l[\mathbb P\l(u^TS\geq s\sqrt{d/n}\mid \{X_i\}_{i=1}^n\r)\r]\\
&=\E\l[\mathbb P\l(\sum_{i=1}^n(Y_i-\E[Y_i\mid X_i])u^TX_i\geq s\sqrt{nd}\mid \{X_i\}_{i=1}^n\r)\r]\\
&\leq \E\l[\e\l(-\frac{2s^2nd}{\sum_{i=1}^n(u^TX_i)^2}\r)\r] = \E\l[\e\l(-\frac{2s^2nd}{\|Z\|^2}\r)\r].\eeqs
In the last line, $Z\sim\mathcal N(0, I_n)$. To get the third line, we used Hoeffding's inequality. Now, we have
\beqs\label{P3}
\E\l[\e\l(-\frac{2s^2nd}{\|Z\|^2}\r)\r]&\leq \mathbb P(\|Z\|\geq 2\sqrt n) + \e\l(-s^2d/2\r)\leq e^{-n/2} + e^{-s^2d/2}.
\eeqs 
Combining~\eqref{P1},~\eqref{P2},~\eqref{P3} gives
\beqs
\mathbb P(\|S\|\geq 2s\sqrt{d/n})&\leq e^{d\log 5 - n/2} + e^{d(\log 5 -s^2/2)} \eeqs%&\leq e^{-n/4} + e^{-d/2},\eeqs using that $d/n<1/4\log3$ to get the last line.
Taking $s= \log n$ gives 
\beqsn
\mathbb P(\|S\|\geq 2\log n\sqrt{d/n})&\leq e^{d\log 5 - n/2} + e^{d(\log 5 -\frac12\log^2n)} \\
&\leq e^{-n/4} + e^{-d(\log n)^2/4}\leq e^{- n/4} + n^{-d/4}\eeqsn using that $\log 5 -\frac12\log^2 n\leq -\frac14\log^2n$ when $n\geq13$, and that $d\log 5-n/2<-n/4$ when $d/n\leq0.1$. Finally, note that
$$2\log n\sqrt{d/n} = \frac{2\log n}{n^{1/4}}\l(\frac d{\sqrt n}\r)^{\frac12}= \frac{8\log (n^{1/4})}{n^{1/4}}\l(\frac d{\sqrt n}\r)^{\frac12}\leq 8\l(\frac d{\sqrt n}\r)^{\frac12}.$$
\end{proof}
\begin{proof}[Proof of Lemma~\ref{lma:34deriv}]
First, note that $\nabla^kv(\theta)=\frac1n\sum_{i=1}^n\psi^{(k)}(X_i^T\theta)X_i^{\otimes k}$ for $k=3,4$. Therefore, 
\beq\label{nablakv2}
\|\nabla^kv(\theta)\|\les 1 + \sup_{\|u\|=1}\frac1n\sum_{i=1}^n|u^TX_i|^4,\quad k=3,4.
\eeq using also that $\|\psi^{(k)}\|_\infty\les1$. Now, we apply the following result.
\begin{theorem}[Adapted from Theorem 4.2 in~\cite{adamczak2010quantitative}]\label{adam}
Let $X_1,\dots, X_n$ be i.i.d. standard Gaussian. There exist absolute constants $C, C'$ such that if $Ct^8[\log(2t^2)]^6\leq n/d^2$ then
\beqs\label{c34logreg}
\sup_{\|u\|=1}\frac1n\sum_{i=1}^n|u^TX_i|^4 \leq 4
\eeqs with probability at least $1-e^{-C't\sqrt d}$. 
\end{theorem}
We apply the theorem with $t=c(n/d^2)^{1/9}$. Let $M=n/d^2$ and suppose $M\geq1$. The condition $Ct^8[\log(2t^2)]^6\leq n/d^2$ can then be written as $Cc^8\log(2c^2M^{2/9})\leq M^{1/9}$ and this can be satisfied by taking $c$ a small enough absolute constant. Thus the probability that~\eqref{c34logreg} holds is $1-\e(-C''(nd)^{1/9})$ for some other absolute constant $C''>0$. Combining~\eqref{c34logreg} with~\eqref{nablakv2} concludes the proof.
\end{proof}
\begin{proof}[Proof of Corollary~\ref{corr:log}]
We claim the statements hold on the event $E=E_1\cap E_2\cap E_3$, where $E_1, E_2, E_3$ are defined in Lemmas~\ref{lma:grad},~\ref{lma:hess},~\ref{lma:34deriv}, respectively.  Note that $\mathbb P(E)$ satisfies the desired lower bound. Now, on $E$, we have $\|\nabla\ell(\theta_0)\|\leq 8(d/\sqrt n)^{1/2}$ and $\nabla^2\ell(\theta_0)\succeq \tau(\|\theta_0\|)I_d$. Furthermore, if $\|\theta'-\theta_0\|\leq s$, then $\|\nabla^2\ell(\theta') - \nabla^2\ell(\theta_0)\|\leq C_3s.$ Let $\lambda=\tau(\|\theta_0\|)$ and $s=4(d/\sqrt n)^{1/2}/\lambda$. If $(d/\sqrt n)^{1/2}\leq \lambda^2/(16C_3) = \tau(\|\theta_0\|)^2/(16C_3)$, then the inequalities~\eqref{freq} are all satisfied for $f=\ell$ and $\theta=\theta_0$. We therefore conclude that there exists a unique global minimizer $\theta_{\ell}^*$ of $\ell$, which satisfies $$\|\theta_{\ell}^*-\theta_0\|\leq (4/\tau(\|\theta_0\|))(d/\sqrt n)^{1/2}\leq C(d/\sqrt n)^{1/2}$$ Next, note that on $E$, we have 
$$\|\nabla v(\theta_{\ell}^*)\|= \|\nabla\ell(\theta_{\ell}^*) + n^{-1}\Sigma^{-1}\theta_{\ell}^*\| =n^{-1}\|\Sigma^{-1}\theta_{\ell}^*\|\les n^{-1}\|\theta_{\ell}^*\|$$ and 
$$\nabla^2v(\theta_{\ell}^*)\succeq \nabla^2\ell(\theta_{\ell}^*) \succeq \tau(C)I_d,$$ since $\|\theta_{\ell}^*\|\leq C(d/\sqrt n)^{1/2}+\|\theta_0\|\leq C$. Furthermore, $\sup_\theta\|\nabla^3v(\theta)\|\leq C_3$, since $\nabla^3v=\nabla^3\ell$. Therefore, since $n^{-1}$ is assumed to be sufficiently small, we can use Lemma~\ref{IVT-LL} to conclude $v$ has a critical point (and therefore a unique global minimizer) $\theta^*$, and $\|\theta^* - \theta_{\ell}^*\|\les n^{-1}$.  That $\|\theta^*\|\leq C$ is bounded follows from the bounds on $\|\theta_0\|,\|\theta_{\ell}^*-\theta_0\|,\|\theta^* - \theta_{\ell}^*\|$. Furthermore, we have $\lambda_{\min}(H_v)\geq\lambda_{\min}(\nabla^2\ell(\theta^*))\geq \tau(C)$ using the definition of event $E_2$. To prove point 3, we use the lower bound on $\lambda_{\min}(H_v)$ from point 2, combined with Lemma~\ref{lma:34deriv} and the fact that $\nabla^kv=\nabla^k\ell$, $k\geq3$, to get
\beqs
\|\nabla^kv(\theta)\|_{H_v}\leq\lambda_{\min}(H_v)^{-k/2}\|\nabla^kv(\theta)\|\leq C
\eeqs on event $E$. Finally, point 4 is satisfied by convexity of $v$, using Lemma~\ref{lma:c0convex} and the boundedness of $a_3$ (needed to verify the assumption of the lemma).
\end{proof}

\section{Properties of $W$ and $\Wbar$}In this section, we prove several key properties of the rescaled functions $W(x) = V(\mhat + H_V^{-1/2}x)$ and $\Wbar(x)= V(\mvi\pi + \Svi\pi^{1/2}x)=V(T^{-1}(x))$. These properties are used in the proofs in Sections~\ref{sec:proof} and~\ref{sec:m-sig-exist}. First, note that for a measure $\lambda$ on $\R^d$, we have
$$\l\|\|\nabla^kW\|\r\|_{L^p(\lambda)} = \l(\int\|\nabla^kW(x)\|^pd\lambda(x)\r)^{1/p},$$ where the norm inside the integral is the standard tensor operator norm
\begin{lemma}\label{lma:Wprops}
If $v$ satisfies Assumptions~\ref{assume:1},~\ref{assume:glob},~\ref{assume:c0} then 0 is the unique global minimizer of $W$, with $\nabla^2W(0)=I_d$, and
\beqs\label{W34glob}
\|\nabla^3W(x)\|&\leq \frac{a_3}{\sqrt n}(1+\|x/\sqrt d\|^q),\quad\forall x\in\R^d\\
\|\nabla^4W(x)\|&\leq \frac{a_4}{n}(1+\|x/\sqrt d\|^q),\quad\forall x\in\R^d.
\eeqs
Furthermore, 
\begin{align}
W(x)-W(0)&\les  1+\|x\|^{3+q}\quad\forall x\in\R^d,\label{Wup}\\
W(x) - W(0)&\geq c_0\sqrt d\|x\| \quad\forall \|x\|\geq \sqrt{\frac23}\sqrt {d}, \label{Wlow}
\end{align}
Finally, for $m\in\R^d$, $\sigma\in\R^{d\times d}$, we have
\beqs\label{Enabla3W}
\l\|\|\nabla^3W\|\r\|_{L^p(\mathcal N(m,\sigma\sigma^T))}\les_{p,q}  &\frac{a_3}{\sqrt n}\l(1+\|m/\sqrt d\|^{q} + \|\sigma\|^{q}\r), \\
\l \| \|\nabla^4W\| \r\|_{L^p(\mathcal N(m,\sigma\sigma^T))}\les_{p,q} &\frac{a_4}{n}\l(1+\|m/\sqrt d\|^{q} + \|\sigma\|^{q}\r).
\eeqs 
\end{lemma}
See the end of the section for the proof.  Next, we collect properties of the function $\Wbar(x)= V(\mvi\pi + \Svi\pi^{1/2}x)=V(T^{-1}(x))$ used in the proof of Theorem~\ref{thm:W}. Here, $\mvi\pi,\Svi\pi$ are the solutions to the first-order optimality conditions from Lemma~\ref{lma:intro:exist:V}. We first rewrite $\Wbar$ in terms of $W$. To do so, let $\hat m_0=m, \hat \sigma_0=\sigma$, where $(m,\sigma)$ is the unique solution to~\eqref{kl-opt-1},~\eqref{kl-opt-2} furnished by Lemma~\ref{lma:m-sig-soln} (these are just the rescaled optimality conditions). Using the relationship~\eqref{mVW} between $(m,\sigma)=(\hat m_0,\hat\sigma_0)$ and $(\mvi\pi,\Svi\pi^{1/2})$ given in Remark~\ref{rk:mVW}, it is straightforward to show that
\beq\label{barWviaW}\Wbar (x) = V(\mvi\pi + \Svi\pi^{1/2}x)=W(\hat m_0 + \hat \sigma_0 x).\eeq
Also, recall from Lemma~\ref{lma:m-sig-soln} that
\beq\label{hatmsig}
\|\hat m_0\|\leq 2\sqrt2,\qquad \sqrt{\frac23}I_d\preceq\hat \sigma_0\preceq \sqrt 2I_d.
\eeq
We now have the following properties of $\Wbar$.
\begin{lemma}\label{lma:WtobarW} It holds %, where $\|\mu\|\leq C\sqrt d$ and $\|\sigma'\|\leq C$. 
\begin{align}
|\Wbar (x)-\Wbar (0)|\les_q &1+\|x\|^{3+q}\quad\forall x\in\R^d,\label{barWup}\\
\Wbar (x)-\Wbar (0)\geq &\frac{c_0\sqrt d}{2}\|x\| -C(q) \quad\forall \|x\|\geq \sqrt d + 8\sqrt2.\label{barWlow}
\end{align} Furthermore, if $0<s\leq1$, then
\begin{align}
\l \| \|\nabla^3\Wbar\| \r\|_{L^p(\mathcal N(\mu, s^2I_d))}\les_q  &\frac{a_3}{\sqrt n}\l(1+\|\mu/\sqrt d\|^q\r),\label{barW3}\\
\l \| \|\nabla^4\Wbar\| \r\|_{L^p(\mathcal N(\mu, s^2I_d))}\les_q  & \frac{a_4}{n}\l(1+\|\mu/\sqrt d\|^q\r).\label{barW4}
\end{align} In particular,
\begin{align}
d\l\|\|\nabla^3\Wbar\|\r\|_{L^p(\mathcal N(0, s^2I_d))}&\les_q  a_3\epsilon,\label{barW3spec}\\
d^2\l \| \|\nabla^4\Wbar\| \r\|_{L^p(\mathcal N(0, s^2I_d))}&\les_q a_4\epsilon^2,\label{barW4spec}
\end{align} where $\epsilon=d/\sqrt n$.\end{lemma}
 We now prove Lemmas~\ref{lma:Wprops} and~\ref{lma:WtobarW}.
\begin{proof}[Proof of Lemma~\ref{lma:Wprops}]
We first note that \beq\label{WVderivconv}\|\nabla^kW(x)\| = n^{1-k/2}\|\nabla^kv(\mhat + H_v^{-1/2}x/\sqrt n)\|_{H_v}.\eeq Now, using Assumption~\ref{assume:glob}, we have
\beqs
\|\nabla^kW(x)\| &= n^{1-k/2}\|\nabla^kv(\mhat + H_v^{-1/2}x)\|_{H_v} \\
&\leq \frac{a_k}{n^{\frac k2-1}}\l[1+\l(\sqrt{n/d}\|\mhat + H_v^{-1/2}x/\sqrt n-\mhat\|_{H_v}\r)^q\r]\\
&= \frac{a_k}{n^{\frac k2-1}}\l[1+\|x/\sqrt d\|^q\r],
\eeqs as desired. To prove~\eqref{Wup} we Taylor expand $W$ about 0 and apply~\eqref{W34glob}. We get
\beqs
W(x)-W(0)&=\frac{\|x\|^2}{2} +\frac{1}{3!}\la\nabla^3W(tx),x^{\otimes 3}\ra \\
&\leq \frac{\|x\|^2}{2} +\frac{\|x\|^3}{6}(a_3/\sqrt n)(1+\|x/\sqrt d\|^q)\\
& \les 1+\|x\|^{3+q},
\eeqs recalling $a_3/\sqrt n\leq1$ by~\eqref{cconds}. The inequality~\eqref{Wlow} follows from the definition of $W$ and~\eqref{WVderivconv}.
%\frac{\|x\|^2}{2}+\frac{\cg{3}}{6\sqrt n}\|x\|^3(1+\|x/\sqrt d\|^q)\leq \l(1+\frac{\cg{3}}{\sqrt n}\r)(1+\|x\|^{3+q})
To prove~\eqref{Enabla3W}, we use~\eqref{W34glob} to get
\beqs\label{E3Wm}
\E[\|\nabla^3W(m+\sigma Z)\|^{p}]&\leq \l(a_3/\sqrt n\r)^{p}\E\l[ 1+ (\|m+\sigma Z\|/\sqrt d)^{pq}\r]\\
&\les_{p,q}  \l(a_3/\sqrt n\r)^{p}\l(1+\|m/\sqrt d\|^{pq} + \|\sigma\|^{pq}\r).\eeqs We take the $p$th root to conclude. The bound on $\E[\|\nabla^4W(m+\sigma Z)\|^p]$ is shown analogously.
\end{proof}
\begin{proof}[Proof Lemma~\ref{lma:WtobarW}] Recall the expression~\eqref{barWviaW} for $\Wbar $ in terms of $W$. We start by bounding $W(\hat m_0)-W(0)$. A Taylor expansion, the fact that $\|\hat m_0\|\leq2\sqrt2$, and~\eqref{cconds}, gives
$$0\leq W(\hat m_0)-W(0)\leq \frac{\|\hat m_0\|^2}{2} + \frac{\|\hat m_0\|^3}{3!}\frac{a_3(1+\|\hat m_0/\sqrt d\|^q)}{\sqrt n}\leq C(q).$$ Now we prove~\eqref{barWup}. Using~\eqref{Wup} and~\eqref{hatmsig}, we have
\beqs
|\Wbar (x)-\Wbar (0)|&= |W(\hat m_0+\hat \sigma_0 x)-W(\hat m_0)|\leq |W(\hat m_0+\hat \sigma_0 x)-W(0)| + C(q)\\
&\les 1+\|\hat m_0+\hat \sigma_0 x\|^{3+q}+C(q)\\
&\les_q  1+\|x\|^{3+q}.
\eeqs To prove~\eqref{barWlow}, fix $\|x\|\geq \sqrt d+9$. Then 
$$\|\hat m_0+\hat \sigma_0 x\|\geq \|\hat \sigma_0 x\|-\|\hat m_0\|\geq\lambda_{\min}(\hat \sigma_0)\|x\|-\|\hat m_0\|\geq \sqrt{\frac23}\|x\|-2\sqrt2\geq \frac12\sqrt d.$$
Hence
\beqs
\Wbar (x)-\Wbar (0) &= W(\hat m_0+\hat \sigma_0 x) - W(\hat m) \\
&= W(\hat m_0+\hat \sigma_0 x)-W(0)+W(0)-W(\hat m_0)\\
&\geq c_0\sqrt d\|\hat m_0+\hat \sigma_0 x\| + W(0)-W(\hat m_0)\\
&\geq c_0\sqrt d\l(\sqrt{\frac23}\|x\|-2\sqrt 2\r)-C(q)\\
&\geq \frac{c_0}{2}\sqrt d\|x\|-C(q).\eeqs  To get the last line, we used that $\sqrt{2/3}\|x\|-2\sqrt2\geq \|x\|/2$ when $\|x\|\geq9$. This concludes the proof of~\eqref{barWlow}. To prove~\eqref{barW3}, we use the boundedness of $\|\hat\sigma_0\|$ to get that
\beqs
\|\nabla^k\Wbar (y)\|\les\|\nabla^kW(\hat m_0+\hat \sigma_0 y)\|,\quad k=3,4.
\eeqs
Therefore,
\beq
\E\l[\|\nabla^3\Wbar (Y)\|^p\r]\les\E\l[\|\nabla^3W(\hat m_0+\hat \sigma _0Y)\|^p\r].\eeq 
Now, $\hat m_0+\hat\sigma_0 Y= \hat m_0+\hat\sigma_0(\mu+S^{1/2} Z) = m+\sigma Z$, where $m=\hat m_0+\hat\sigma_0 \mu$ and $\sigma=\hat\sigma_0 S^{1/2}$. Note that $\|m\|\leq 2\sqrt2+\sqrt2\|\mu\|$ and $\|\sigma\|\leq\sqrt2s\leq\sqrt 2$. Therefore, applying~\eqref{Enabla3W} from Lemma~\ref{lma:Wprops}, we get
\beqsn
\l \| \|\nabla^3\Wbar\| \r\|_{L^p(\mathcal N(\mu, s^2I_d))}&\les_q \frac{a_3}{\sqrt n}\l(1+\|\mu/\sqrt d\|^{q}\r).
\eeqsn The bound~\eqref{barW4} on $\l \| \|\nabla^4\Wbar\| \r\|_{L^p(\mathcal N(\mu, S))}$ is analogous. The bounds~\eqref{barW3spec},~\eqref{barW4spec} follow immediately from~\eqref{barW3} and~\eqref{barW4}. 
\end{proof}

\section{Proofs from Section~\ref{sec:proof}}\label{app:sec:proof}For the equivalence of~\eqref{overallbd2} and~\eqref{overallbd2V}, see Section~\ref{supp:aux}. Here, we prove the key lemmas from Section~\ref{subsec:key}. We start by making two observations: first, if $g$ satisfies~\eqref{gcond} and $f(x)=g(T^{-1}(x))=g(\mvi\pi +\Svi\pi^{1/2}x)$, then $f$ satisfies
\beq
|\fbar(x)|=|f(x)-\gamma(f)|\leq \e\l(\frac{c_0}{4}\sqrt d\|x\|\r)\quad\forall\|x\|\geq R_g\sqrt d.
\eeq
Second, using~\eqref{V0-expand} from the proof outline, note that
$$\Wbar(x)-\Wbar(0)=\delta + r_3(x)-\|x\|^2/2$$ for some constant $\delta$. Taking $x=0$ in this equation, we get that $0=\delta+r_3(0)$, i.e. $\delta=-r_3(0)$. But $r_3(0)=\frac1{3!}\la\A_3,\H_3(0)\ra+r_4(0)=r_4(0)$. Therefore, $\delta=-r_4(0)$, i.e. we have
\beq\label{rtrix2}\Wbar(x) - \Wbar(0)=-r_4(0) + r_3(x) +\frac{\|x\|^2}{2}.\eeq \\

Now, to prove Lemma~\ref{alldabounds}, we prove a number of supplementary lemmas. The overall proof structure is similar to that of~\cite{katsTDD}. The outline given in Section 4.3 of that work may be useful to follow the below proofs.
\begin{lemma}\label{lma:epsilon}Let $g, R_g$ be as in Theorem~\ref{thm:Vgen}, and $f=g\circ T$. 
Let $R\geq \max(R_g, 1+8\sqrt{2/d})$ and $\mathcal U(R)=\{\|x\|\leq R\sqrt d\}$ Then
 \beqs\label{totalRbd}
|E(f)| \les_q   &\l(1+\l\|e^{-\rtri}\mathbbm{1}_{\U(R)}\r\|_4\r)\|\fbar \|_2\|\rtri\|_{12}^3 +I(R, 9+3q,c_0/4),\\
%\eeqs %\beqs\label{totalRbdE}
|E| \les_q   &\l(1+\l\|e^{-\rtri}\mathbbm{1}_{\U(R)}\r\|_2\r)\|\rtri\|_{4}^2+I(R, 6+2q, c_0/2),
\eeqs where 
\beq\label{IRpb}
I(R, p, b) = \frac{1}{(2\pi)^{d/2}}\int_{\|x\|\geq R\sqrt d}^\infty \|x\|^{p}e^{-b\sqrt d\|x\|}dx.
\eeq
\end{lemma} See the end of this section for the proof. We see that we need to bound $I(R,p,b)$, $\l\|e^{-\rtri}\mathbbm{1}_{\U(R)}\r\|_4$, and $\|\rtri\|_p$. The bounds on the first two quantities are stated in the next two lemmas. The bound on $\|\rtri\|_p$ is given in Lemma~\ref{rUUc} in Section~\ref{app:subsec:Erkm}.
\begin{lemma}[Bound on tail integral]\label{lma:tail}There is some $C_p$ depending only on $p$ such that if $R\geq C_p(1\vee b^{-1})^2$ then $I(R,p,b)\les_p e^{-Rbd/2}.$ In particular, 
$$I(R, 9+3q, c_0/4)\vee I(R, 6+2q, c_0/2) \les_q e^{-Rc_0d/8}\quad \text{if}\quad R\geq C(q)(1\vee c_0^{-1})^2$$ for some $C(q)$ depending only on $q$.
\end{lemma} See Section~\ref{supp:aux} for the proof.
\begin{lemma}[Bound on exponential integral in $\U(R)$]\label{lma:expr3}
Let $R\geq 1$. Then 
$$\|e^{-\rtri}\mathbbm{1}_{\mathcal U(R)}\|_4\leq \e\l(C(q)R^{4+q}\omega\r),$$ where $\omega=a_3\epsilon+a_4\epsilon^2$.
\end{lemma} See the end of this section for the proof. With these lemmas in hand, as well as a bound on $\|\rtri\|_p$ in Lemma~\ref{rUUc} below, we can now prove Lemma~\ref{alldabounds}.
  \begin{proof}[Proof of Lemma~\ref{alldabounds}]
We have 
\beqs\label{A1A2}
 |L(f)|\leq\|\fbar \|_2(\|\rtri\|_2 + \|\rtri\|_4^2)\les_q\|\fbar\|_2\omega,
\eeqs
using Lemma~\ref{rUUc} and the fact that $\omega\leq1$ by~\eqref{cconds}. Next, let $R\geq \bar C:=\max(R_g, 1+8\sqrt{2/d}, C(q)(1\vee c_0^{-1})^2)$. We can then use~\eqref{totalRbd} from Lemma~\ref{lma:epsilon} to bound $|E|$ and $|E(f)|$. Furthermore, for this choice of $R$ we can use Lemma~\ref{lma:tail} to get that the tail integrals $I(R, 9+3q, c_0/4)$ and $I(R,6+2q,c_0/2)$ are both bounded by $e^{-Rc_0d/8}$. Finally, Lemma~\ref{lma:expr3} bounds $\|e^{-\rtri}\mathbbm{1}_{\mathcal U(R)}\|_4$, and Lemma~\ref{rUUc} bounds $\|r_3\|_p$. Combining all of these bounds,~\eqref{totalRbd} reduces to 
\beqs\label{EEf}
&|E| \les_q \omega^2\e\l(C(q)R^{4+q}\omega\r) + e^{-\frac{Rc_0}{8}d}\\
&|E(f)|\les_q  \|\fbar \|_2\omega^3\e\l(C(q)R^{4+q}\omega\r)+  e^{-\frac{Rc_0}{8}d}.\eeqs  Now, if $\omega=0$ then we get $|E|,|E(f)|\les e^{-\frac{Rc_0}{8}d}$, and we can take $R\to\infty$ to conclude $|E|,|E(f)|=0$. If $\omega>0$ then by~\eqref{cconds}, we know $0<\omega\leq1$. Let \beq\label{RCbar}R= \max\l(\bar C, \frac{8}{c_0d}\log(\omega^{-3})\r),\eeq so that $\e(-(Rc_0/8)d)\leq\omega^3$. Furthermore, using that $\log(1/\omega)^{4+q}\omega\leq C(q)$ since $\omega\leq1$, we have that for $R$ as in~\eqref{RCbar}, the exponent $C(q)R^{4+q}\omega$ in~\eqref{EEf} is bounded above by some fixed function $C(R_g, c_0^{-1},q)$. Using these bounds, we get
\beq
|E|\leq C(q, c_0^{-1}, R_g)\omega^2, \qquad |E(f)| \leq C(q,c_0^{-1}, R_g)(1+ \|\fbar \|_2)\omega^3
\eeq 
Note that this bound is still valid if $\omega=0$. Finally, recalling the definition of $\omega$ concludes the proof.
\end{proof}
Next, we prove Lemma~\ref{leading-leading}, which follows immediately from Lemma~\ref{rUUc} and a few algebraic manipulations.
\begin{proof}[Proof of Lemma~\ref{leading-leading}]
We have
\beqsn
L(f) &+  \int \fbar p_3d\gamma +\int \fbar\l(r_4- \frac12 p_3^2\r)d\gamma =\int\fbar \l(\l[-\rtri + \frac12\rtri^2\r]+\rtri -\frac12p_3^2\r)d\gamma \\
&=\frac12\int \fbar ((p_3+r_4)^2-p_3^2)d\gamma=\int\fbar p_3r_4d\gamma + \frac12\int \fbar r_4^2d\gamma.
\eeqsn
Next,
\beqsn
\bigg|\int\fbar p_3r_4d\gamma &+ \frac12\int \fbar r_4^2d\gamma\bigg|\leq \|\fbar \|_2\l(\|p_3\|_4\|r_4\|_4 + \|r_4\|_4^2\r)\\
&\les \|\fbar \|_2\l(a_3\epsilon a_4\epsilon^2 + (a_4\epsilon^2)^2\r)\\
&\les  \|\fbar \|_2\l(a_3a_4\epsilon^3+a_4^2\epsilon^4\r).
\eeqsn
Next, we have
\beqsn
\l|\int \fbar \l(r_4- \frac12p_3^2\r)d\gamma\r|\leq \|\fbar\|_2\l(\|r_4\|_2+\|p_3\|_4^2\r)\les \|\fbar\|_2(a_3^2+ a_4)\epsilon^2.\eeqsn Finally,
$$\l|\int fp_3d\gamma\r|\leq \|\fbar \|_2\|p_3\|_2\les \|\fbar\|_2a_3\epsilon.$$ 
\end{proof}

 \begin{proof}[Proof of Lemma~\ref{lma:epsilon}]
A Taylor remainder argument shows that
\beqsn
\l|e^{-\rtri}-1+\rtri-\rtri^2/2\r|&\leq |\rtri|^{3}+|\rtri|^{3}e^{-\rtri}.
\eeqsn Therefore,
 \beqs\label{epsgbd}
|E(f)|&\leq  \int |\fbar ||\rtri| ^{3}d\gamma + \int |\fbar ||\rtri|^3e^{-\rtri }d\gamma\\
& =  \int |\fbar ||\rtri| ^{3}d\gamma + \int |\fbar |\rtri ^{3}(e^{-\rtri }\mathbbm{1}_{\U(R)})d\gamma+ \int_{\U(R)^c} |\fbar ||\rtri |^{3}e^{-\rtri }d\gamma\\
&\leq  \|\fbar \|_2\|\rtri \|_{12}^{3}(1+\|e^{-\rtri }\mathbbm{1}_{\U(R)}\|_4) + \int_{\U(R)^c} |\fbar ||\rtri| ^{3}e^{-\rtri }d\gamma.\eeqs In the third line we used generalized H{\"o}lder with powers $2,4,4$. Next, by~\eqref{rtrix2} we have 
$$e^{-r_3(x)}d\gamma(x) = e^{-r_4(0)}(2\pi)^{-d/2}e^{\Wbar (0)-\Wbar (x)}dx.$$ We now apply the bound~\eqref{barWlow} on $\Wbar(0)-\Wbar(x)$ in the region $\|x\|\geq R\sqrt d\geq \sqrt d+8\sqrt 2$, to get the following bound on the tail integral:
\beqs\label{intURc}\int_{\U(R)^c}|\fbar | |\rtri| ^{3}e^{-\rtri}d\gamma(x)& = \frac{e^{-r_4(0)}}{(2\pi)^{\frac d2}}\int_{\U(R)^c}|\fbar | |\rtri| ^{3}e^{\Wbar(0)-\Wbar(x)}dx \\
&\les_q \frac{e^{-r_4(0)}}{(2\pi)^{\frac d2}}\int_{\U(R)^c}|\fbar | |\rtri| ^{3}e^{-\frac{c_0}{2}\sqrt d\|x\|}dx.\eeqs Next, we have by~\eqref{rtrix2} that $|\rtri(x)|\leq|\Wbar (x)-\Wbar (0)| +|r_4(0)| +\|x\|^2$. Therefore, by the first bound in Lemma~\ref{lma:WtobarW} we have
\beqs\label{rtrikles}
|\rtri(x)|^3&\les |\Wbar (x)-\Wbar (0)|^3 +|r_4(0)|^3 +\|x\|^{6} \\
&\les_q  1+\|x\|^{9+3q} + \|x\|^6+|r_4(0)|^3\\
&\les_q e^{|r_4(0)|}\|x\|^{9+3q}.
\eeqs We have used that $\|x\|\geq R\sqrt d\geq1$. Finally, note that $|\fbar(x) |\leq \e(c_0\sqrt d\|x\|/4)$ for $x\in\U(R)$, since $\|x\|\geq R\sqrt d\geq R_g\sqrt d$. Substituting this bound and~\eqref{rtrikles} into~\eqref{intURc}, and then substituting the resulting bound into~\eqref{epsgbd} gives
\beq\label{Eflesq}|E(f)|\les_q \|\fbar \|_2\|\rtri \|_{12}^{3}(1+\|e^{-\rtri }\mathbbm{1}_{\U(R)}\|_4) + e^{2|r_4(0)|}(2\pi)^{-\frac d2}\int_{\U(R)^c}\|x\|^{9+3q}e^{-\frac{c_0}{4}\sqrt d\|x\|}dx.\eeq Next, note that using Lemma~\ref{lma:r4ptwise} and~\eqref{cconds}, we have $|r_4(0)|\leq C(q)a_4\epsilon^2\leq C(q)$. Using this bound in~\eqref{Eflesq}, as well as the definition of $I(R, p, b)$ with $p=9+3q$ and $b=c_0/4$ finishes the proof. The bound on $|E|$ is shown analogously.
\end{proof}
\begin{proof}[Proof of Lemma~\ref{lma:expr3}]
Recall that $\A_3=\E[\nabla^3\Wbar (Z)]$, and that $r_3=\frac{1}{3!}\la \A_3,\H_3(x)\ra +r_4(x)$. Hence
\beqs\label{ertrinew}
\|e^{-\rtri}\mathbbm{1}_{\mathcal U(R)}\|_4 &\leq\sup_{\|x\|\leq R\sqrt d}e^{|r_4(x)|}\|e^{-\frac{1}{3!}\la \A_3,\H_3\ra}\mathbbm{1}_{\mathcal U(R)}\|_4\\
&\leq e^{C(q)R^{4+q}a_4\epsilon^2}\|e^{-\frac{1}{3!}\la \A_3,\H_3\ra}\mathbbm{1}_{\mathcal U(R)}\|_4,
\eeqs
using the pointwise bound on $|r_4|$ from Lemma~\ref{lma:r4ptwise} to get the second line. To bound $\|e^{-\frac{1}{3!}\la \A_3,\H_3\ra}\mathbbm{1}_{\mathcal U(R)}\|_4$, we use that if the restriction of $f$ to an open set $\U$ is Lipschitz has Lipschitz constant $L$, then
$$\l|\int_{\U}e^fd\gamma\r| \leq \e\l(CL^2+L\|f\|_1\r)$$ for some absolute constant $C$. This follows from Proposition 5.4.1 in~\cite{bakry2014analysis} and the fact that $\gamma$ satisfies a log Sobolev inequality with an absolute constant. We apply this result with $f=-\frac{1}{3!}\la \A_3,\H_3\ra$. We have
$$\|\nabla f(x)\|= \frac{1}{2}\|\la \A_3, xx^T-I_d\ra\|\les R^2d\|\A_3\| \les_q R^2a_3\epsilon,\quad\|x\|\leq R\sqrt d.$$ The bound on $d\|\A_3\|$ is by~\eqref{barW3spec} of Lemma~\ref{lma:WtobarW}. Therefore, $f$ has Lipschitz constant $L=C(q)R^2a_3\epsilon$ when restricted to $\U(R)$. Furthermore, $\|\la A_3,\H_3\ra\|_1\les d\|A_3\|\les a_3\epsilon$ by~\eqref{TH3}. Therefore,
$$\|e^{-\frac{1}{3!}\la \A_3,\H_3\ra}\mathbbm{1}_{\mathcal U(R)}\|_4\leq \e\l(C(q)R^4(a_3\epsilon)^2\r).$$ Substituting this bound into~\eqref{ertrinew}, we get
\beqs
\|e^{-\rtri}\mathbbm{1}_{\mathcal U(R)}\|_4\leq \e\l(C(q)\l[R^{4+q}a_4\epsilon^2 + R^4(a_3\epsilon)^2\r]\r) \leq  \e\l(C(q)R^{4+q}\omega\r),
\eeqs where $\omega=a_3\epsilon+a_4\epsilon^2$. Here we used that $R^4\leq R^{4+q}$ and $(a_3\epsilon)^2\leq a_3\epsilon$ by~\eqref{cconds}.
\end{proof}

\section{Hermite Series Remainder}\label{hermite-tail} 
\subsection{Brief Primer}\label{subsec:primer}
Here is a brief primer on Hermite polynomials, multinomials, and series expansions. We let $H_k:\R\to\R$, $k=0,1,2,\dots$ be the $k$th order probabilist's Hermite polynomial. We have $H_0(x)=1, H_1(x)=x, H_2(x)=x^2-1, H_3(x) = x^3-3x$. For all $k\geq 1$, we can generate $H_{k+1}$ from the recurrence relation
\beq\label{eq:recur}
H_{k+1}(x) = xH_{k}(x) - kH_{k-1}(x),\quad k\geq 1.
\eeq
In particular, $H_k(x)$ is an order $k$ polynomial given by a sum of monomials of the same parity as $k$. The $H_k$ are orthogonal with respect to the Gaussian measure; namely, we have $\E[H_k(Z)H_j(Z)] = k!\delta_{jk}$. We also note for future reference that
\beq\label{EZHk}\E[ZH_k(Z)H_{k+1}(Z)] = \E\l[\l(H_{k+1}(Z)+kH_{k-1}(Z)\r)H_{k+1}(Z)\r] = (k+1)!,\eeq using the recurrence relation~\eqref{eq:recur}.

The Hermite multinomials are given by products of Hermite polynomials, and are indexed by $\gamma\in\{0,1,2,\dots\}^d$. Let $\gamma=(\gamma_1,\dots,\gamma_d)$, with $\gamma_j\in \{0,1,2,\dots\}$.  Then
$$H_\gamma(x_1,\dots, x_d)=\prod_{j=1}^dH_{\gamma_j}(x_j),$$ which has order $|\gamma|:=\sum_{j=1}^d\gamma_j$. Note that if $|\gamma|=k$ then $H_\gamma(x)$ is given by a sum of monomials of the same parity as $k$. Indeed, each $H_{\gamma_j}(x_j)$ is a linear combination of $x_j^{\gamma_j-2p}$, $p\leq\lfloor\gamma_j/2\rfloor$. Thus $H_\gamma(x)$ is a linear combination of monomials of the form $\prod_{j=1}^dx_j^{\gamma_j-2p_j}$, which has total order $k - 2\sum_jp_j$. Using the independence of the entries of $Z= (Z_1,\dots, Z_d)$, we have
$$\E[H_\gamma(Z)H_{\gamma'}(Z)] = \gamma!\prod_{j=1}^d\delta_{\gamma_j,\gamma_j'},$$ where $\gamma!:=\prod_{j=1}^d\gamma_j!$.  The $H_\gamma$ can also be defined explicitly as follows:
\beq
e^{-\|x\|^2/2}H_\gamma(x) = (-1)^{|\gamma|}\partial^\gamma\l(e^{-\|x\|^2/2}\r),\eeq where 
$\partial^\gamma f(x) = \partial_{x_1}^{\gamma_1}\dots\partial_{x_d}^{\gamma_d}f(x).$ This leads to the useful Gaussian integration by parts identity,
$$\E[f(Z)H_\gamma(Z)] = \E[\partial^\gamma f(Z)],\quad\text{if }f\in C^{|\gamma|}(\R^d).$$

The Hermite polynomials $H_\gamma$, $\gamma\in\{0,1,\dots\}^d$, form a complete orthogonal basis of the Hilbert space of functions $f:\R^d\to\R$ with inner product $\la f,g\ra = \E[f(Z)g(Z)]$. In particular, if $f:\R^d\to\R$ satisfies $\E[f(Z)^2]<\infty$, then $f$ has the following Hermite expansion:
\beq\label{termwise-series}
f(x) =  \sum_{\gamma\in\{0,1,\dots\}^d}\frac{1}{\gamma!}a_\gamma(f)H_\gamma(x),\qquad a_\gamma(f) := \E[f(Z)H_\gamma(Z)].
\eeq

Let \beq\label{rmp1}r_k(x) = f(x) - \sum_{|\gamma|\leq k-1}\frac{1}{\gamma!}a_\gamma(f)H_\gamma(x)\eeq be the remainder of the Hermite series expansion of $f$ after taking out the order $\leq k-1$ polynomials. We can write $r_k$ as an integral of $f$ against a kernel. Namely,
define \beq\label{kernel-def}K(x, y)= \sum_{|\gamma|\leq k-1}\frac{1}{\gamma!}H_\gamma(x)H_\gamma(y).\eeq Note that $$\E[f(Z)K(x,Z)]=\sum_{|\gamma|\leq k-1}\frac{1}{\gamma!}a_\gamma(f)H_\gamma(x)$$ is the truncated Hermite series expansion of $f$. Therefore, the remainder $r_k$ can be written as $$r_k(x) = f(x) - \E[f(Z)K(x,Z)]  = \E[(f(x)-f(Z))K(x,Z)],$$ using that $\E[K(x,Z)]=1$. 

\subsection{Exact Expression for the Remainder}
\begin{lemma}\label{lma:rk-entrywise}
Let $k\geq1$ and $r_{k}$, $K$, be as in~\eqref{rmp1},~\eqref{kernel-def}, respectively. Assume that $f\in C^{1}$, and that $\|\nabla f(x)\|\lesssim e^{c\|x\|^2}$ for some $0\leq c<1/2$. Then 
\beqs\label{rk-entrywise}
r_{k}(x)&=\E[(f(x)-f(Z))K(x, Z)] \\
&=\int_0^1\sum_{i=1}^d\sum_{|\gamma|=k-1}\frac{1}{\gamma!}\E[\partial_if((1-t)Z+tx)\l(H_{\gamma+e_i}(x)H_\gamma(Z) -H_{\gamma+e_i}(Z)H_\gamma(x)\r)]dt.\eeqs
\end{lemma}
The proof relies on the following identity:
\begin{lemma}\label{lma:Kxy-identity}
For each $i=1,\dots, d$, it holds that 
\beqs\label{Kxy-identity}
K(x, y) = \frac{1}{x_i-y_i}\sum_{|\gamma|=k-1}\frac{1}{\gamma!}\l(H_{\gamma+e_i}(x)H_\gamma(y) -H_{\gamma+e_i}(y)H_\gamma(x)\r).\eeqs
\end{lemma}
The proof of this identity is given at the end of the section. 
\begin{proof}[Proof of Lemma~\ref{lma:rk-entrywise}]
Write
\beqs
f(x) - f(Z) &= \int_0^1(x-Z)^T\nabla f((1-t)Z+tx)dt \\
&= \sum_{i=1}^d\int_0^1(x_i-Z_i)\partial_if((1-t)Z+tx)dt,\eeqs so that, using~\eqref{Kxy-identity}, we have
\beqs\label{EfxZKm-0}\E[(&f(x)-f(Z))K(x, Z)]\\
&= \sum_{i=1}^d\sum_{|\gamma|=k-1}\frac{1}{\gamma!}\E\l[\int_0^1\partial_if((1-t)Z+tx)\l(H_{\gamma+e_i}(x)H_\gamma(Z) -H_{\gamma+e_i}(Z)H_\gamma(x)\r)dt\r].\eeqs By assumption, 
\beqs
\sup_{t\in[0,1]}|\partial_i&f((1-t)Z+tx)|\l(|H_\gamma(Z)| + |H_{\gamma+e_i}(Z)|\r)\\&\lesssim \e\l(c\|Z\|^2 + 2c\|Z\|\|x\|\r)\l(|H_\gamma(Z)| + |H_{\gamma+e_i}(Z)|\r)\eeqs for some $0\leq c<1/2$. The right-hand side is integrable with respect to the Gaussian measure, and therefore we can interchange the expectation and the integral in~\eqref{EfxZKm-0}. Therefore,
\beqs\E[&(f(x)-f(Z))K(x, Z)] \\
&= \int_0^1\sum_{i=1}^d\sum_{|\gamma|=k-1}\frac{1}{\gamma!}\E[\partial_if((1-t)Z+tx)\l(H_{\gamma+e_i}(x)H_\gamma(Z) -H_{\gamma+e_i}(Z)H_\gamma(x)\r)]dt.\eeqs
\end{proof}
\begin{proof}[Proof of Lemma~\ref{lma:Kxy-identity}]
Without loss of generality, assume $i=1$. To simplify the proof, we will also assume $d=2$. The reader can check that the proof goes through in the same way for general $d$. By the recursion relation~\eqref{eq:recur} for 1d Hermite polynomials, we get that 
$$H_{\gamma_1+1,\gamma_2}(x) =x_1H_{\gamma_1,\gamma_2}(x) - \gamma_1H_{\gamma_1-1,\gamma_2}(x),$$ where $x=(x_1,x_2)$. Multiply this equation by $H_\gamma(y)$ (where $y=(y_1,y_2)$) and swap $x$ and $y$, to get the two equations
\beqs\label{two-Hgamma-eqs}
H_{\gamma_1+1,\gamma_2}(x)H_{\gamma_1,\gamma_2}(y) &=x_1H_{\gamma_1,\gamma_2}(x)H_{\gamma_1,\gamma_2}(y) - \gamma_1H_{\gamma_1-1,\gamma_2}(x)H_{\gamma_1,\gamma_2}(y),\\
H_{\gamma_1+1,\gamma_2}(y)H_{\gamma_1,\gamma_2}(x) &=y_1H_{\gamma_1,\gamma_2}(x)H_{\gamma_1,\gamma_2}(y) - \gamma_1H_{\gamma_1-1,\gamma_2}(y)H_{\gamma_1,\gamma_2}(x).
\eeqs 
Let 
$$S_{\gamma_1,\gamma_2} =H_{\gamma_1+1,\gamma_2}(x)H_{\gamma_1,\gamma_2}(y)-H_{\gamma_1+1,\gamma_2}(y)H_{\gamma_1,\gamma_2}(x).$$ Subtracting the second equation of~\eqref{two-Hgamma-eqs} from the first one, and using the $S_{\gamma_1,\gamma_2}$ notation, gives
\beq
S_{\gamma_1,\gamma_2}= (x_1-y_1)H_{\gamma_1,\gamma_2}(x)H_{\gamma_1,\gamma_2}(y) + \gamma_1S_{\gamma_1-1,\gamma_2}\eeq and hence
$$\frac{S_{\gamma_1,\gamma_2}}{\gamma_1!\gamma_2!} = (x_1-y_1)\frac{H_{\gamma_1,\gamma_2}(x)H_{\gamma_1,\gamma_2}(y)}{\gamma_1!\gamma_2} +\frac{S_{\gamma_1-1,\gamma_2}}{(\gamma_1-1)!\gamma_2!}.$$ Iterating this recursive relationship $\gamma_1-1$ times, we get
\beq\label{Sgamma-sum}\frac{S_{\gamma_1,\gamma_2}}{\gamma_1!\gamma_2!} = (x_1-y_1)\sum_{j=0}^{\gamma_1-1}\frac{H_{\gamma_1-j,\gamma_2}(x)H_{\gamma_1-j,\gamma_2}(y)}{(\gamma_1-j)!\gamma_2!} +\frac{S_{0,\gamma_2}}{0!\gamma_2!}.\eeq
Now, we have
\beqs
S_{0,\gamma_2} &= H_{1, \gamma_2}(x)H_{0,\gamma_2}(y)-H_{1,\gamma_2}(y)H_{0,\gamma_2}(x)\\
&= H_1(x_1)H_{\gamma_2}(x_2)H_{\gamma_2}(y_2) - H_{1}(y_1)H_{\gamma_2}(y_2)H_{\gamma_2}(x_2)\\
&= (x_1-y_1)H_{\gamma_2}(x_2)H_{\gamma_2}(y_2)\\
&=(x_1-y_1)H_{0,\gamma_2}(x)H_{0,\gamma_2}(y).
\eeqs
Therefore, 
$$\frac{S_{0,\gamma_2}}{0!\gamma_2!} =  (x_1-y_1)\frac{H_{\gamma_1-j,\gamma_2}(x)H_{\gamma_1-j,\gamma_2}(y)}{(\gamma_1-j)!\gamma_2!},\qquad j=\gamma_1$$ so~\eqref{Sgamma-sum} can be written as
$$\frac{S_{\gamma_1,\gamma_2}}{\gamma_1!\gamma_2!} = (x_1-y_1)\sum_{j=0}^{\gamma_1}\frac{H_{\gamma_1-j,\gamma_2}(x)H_{\gamma_1-j,\gamma_2}(y)}{(\gamma_1-j)!\gamma_2!} $$ and hence
\beqs
\frac{1}{x_1-y_1}\sum_{\gamma_1+\gamma_2= k-1}\frac{S_{\gamma_1,\gamma_2}}{\gamma_1!\gamma_2!}  &= \sum_{\gamma_1+\gamma_2=k-1}\sum_{j=0}^{\gamma_1}\frac{H_{\gamma_1-j,\gamma_2}(x)H_{\gamma_1-j,\gamma_2}(y)}{(\gamma_1-j)!\gamma_2!}\\
&=\sum_{\gamma_1+\gamma_2\leq k-1}\frac{H_{\gamma_1,\gamma_2}(x)H_{\gamma_1,\gamma_2}(y)}{\gamma_1!\gamma_2!} = K(x, y),
\eeqs
using the observation that
\beqs
\{(\gamma_1-j,\gamma_2)\; : &\;\gamma_1+\gamma_2=k-1, \; 0\leq j\leq\gamma_1\} \\
= \{(\tilde\gamma_1, \gamma_2)\; : &\;\tilde\gamma_1+\gamma_2\leq k-1\}.\eeqs
Substituting back in the definition of $S_{\gamma_1,\gamma_2}$ gives the desired result.
\end{proof}

\subsection{Hermite Series Remainder in Tensor Form}\label{app:tensor-hermite}
Using~\eqref{rk-entrywise}, it is difficult to obtain an upper bound on $|r_k(x)|$, since we need to sum over all $\gamma$ of order $k-1$. In this section, we obtain a more compact representation of $r_k$ in terms of a scalar product of $k$-tensors. We then take advantage of a very useful representation of the tensor of order-$k$ Hermite polynomials, as an expectation of a vector outer product. This allows us to bound the scalar product in the $r_k$ formula in terms of an operator norm rather than a Frobenius norm (the latter would incur larger $d$ dependence).  

 First, let us put all the unique $k$th order Hermite polynomials into a tensor of $d^k$ entries, some of which are repeating, enumerated by multi-indices $\alpha = (\alpha_1,\dots,\alpha_k)\in [d]^k$. Here, $[d]=\{1,\dots,d\}$. We do so as follows: given $\alpha\in[d]^k$, define $\gamma(\alpha) = (\gamma_1(\alpha),\dots, \gamma_d(\alpha))$ by $$\gamma_j(\alpha) = \sum_{\ell=1}^k\mathbbm 1\{\alpha_\ell = j\},$$ i.e. $\gamma_j(\alpha)$ counts how many times index $j$ appears in $\alpha$. For this reason, we use the term \emph{counting index} to denote indices of the form $\gamma = (\gamma_1,\dots, \gamma_d) \in \{0,1,2,\dots\}^d$, whereas we use the standard term ``multi-index" to refer to the $\alpha$'s. Note that we automatically have $|\gamma(\alpha)|=k$ if $\alpha\in[d]^k$. Now, for $x\in\R^d$, define $\H_0(x)=1$ and $\H_k(x)$, $k\geq1$ as the tensor $$\H_k(x) = \{H_{\gamma(\alpha)}(x)\}_{\alpha\in[d]^k},\quad x\in\R^d.$$ When enumerating the entries of $\H_k$, we write $H_k^{(\alpha)}$ to denote $H_{\gamma(\alpha)}$. Note that for each $\gamma$ with $|\gamma|=k$, there are ${k\choose\gamma}$ $\alpha$'s such that $\gamma(\alpha)=\gamma$. 
\begin{example}
Consider the $\alpha=(i, j, j, k, k, k)$ entry of the tensor $\H_6(x)$, where $i,j,k\in[d]$ are all distinct. We count that $i$ occurs once, $j$ occurs twice, and $k$ occurs thrice. 
Thus
$$H_6^{(i,j,j, k, k, k)}(x) = H_1(x_i)H_2(x_j)H_3(x_k) = x_i(x_j^2-1)(x_k^3-3x_k).$$ 
The first two tensors $\H_1, \H_2$ can be written down explicitly. For the entries of $\H_1$, we simply have $H_1^{(i)}(x) = H_1(x_i) = x_i$, i.e. $\H_1(x) = x$. For the entries of $\H_2$, we have $H_2^{(i,i)}(x) = H_2(x_i) = x_i^2-1$ and $H_2^{(i,j)}(x) = H_1(x_i)H_1(x_j) = x_ix_j$, $i\neq j$. Thus $\H_2(x)=xx^T-I_d$.
\end{example}

We now group the terms in the Hermite series expansion~\eqref{termwise-series} based on the order $|\gamma|$. Consider all $\gamma$ in the sum such that $|\gamma|=k$. We claim that
\beqs\label{gam-2-alph}
\sum_{|\gamma|=k}\frac{1}{\gamma!}a_\gamma(f)H_\gamma(x) = \frac{1}{k!}\sum_{\alpha\in[d]^k}a_{\gamma(\alpha)}(f)H_k^{(\alpha)}(x).
\eeqs
Indeed, for a fixed $\gamma$ such that $|\gamma|=k$, there are ${k\choose\gamma}$ $\alpha$'s in $[d]^k$ for which $\gamma(\alpha)= \gamma$, and the summands in the right-hand sum corresponding to these $\alpha$'s are all identical, equalling $a_\gamma(f)H_\gamma(x)$. Thus we obtain ${k\choose\gamma}$ copies of $a_\gamma(f)H_\gamma(x)$, and it remains to note that ${k\choose\gamma}/k! = 1/\gamma!$. 

Analogously to $\H_k(x)$, define the tensor $\A_k\in(\R^d)^{\otimes k}$, whose $\alpha$'th entry is 
$$\A_k^{(\alpha)} = a_{\gamma(\alpha)} = \E[f(Z)H_{\gamma(\alpha)}(Z)] = \E[f(Z)H_k^{(\alpha)}(Z)].$$We then see that the sum~\eqref{gam-2-alph} can be written as $\frac{1}{k!}\la \A_k, \;\H_k(x)\ra$, and hence the series expansion of $f$ can be written as
\beq
f(x) = \sum_{k=0}^\infty\frac{1}{k!}\la \A_k(f), \;\H_k(x)\ra,\qquad \A_k(f) := \E[f(Z)\H_k(Z)].
\eeq

\begin{lemma}\label{lma:rk-exact}
Let $f$ satisfy the assumptions of Lemma~\ref{lma:rk-entrywise}, and additionally, assume $f\in C^k$. Then the remainder $r_k$, given as~\eqref{rk-entrywise} in Lemma~\ref{lma:rk-entrywise}, can also be written in the form
\beqs\label{rk-exact-2}
r_{k}(x)=\int_0^1\frac{(1-t)^{k-1}}{(k-1)!}\E\l[\lla\nabla^{k}f\l((1-t)Z+tx\r),\; \H_{k}(x)-Z\otimes\H_{k-1}(x)\rra\r]
\eeqs\end{lemma}
\begin{proof}
Recall that $\partial^\gamma := \partial_1^{\gamma_1}\dots\partial_d^{\gamma_d}$, and that 
$$H_\gamma(z)e^{-\|z\|^2/2} = (-1)^{|\gamma|}\partial^\gamma(e^{-\|z\|^2/2}).$$ We then have for $|\gamma|=k-1$, 
\beqs
\E[\partial_if((1-t)Z+tx)H_\gamma(Z)]&= (1-t)^{k-1}\E[\partial^{\gamma+e_i}f],\\
\E[\partial_if((1-t)Z+tx)H_{\gamma+e_i}(Z)]&= (1-t)^{k-1}\E[\partial^{\gamma+e_i}fZ_i],\eeqs using the fact that $f\in C^k$. 
%$H_\gamma(Z) = \partial^\gamma(e^{-\|Z\|^2/2})$ We have used the fact that 
We omitted the argument $(1-t)Z+tx$ from the right-hand side for brevity. To get the second equation, we moved only $\gamma$ of the $\gamma+e_i$ derivatives from $e^{-\|z\|^2/2}$ onto $\partial_if$, leaving $-\partial_i(e^{-\|z\|^2/2})=z_i$. Substituting these two equations into~\eqref{rk-entrywise}, we get
\beqs\label{EfxZKm-3}
\E[&(f(x)-f(Z))K(x, Z)] \\
&=\int_0^1(1-t)^{k-1} \sum_{i=1}^d\sum_{|\gamma|=k-1}\frac{1}{\gamma!}\E[(\partial^{\gamma+e_i}f)\l(H_{\gamma+e_i}(x) -Z_iH_\gamma(x)\r)]dt\\
&= \frac{1}{(k-1)!}\int_0^1(1-t)^{k-1}\sum_{i=1}^d\sum_{|\gamma|=k-1}{k-1\choose\gamma}\E[(\partial^{\gamma+e_i}f)\l(H_{\gamma+e_i}(x) -Z_iH_\gamma(x)\r)]dt\eeqs 
Now, define the sets
\beqs
A&=\{(i, \gamma+e_i)\; :\; i=1,\dots, d,\;\gamma\in\{0,1,\dots\}^d,\;|\gamma|=k-1\} ,\\
B&= \{(i,\tilde\gamma)\; :\; \tilde\gamma\in\{0,1,\dots\}^d,\;|\tilde\gamma|=k,\;\tilde\gamma_i\geq1\}.\eeqs
It is straightforward to see that $A=B$. Therefore,
\beqs\label{EfxZKm-term1}
\sum_{i=1}^d\sum_{|\gamma|=k-1}&{k-1\choose\gamma}\E[\partial^{\gamma+e_i}f]H_{\gamma+e_i}(x) \\
&= \sum_{|\tilde\gamma|=k}\sum_{i: \,\tilde\gamma_i\geq1}{k-1\choose{\tilde\gamma-e_i}}\E[\partial^{\tilde\gamma}f]H_{\tilde\gamma}(x) \\
&= \sum_{|\tilde\gamma|=k}\sum_{i: \,\tilde\gamma_i\geq1}{k\choose\tilde\gamma}\frac{\tilde\gamma_i}{k} \E[\partial^{\tilde\gamma}f]H_{\tilde\gamma}(x)\\
&=\sum_{|\tilde\gamma|=k}{k\choose\tilde\gamma}\E[\partial^{\tilde\gamma}f]H_{\tilde\gamma}(x) = \la \E[\nabla^kf], \, \H_k(x)\ra
\eeqs
Next, note that 
$$\sum_{|\gamma|=k-1}{k-1\choose\gamma}\partial^{\gamma}\partial_ifH_\gamma(x) = \la \nabla^{k-1}\partial_if,\, \H_{k-1}(x)\ra,$$ and therefore
\beq\label{EfxZKm-term2}\sum_{i=1}^d\sum_{|\gamma|=k-1}{k-1\choose\gamma}\E[(\partial^{\gamma+e_i}f)Z_i]H_\gamma(x) = \E[\la\nabla^kf, Z\otimes\H_{k-1}(x)\ra].\eeq Substituting~\eqref{EfxZKm-term1} and~\eqref{EfxZKm-term2} into~\eqref{EfxZKm-3} gives
\beqs
r_k(x)&=\E[(f(x)-f(Z))K(x, Z)] \\
&=\int_0^1\frac{(1-t)^{k-1}}{(k-1)!}\E\l[\lla\nabla^kf\l((1-t)Z+tx\r),\; \H_k(x)-Z\otimes\H_{k-1}(x)\rra\r]
\eeqs
\end{proof}
In the next section, we obtain a pointwise upper bound on $|r_k(x)|$ in the case $f=\Wbar $. In order for this bound to be tight in its dependence on $d$, we need a supplementary result on inner products with Hermite tensors. 
To motivate this supplementary result, consider bounding the inner product in~\eqref{rk-exact-2} by the product of the Frobenius norms of the tensors on either side. As a rough heuristic, $\|\nabla^kf\|_F\sim d^{k/2}\|\nabla^kf\|$, where recall that $\|\nabla^kf\|$ is the \emph{operator} norm of $\nabla^kf$. Therefore, we would prefer to bound the inner product in terms of $\|\nabla^kf\|$ to get a tighter dependence on $d$. Apriori, however, this seems impossible, since $\H_k(x)$ is not given by an outer product of $k$ vectors. But the following representation of the order $k$ Hermite polynomials will make this possible. 
\beq\label{eH}
\H_k(x)=\E[(x+iZ)^{\otimes k}],\eeq  where $Z\sim\mathcal N(0, I_d)$.
Using~\eqref{eH}, we can bound scalar products of the form $\la\nabla^kf, \H_k(x)\ra$ and $\la\nabla^kf, Z\otimes\H_{k-1}(x)\ra$ in terms of the operator norm of $\nabla^kf$. More generally, we have the following lemma.
 \begin{lemma}\label{tensor-H-opnorm}
Let $T\in(\R^d)^{\otimes k}$ be a $k$-tensor, and $v\in\R^d$. Then for all $0\leq \ell\leq k$, we have
$$|\la T, \; v^{\otimes \ell}\otimes\H_{k-\ell}(x)\ra| \les_k \|T\|\|v\|^\ell(\|x\|^{k-\ell}+d^{\frac{k-\ell}2}).$$
\end{lemma}
\begin{proof}Using~\eqref{eH}, we have 
$$\la T, \; v^{\otimes \ell}\otimes\H_{k-\ell}(x)\ra =\E[\la T, \; v^{\otimes \ell}\otimes (x+iZ)^{\otimes (k-\ell)}\ra]$$ and hence
\beqs
|\la T, \; v^{\otimes \ell}\otimes\H_{k-\ell}(x)\ra| &\leq \E|\la T, \;  v^{\otimes \ell}\otimes (x+iZ)^{\otimes k-\ell}\ra| \\&\leq\|T\|\|v\|^\ell\E\l[\|x+iZ\|^{k-\ell}\r]\\
&\les_k \|T\|\|v\|^\ell(\|x\|^{k-\ell}+\sqrt d^{k-\ell}).
\eeqs
\end{proof}

\subsection{Hermite Remainder Formula Applied to $\Wbar $}\label{app:subsec:Erkm}
We now consider the case $f=\Wbar$, where $\Wbar$ is defined in~\eqref{barWviaW} and satisfies the properties stated in Lemma~\ref{lma:WtobarW}. Recall that
\beq\label{rk-def-again}r_k(x) = \Wbar (x) - \sum_{j=0}^{k-1}\frac{1}{j!}\lla \A_j, \H_j(x)\rra,\eeq where $\A_j=\E[\Wbar (Z)\H_j(Z)]=\E[\nabla^j\Wbar(Z)]$, and the second equality holds provided $\Wbar\in C^j$. We first prove a pointwise bound on $r_4$, and then bound the $L^p(\gamma)$ norm of $r_3$ and $r_4$.
\begin{lemma}\label{lma:r4ptwise}
It holds
$$
|r_4(x)|\les_q\frac{a_4}{n}\l(\|x\|^4 + d^2 + \sqrt d\|x\|^3\r)(1+\|x/\sqrt d\|^q).$$
\end{lemma}
\begin{proof}
Let $\nabla^{k}\Wbar $ be shorthand for $\nabla^{k} \Wbar ((1-t)Z+tx)$. Using~\eqref{rk-exact-2} for $f= \Wbar$, we have
\beqs
\label{rk-bd-2}
|r_4(x)| &\leq \int_0^1\E\l[\l|\lla \nabla^4 \Wbar , \,\; \H_4(x)-Z\otimes\H_3(x) \rra \r|\r]dt\\
&\les\int_0^1\E\l[\|\nabla^4\Wbar\|\l(\|x\|^4+d^2 + \|Z\|(\|x\|^3+d^{3/2})\r)\r]dt\\
&\les \l(\|x\|^4 + d^2 + \sqrt d\|x\|^3\r)\int_0^1\sqrt{\E\l[\|\nabla^4\Wbar((1-t)Z+tx)\|^2\r]}dt\\
&= \l(\|x\|^4 + d^2 + \sqrt d\|x\|^3\r)\int_0^1\l \| \|\nabla^4\Wbar\| \r\|_{L^2(\mathcal N(x, (1-t)^2I_d))}dt\\
&\les_q \frac{a_4}{n}\l(\|x\|^4 + d^2 + \sqrt d\|x\|^3\r)(1+\|x/\sqrt d\|^q)
\eeqs In the second line we used Lemma~\ref{tensor-H-opnorm}, and in the third line we used Cauchy-Schwarz and the fact that $\E[\|Z\|^k]\les_k d^{k/2}$. To get the fifth line, we used~\eqref{barW4} of Lemma~\ref{lma:WtobarW}
\end{proof}
The $d$-dependence in the bound on $\|r_3\|_p$ in the next lemma relies on Lemmas 2.1, C.2 and C.1 of~\cite{katsTDD}. Combined, the three lemmas give that for a symmetric $d\times d\times d$ tensor $T$, it holds
\beq\label{TH3}
\|\la T, \H_3\ra\|_{2k} \les_k \|\la T, \H_3\ra\|_{2} \les \|T\|_F \les d\|T\|
\eeq
\begin{lemma}\label{rUUc}Let $p_3(x)=\frac{1}{3!}\la \A_3,\H_3\ra(x)$. It holds
\beqsn
\|r_4\|_p &\les_{p,q}a_4\epsilon^2\\
\|p_3\|_p&\les_{p,q} a_3\epsilon,\\
\|r_3\|_p &\les_{p,q} a_3\epsilon + a_4\epsilon^2=:\omega,
\eeqsn  where $\epsilon=d/\sqrt n$.\end{lemma}
\begin{proof}
For the bound on $\|r_4\|_p$, we use Lemma~\ref{lma:r4ptwise} which gives
$$
\E\l[|r_4(Z)|^p\r]\les_{p,q}\l(\frac{a_4}{n}\r)^p\E\l[\l(\|Z\|^4 + d^2 + \sqrt d\|Z\|^3\r)^p(1+\|Z/\sqrt d\|^q)^p\r] \les_{p,q}\l(\frac{a_4}{n}\r)^pd^{2p}.
$$Taking the $p$th root gives $\|r_4\|_p\les_{p,q}a_4d^2/n=a_4\epsilon^2$. To bound $\|p_3\|_p$, we use~\eqref{TH3} and~\eqref{barW3spec} of Lemma~\ref{lma:WtobarW}:
\beq
\|p_3\|_p\leq \|\la \A_3, \H_3\ra\|_p \les_p d\|\A_3\| \leq d\E[\|\nabla^3\Wbar (Z)\|]\les_q  a_3\epsilon.\eeq Finally, note that $r_3=p_3+r_4$, and the bound on $\|r_3\|_p$ follows.
\end{proof}

\section{Proofs from Section~\ref{sec:m-sig-exist}: existence and uniqueness of solutions to stationarity conditions}\label{app:sec:m-sig-exist}
For the proof of Lemma~\ref{lma:intro:exist:V} using Lemma~\ref{lma:m-sig-soln}, see Section~\ref{supp:aux}. In this section, we prove Lemma~\ref{lma:m-sig-soln} via the auxiliary lemmas outlined in Section~\ref{sec:proof-exist}. The proofs rely on tensor-matrix and tensor-vector scalar products. Let us review the rules of such scalar products, and how to bound the operator norms of these quantities. Let $v\in\R^d$, $A\in\R^{d\times d}$, and $T\in\R^{d\times d\times d}$. We define the vector $\la T, A\ra\in\R^d$ and the matrix $\la T, v\ra\in\R^{d\times d}$ by
\beqs
\la T, A\ra_i &= \sum_{j,k=1}^d T_{ijk}A_{jk}, \quad i=1,\dots, d,\\
\la T, v\ra_{ij} &= \sum_{k=1}^dT_{ijk}v_k,  \quad i,j=1,\dots, d.\eeqs We will always sum over the last two or last one indices of the tensor. 
Note that the norm of the matrix $\la T, v\ra$ is given by
$\|\la T, v\ra\| = \sup_{\|u\|=\|w\|=1}u^T\la T, v\ra w$, and we have
$$u^T\la T, v\ra w = \sum_{i,j=1}^du_iw_j\sum_{k=1}^dT_{ijk}v_k= \la T, u\otimes w \otimes v\ra\leq \|T\|\|v\|.$$ Therefore, $\|\la T, v\ra\| \leq \|T\|\| v\|$.

We also review the notion of operator norm for derivatives of a function, and note the distinction between this kind of operator norm and the standard tensor operator norm.  Specifically, consider a $C^2$ function $f=(f_1,\dots, f_d):\R^d\times\R^{d\times d}\to\R^d$, where $\R^{d\times d}$ is endowed with the standard matrix norm. %, and $\R^d\times\R^{d\times d}$ is endowed with the norm$$\|(m, \sigma)\|^2 = \|\sigma\|^2 + \|m\|^2,\qquad\sigma\in\R^{d\times d},\; m\in\R^d.$$ 
 Then $\nabla_\sigma f(m, \sigma)$ is a linear functional from $\R^{d\times d}$ to $\R^d$, and we let $\la\nabla_\sigma f(m, \sigma),\; A\ra\in\R^d$ denote the application of $\nabla_\sigma f(m, \sigma)$ to $A$. Note that we can represent $\nabla_\sigma f$ by the $d\times d\times d$ tensor $(\nabla_{\sigma_{jk}}f_i)_{i,j,k=1}^d$, so that $\la\nabla_\sigma f(m, \sigma),\; A\ra$ coincides with the definition given above of tensor-matrix scalar products. However, $\|\nabla_\sigma f\|_\op$ is \emph{not} the standard tensor operator norm. Rather, %$\|\nabla_\sigma f\|_\op^2$ is the operator norm obtained by considering $\nabla_\sigma f$ as a linear functional from $\R^{d\times d}$ to $\R^d$. Namely,
$$\|\nabla_\sigma f\|_\op = \sup_{A\in\R^{d\times d}, \|A\|=1}\|\la \nabla_\sigma f, A\ra\| = \sup_{\substack{A\in\R^{d\times d},\|A\|=1,\\ u\in\R^d, \|u\|=1}}\la \nabla_\sigma f, A\otimes u\ra.$$ We continue to write $\|\nabla_\sigma f\|$ to denote the \emph{standard} tensor operator norm, i.e.
$$\|\nabla_\sigma f\| = \sup_{\substack{u,v,w\in\R^{d},\\\|u\|=\|v\|=\|w\|=1}}\la \nabla_\sigma f, u\otimes v\otimes w\ra.$$
Note also that $\nabla_mf\in\R^{d\times d}$ is a matrix, and that
\beqs
\max\bigg(\l\|\nabla_\sigma f(m, \sigma)\r\|_\op, \; &\l\|\nabla_m f(m, \sigma)\r\|_\op\bigg) \\
\leq \|\nabla f&(m, \sigma)\|_\op \leq\|\nabla_\sigma f(m, \sigma)\|_\op+\|\nabla_m f(m, \sigma)\|_\op.\eeqs
Finally, recall the notation
\beqs
B_r(0,0) &= \{(m,\sigma)\in \R^d\times\R^{d\times d} \; : \; \|\sigma\|^2+\|m\|^2\leq r^2\},\\
B_r &= \{\sigma\in\R^{d\times d}\; : \; \|\sigma\|\leq r\},\\
S_{c_1, c_2} &= \{\sigma\in\S^{d}_{+}\; : \; c_1I_d\preceq \sigma\preceq c_2I_d\}.\eeqs
The proof of Lemma~\ref{IVT} uses the following lemma.
\begin{lemma}[Lemma 1.3 in Chapter XIV of~\cite{langanalysis}]\label{lma:lang} Let $U$ be open in a Banach space $E$, and let $f:U\to E$ be of class $C^1$. Assume that $f(0)=0$ and $f'(0)=I$. Let $r>0$ be such that $\bar B_r(0)\subset U$. If 
$$|f'(z)-f'(x)|\leq s,\qquad\forall z,x\in\bar B_r(0)$$ for some $s\in(0,1)$, then $f$ maps $\bar B_r(0)$ bijectively onto $\bar B_{(1-s)r}(0)$.\end{lemma}
\begin{proof}[Proof of Lemma~\ref{IVT}]
Let $\phi:\R^d\times\R^{d\times d}\to \R^d\times\R^{d\times d}$ be given by $\phi(m,\sigma) = (f(m,\sigma),\sigma)$, so that $\phi(0, 0) = (0,0)$, and 
\beqs
\nabla\phi(m, \sigma) = \begin{pmatrix} \nabla_mf(m,\sigma) & \nabla_\sigma f(m, \sigma)\\ 0 & I_d\end{pmatrix}.
\eeqs For each $(m, \sigma), (m',\sigma')\in B_r(0,0)$, we have
\beqs
\|\nabla&\phi(m,\sigma) - \nabla\phi(m',\sigma')\|_\op = \|\nabla f(m,\sigma) - \nabla f(m',\sigma')\|_\op\\
&\leq 2\sup_{(m, \sigma)\in B_r(0,0)}\|\nabla f(m,\sigma) - \nabla f(0,0)\|_\op \leq  \frac12. \eeqs 
Note also that $\nabla\phi(0,0)$ is the identity. Thus by Lemma~\ref{lma:lang}, we have that $\phi$ is a bijection from $B_r(0, 0)$ to $B_{r/2}(\phi(0,0)) = B_{r/2}(0, 0)$. Now, fix any $\sigma\in \R^{d\times d}$ such that $\|\sigma\|\leq r/2$. Then $(0,\sigma)\in B_{r/2}(0,0)$, and hence there exists a unique $(m,\sigma')\in B_r(0, 0)$ such that $(0,\sigma) = \phi(m,\sigma') = (f(m,\sigma'),\sigma')$. Thus $\sigma=\sigma'$ and $f(m, \sigma)=0$. In other words, for each $\sigma$ such that $\|\sigma\|\leq r/2$ there exists a unique $m=m(\sigma)$ such that $(m(\sigma), \sigma)\in B_r(0, 0)$ and such that $0=f(m, \sigma)$. 

The map $\sigma\mapsto m(\sigma)$ is $C^2$ by standard Implicit Function Theorem arguments. To show that the first inequality of~\eqref{IVT-props} holds, note that we have $$\|\nabla_mf(m(\sigma), \sigma) - \nabla_mf(0, 0)\|_\op \leq \|\nabla f(m(\sigma), \sigma) - \nabla f(0, 0)\|_\op \leq 1/4 \leq 1/2$$ by~\eqref{supsigmam} since we know that $(m(\sigma), \sigma) \in B_r(0, 0)$. Thus, \beqs
 I_d= \nabla^2 W(0)&=\nabla_mf(0,0)\\&\implies\frac{1}2I_d\preceq\nabla_mf(m(\sigma), \sigma)\preceq \frac32I_d.\eeqs 

To show the second inequality of~\eqref{IVT-props}, we first need the supplementary bound 
\beq\label{app:supp}\|\nabla_\sigma f(m(\sigma), \sigma)\|_\op=\|\nabla_\sigma f(m(\sigma), \sigma) - \nabla_\sigma f(0, 0)\|_\op \leq 1/2\eeq which holds by the same reasoning as above. Now, $$\partial_{\sigma_{jk}}m = -\nabla_mf(m, \sigma)^{-1}\partial_{\sigma_{jk}}f(m, \sigma)\in\R^d$$ by standard Implicit Function Theorem arguments, where $\nabla_mf(m, \sigma)$ is a matrix, $\partial_{\sigma_{jk}}f(m, \sigma)$ is a vector, and $\nabla_\sigma m$, $\nabla_\sigma f$ are linear maps from $\R^{d\times d}$ to $\R^d$. Hence by the first inequality in~\eqref{IVT-props} combined with~\eqref{app:supp} we have
\beqs
\|\nabla_\sigma m(\sigma)\|_\op &=\sup_{\|A\|=1}\|\la\nabla_\sigma m(\sigma), A\ra\|\\
& = \sup_{\|A\|=1}\|\nabla_mf(m, \sigma)^{-1}\sum_{j,k=1}^d\partial_{\sigma_{jk}}f(m, \sigma)A_{jk}\|\\
& = \sup_{\|A\|=1}\|\nabla_mf(m, \sigma)^{-1}\la\nabla_\sigma f, A\ra\|\\
& \leq \|\nabla_mf(m, \sigma)^{-1}\|\|\nabla_\sigma f\|_\op \leq 2\times\frac12 = 1. \eeqs
\end{proof}
\begin{proof}[Proof of Lemma~\ref{m-of-sig}]
Note that $f$ is $C^2$ thanks to the fact that $W$ is $C^3$ and $\nabla W$ grows polynomially by Assumption~\ref{assume:glob} (since the third derivative tensor grows polynomially). We then immediately have $f(0,0) = \nabla W(0) = 0$, $\nabla_mf(m, \sigma) =\E[\nabla^2 W(m+\sigma Z)]$ is symmetric for all $m,\sigma$, and $\nabla_mf(0, 0)= \nabla^2W(0)=I_d$. To show $\nabla_\sigma f(0, 0) = 0$, we compute the $i,j,k$ term of this tensor:
$$\partial_{\sigma_{jk}}f_i = \partial_{\sigma_{jk}}\E[\partial_iW(m+\sigma Z)] = \E[\partial^2_{ij}W(m+\sigma Z)Z_k],$$ so that
$\partial_{\sigma_{jk}}f_i(0, 0) = \E[\partial^2_{i,j}W(0)Z_k] = 0$. It remains to show that for $r=2\sqrt2$ we have
\beq\label{sup-sigma}\sup_{(m,\sigma)\in B_r(0, 0)}\|\nabla f(m,\sigma) - \nabla f(0,0)\|_\op \leq\frac{1}{4}.\eeq First, note that \beq\label{app:supsigmam}\sup_{(m,\sigma)\in B_r(0,0)}\|\nabla f(m,\sigma) - \nabla f(0,0)\|_\op  \leq r\sup_{(m,\sigma)\in B_r(0,0)}\|\nabla^2f(m,\sigma)\|_\op,\eeq where $\nabla^2f(m, \sigma)$ is a bilinear form on $(\R^d\times\R^{d\times d})^2$, and we have
\beq\label{nabla2f}\|\nabla^2f(m,\sigma)\|_\op \leq \|\nabla_\sigma^2f(m, \sigma)\|_\op + 2\|\nabla_{\sigma}\nabla_mf(m, \sigma)\|_\op + \|\nabla_m^2f(m, \sigma)\|_\op.\eeq
For $f(m, \sigma) = \E[\nabla W(\sigma Z+m)]$, these second order derivatives are given by
\beqs
\partial_{m_i,m_j}^2f(m, \sigma) &= \E[\partial_{i,j}^2\nabla W(m+\sigma Z)],\\
 \partial_{m_i}\partial_{\sigma_{jk}} f(m, \sigma) &= \E[\partial^2_{i,j}\nabla W(m+\sigma Z)Z_k],\\
\partial^2_{\sigma_{jk},\sigma_{\ell p}}f(m, \sigma) &= \E[\partial^2_{j,\ell}\nabla W(m+\sigma Z)Z_kZ_p],
\eeqs each a vector in $\R^d$.
From the first line, we get that 
\beq\label{bd1IFT}
\|\nabla_m^2f(m, \sigma)\|_\op \leq \E\|\nabla^3W(m+\sigma Z)\|=\l\|\|\nabla^3W\|\r\|_{L^1(\mathcal N(m,\sigma\sigma^T))}\eeq where $\|\nabla^3W\|$ is the standard tensor norm. From the second line, we get
\beqs\label{bd2IFT}
\|\nabla_m&\nabla_\sigma f(m, \sigma)\|_\op \\
&= \sup_{\|A\|=1, \|x\|=1}\bigg \|\E\bigg[\sum_{i,j,k=1}^d\partial^2_{i,j}\nabla W(m+\sigma Z)Z_kx_iA_{jk}\bigg]\bigg\| \\
&= \sup_{\|A\|=1, \|x\|=1}\bigg\|\E\bigg[\sum_{i,j=1}^d\partial_{i,j}^2\nabla W(m+\sigma Z)x_i(AZ)_j\bigg]\bigg\|\\
&= \sup_{\|A\|=1, \|x\|=1}\bigg \|\E\bigg[\bigg\la \nabla^3 W(m+\sigma Z), x\otimes AZ\bigg\ra\bigg]\bigg\|\\
& \leq \sup_{\|A\|=1, \|x\|=1}\E\bigg[\|x\|\|AZ\| \|\nabla^3 W(m+\sigma Z)\|\bigg]\\
& \leq  \sqrt d\l\|\|\nabla^3W\|\r\|_{L^2(\mathcal N(m,\sigma\sigma^T))}
%\sqrt{\E[\|\nabla^3 W(m+\sigma Z)\|^2]}.
\eeqs
A similar computation gives
\beqs\label{bd3IFT}
\|\nabla^2_\sigma f(m, \sigma)\|_\op &\leq \sup_{\|A\|=1, \|B\|=1}\E[\|AZ\|\|BZ\| \|\nabla^3 W(m+\sigma Z)\|]\\
& \leq 2d\l\|\|\nabla^3W\|\r\|_{L^2(\mathcal N(m,\sigma\sigma^T))}
%\sqrt{\E[\|\nabla^3 W(m+\sigma Z)\|^2]}.
\eeqs
Using~\eqref{bd1IFT},~\eqref{bd2IFT}, and~\eqref{bd3IFT} in~\eqref{nabla2f} gives 
\beqs\label{nabla2fs}
\|\nabla^2 f(m, \sigma)\|_\op &\leq (2d+2\sqrt d + 1)\l\|\|\nabla^3W\|\r\|_{L^2(\mathcal N(m,\sigma\sigma^T))} \\
&\leq 5d\sup_{(m,\sigma)\in B_r(0,0)}\l\|\|\nabla^3W\|\r\|_{L^2(\mathcal N(m,\sigma\sigma^T))}.\eeqs Substituting~\eqref{nabla2fs} into~\eqref{app:supsigmam} with $r=2\sqrt2$ and then using~\eqref{Enabla3W} of Lemma~\ref{lma:Wprops}, we get
\beqs\label{rd}
\sup_{(m,\sigma)\in B_{2\sqrt2}(0,0)}&\|\nabla f(m,\sigma) - \nabla f(0,0)\|_\op \\& \leq 10\sqrt 2d\sup_{(m,\sigma)\in B_{2\sqrt2}(0,0)}\l\|\|\nabla^3W\|\r\|_{L^2(\mathcal N(m,\sigma\sigma^T))}\\
&\leq C(q)a_3\frac{d}{\sqrt n} \leq \frac14.
\eeqs The last line holds provided $a_3d/\sqrt n \leq 1/4C(q)$, which is ensured by~\eqref{cconds}. \end{proof}
\begin{proof}[Proof of Lemma~\ref{lma:contract}]
First, let $G(\sigma) = \E[\nabla^2W(m(\sigma) + \sigma Z)]$ 
and $f(m, \sigma) = \E[\nabla W(m+\sigma Z)]$ as in Lemma~\ref{m-of-sig}. Note that $\nabla_mf(m, \sigma) = \E[\nabla^2W(\sigma Z+m)]$, so that $G(\sigma) = \nabla_mf(m, \sigma)\vert_{m=m(\sigma)}$ and hence by~\eqref{IVT-props} of Lemma~\ref{IVT} we have \beq\label{G-range}\frac{1}{2}I_d\preceq G(\sigma)\preceq\frac{3}{2}I_d,\qquad \forall \sigma\in S_{0,r/2}.\eeq But then $G(\sigma)$ has a unique invertible symmetric positive definite square root, and we define $F(\sigma) = G(\sigma)^{-1/2}$ to be the inverse of this square root. Moreover, using~\eqref{G-range}, it follows that 
$$c_1I_d \preceq F(\sigma) \preceq c_2I_d, \qquad \forall \sigma\in S_{0, r/2},$$ where $c_1=\sqrt{2/3}$ and $c_2=\sqrt2=r/2.$ In other words, $F(S_{0,r/2})\subseteq S_{c_1,c_2}\subseteq S_{0,r/2}$. It remains to show $F$ is a contraction on $S_{0,r/2}$. Let $\sigma_1,\sigma_2\in S_{0, \sqrt2}$. We will first bound $\|G(\sigma_1)-G(\sigma_2)\|$. We have
\beqs\label{Gsig-diff}
\|G(\sigma_1)-G(\sigma_2)\|\leq \|\sigma_1 - \sigma_2\|\sup_{\sigma\in S_{0,\sqrt2}}\|\nabla_\sigma G(\sigma)\|_\op,\eeqs and
\beqs
\|\nabla_\sigma G(\sigma)\|_\op&=\sup_{\|A\|=1}\big\|\big\la\nabla_\sigma G(\sigma), A\big\ra\big\|\\&= \sup_{\|A\|=1}\bigg\|\E\bigg\la\nabla^3 W, \;\big\la A, \;\nabla_\sigma\l(m(\sigma)+\sigma Z\r)\big\ra\bigg\ra\bigg\|.
%\|\E[\la \nabla^3v(m(\sigma)+\sigma Z), \nabla_\sigma m + Z\otimes I_d\ra]\|
\eeqs
Here, the quantities inside of the $\|\cdot\|$ on the right are matrices. Indeed, $\la \nabla_\sigma G, A\ra$ denotes the application of $\nabla_\sigma G$ to $A$. Since $G$ sends matrices to matrices, $\nabla_\sigma G$ is a linear functional which also sends matrices to matrices. In the third line, $\nabla_\sigma(m(\sigma)+\sigma Z)$ should be interpreted as a linear functional from $\R^{d\times d}$ to $\R^d$, so $\la A, \nabla_\sigma(m(\sigma)+\sigma Z)\ra$ is a vector in $\R^d$, and the inner product of this vector with the $d\times d\times d$ tensor $\nabla^3W$ is a matrix.  Using that $\|\la T, x\ra\|\leq \|T\|\|x\|$, as explained at the beginning of this section, we have
\beqs\label{nabla3}
\bigg\|\bigg\la\nabla^3 W, \;&\big\la A,\; \nabla_\sigma\l(m(\sigma)+\sigma Z\r)\big\ra\bigg\ra\bigg\|\leq \l \|\nabla^3W\r\| \l \| \lla A, \nabla_\sigma\l(m(\sigma)+\sigma Z\r)\rra\r\|\\
& \leq \|\nabla^3W\|\|\nabla_\sigma(m(\sigma)+\sigma Z)\|_\op \\
&= \|\nabla^3W\|\|\nabla_\sigma m(\sigma) + Z\otimes I_d\|_\op \leq \|\nabla^3W\|(1 + \|Z\|).
\eeqs
To get the last bound, we used that $\|\nabla_\sigma m(\sigma)\|_\op\leq 1$, shown in Lemma~\ref{m-of-sig}. We also use the fact that $\|Z\otimes I_d\|_\op = \sup_{\|A\|=1}\|\lla A, \, Z\otimes I_d\rra\| = \sup_{\|A\|=1}\|AZ\| = \|Z\|$. (Recall that since $Z\otimes I_d$ is part of $\nabla_\sigma m$, we are considering $Z\otimes I_d$ as an operator on matrices rather than as a $d\times d\times d$ tensor, and this is why we take the supremum over matrices $A$.)

Substituting~\eqref{nabla3} back into~\eqref{Gsig-diff}, we get
\beqs\label{Gsigdiff}
\|G(\sigma_1)-G(\sigma_2)\| &\leq \|\sigma_1 - \sigma_2\|\sup_{\sigma\in S_{0,\sqrt2}}\E\l[\l \|\nabla^3W\l(m(\sigma)+\sigma Z\r)\r\|(1+\|Z\|)\r] \\
&\leq \|\sigma_1 - \sigma_2\|C(q)a_3\sqrt{\frac dn} \leq \frac{1}{2\sqrt 2}\|\sigma_1-\sigma_2\|
\eeqs
%\les \|\sigma_1 - \sigma_2\|\bar \c3\frac{\sqrt d}{\sqrt n}\leq \frac{1}{2\sqrt2}\|\sigma_1-\sigma_2\|.\eeqs
The second inequality is by Cauchy-Schwarz and~\eqref{Enabla3W} of Lemma~\ref{lma:Wprops}. The third inequality holds if $a_3\sqrt{d/n} \leq 1/(2\sqrt2C(q))$, which is ensured by~\eqref{cconds} of Assumption~\ref{assume:glob}. Now, note that thanks to Lemma~\ref{m-of-sig}, both $\lambda_{\min}(G(\sigma_1))$ and $\lambda_{\min}(G(\sigma_2))$ are bounded below by $1/2$. Using Lemma~\ref{psd-sq-rt} and~\eqref{Gsigdiff}, we therefore have
\beqs\label{Fsig}
\|F(\sigma_1) - F(\sigma_2)\| &\leq \sqrt2\|G(\sigma_1)-G(\sigma_2)\|\\
&\leq  \frac{1}{2}\|\sigma_1 - \sigma_2\|.
\eeqs Hence $F$ is a strict contraction.\end{proof}
 \begin{lemma}\label{psd-sq-rt}
 Let $A_0$ and $A_1$ be psd, and $A_0^{1/2}$, $A_1^{1/2}$ their unique psd square roots. Assume without loss of generality that $\lambda_{\min}(A_0)\leq \lambda_{\min}(A_1)$. Then 
 $$\|A_1^{-1/2}-A_0^{-1/2}\|\leq \frac{\|A_1- A_0\|}{2\lambda_{\min}(A_0)^{3/2}}.$$
 \end{lemma}
 \begin{proof}
 First note that
 $$A_1^{-1/2}-A_0^{-1/2} = A_1^{-1/2}(A_0^{1/2} - A_1^{1/2})A_0^{-1/2}$$ and hence
 $$\|A_1^{-1/2}-A_0^{-1/2}\| \leq \|A_1^{-1/2}\|\|A_1^{-1/2}\|\|A_1^{1/2} - A_0^{1/2}\| \leq \frac{\|A_1^{1/2} - A_0^{1/2}\|}{\lambda_{\min}(A_0)}.$$
 Now, define $A_t = A_0 + t(A_1-A_0)$ and let $B_t = A_t^{1/2}$, where $B_t$ is the unique psd square root of $A_t$. We then have $\|A_1^{1/2} - A_0^{1/2}\| \leq \sup_{t\in[0,1]}\|\dot B_t\|$. We will now express $\dot B_t$ in terms of $\dot A_t$ and $B_t$. Differentiating $B_t^2=A_t$, we  get
 \beq
 B_t\dot B_t + \dot B_tB_t = \dot A_t = A_1 - A_0.
 \eeq
 Now, one can check that the solution $\dot B_t$ to this equation is given by
 $$\dot B_t = \int_0^\infty e^{-sB_t}(A_1-A_0)e^{-sB_t}ds$$ and hence
 $$\|\dot B_t\| \leq \|A_1-A_0\|\int_0^\infty \|e^{-sB_t}\|^2dt = \frac{\|A_1-A_0\|}{2\lambda_{\min}(B_t)} = \frac{\|A_1-A_0\|}{2\sqrt{\lambda_{\min}(A_t)}}.$$ Now note that $\lambda_{\min}(A_t)\geq \lambda_{\min}(A_0)$, since $A_t$ is just a convex combination of $A_0$ and $A_1$. Hence $\|\dot B_t\|\leq \|A_1-A_0\|/2\sqrt{\lambda_{\min}(A_0)}$ for all $t\in[0,1]$. Combining all of the above estimates gives
 $$\|A_1^{-1/2}-A_0^{-1/2}\| \leq \frac{\|A_1- A_0\|}{2\lambda_{\min}(A_0)^{3/2}}.$$
 \end{proof}
 \section{Auxiliary proofs}\label{supp:aux}We first prove the two change-of-variables results from the paper and then two other postponed lemmas, the proofs of which involve some technical calculations which are not particularly inspiring.\\
 
First, we show the equivalence of~\eqref{overallbd2} and~\eqref{overallbd2V}. It suffices to show $\int f(-p_3)d\gamma=\int gQd\hat\pi$. Recall that $\hat\pi=T_{\#}\gamma$ and $g=f\circ T$. Therefore, $\int f(-p_3)d\gamma = -\int gp_3\circ Td\hat\pi$, so it suffices to show $p_3\circ T=-Q$. Now,
\beqs
p_3(x)& =\lla\E\l[\nabla^3\Wbar(Z)\r],\;\; \frac16x^{\otimes 3} - \frac12x\otimes I_d\rra \\
&=\lla\E\l[\nabla^3V(\mvi\pi+\Svi\pi^{1/2}Z)\r], \;\;\frac16(\Svi\pi^{1/2}x)^{\otimes 3} - \frac12(\Svi\pi^{1/2}x)\otimes \Svi\pi\rra\\
&=\lla\E_{X\sim\hat\pi}\l[\nabla^3V(X)\r], \;\;\frac16(\Svi\pi^{1/2}x)^{\otimes 3} - \frac12(\Svi\pi^{1/2}x)\otimes \Svi\pi\rra.\eeqs
Therefore,
\beqsn
(p_3\circ T)(x)&=p_3(\Svi\pi^{-1/2}(x-\mvi\pi)) \\
&= \lla \E_{X\sim\hat\pi}\l[\nabla^3V(X)\r],\; \; \frac16(x-\mvi\pi)^{\otimes 3} - \frac12(x-\mvi\pi)\otimes \Svi\pi\rra = -Q(x),\eeqsn as desired.\\

Next, we assume the conclusion of Lemma~\ref{lma:m-sig-soln} to prove Lemma~\ref{lma:intro:exist:V}.
\begin{proof}[Proof of Lemma~\ref{lma:intro:exist:V}]
Let $m,\sigma$ be the solution from Lemma~\ref{lma:m-sig-soln}. Then
\beqsn
0&=\E[\nabla W(m+\sigma Z)] = H_V^{-1/2}\E\l[\nabla V\l(\mhat + H_V^{-1/2}(m+\sigma Z)\r)\r] ,\\
(\sigma\sigma^T)^{-1}&=\E[\nabla^2W(m+\sigma Z)] = H_V^{-1/2}\E\l[\nabla V\l(\mhat + H_V^{-1/2}(m+\sigma Z)\r)\r]H_V^{-1/2}
\eeqsn Let
$$\mvi\pi = \mhat + H_V^{-1/2}m,\qquad \Svi\pi=H_V^{-1/2}\sigma\sigma^TH_V^{-1/2}.$$ We conclude that
\beqs
0=\E[\nabla V(\mvi\pi + H_V^{-1/2}\sigma Z)],\\
\Svi\pi^{-1}=\E[\nabla^2V(\mvi\pi + H_V^{-1/2}\sigma Z)]
\eeqs Now, we can replace $H_V^{-1/2}\sigma$ (an asymmetric square root of $\Svi\pi$) in the above equations with $\Svi\pi^{1/2}$ (the symmetric square root of $\Svi\pi$) and the equations will still hold. Thus $\mvi\pi,\Svi\pi$ satisfies $\klopt{V}$. Furthermore, since $(m, \sigma)\in B_{2\sqrt2}(0,0)$ and $\sigma\in S_{\sqrt{2/3},\sqrt2}$, we get that
\beqs
\|H_V^{1/2}\Svi\pi H_V^{1/2}\| + \|H_V^{1/2}(\mvi\pi-\mhat)\|^2 = \|\sigma\sigma^T\| + \|m\|^2 = \|\sigma\|^2+\|m\|^2\leq 8,\eeqs and
$$
\frac23H^{-1}\preceq \Svi\pi=H_V^{-1/2}\sigma\sigma^TH_V^{-1/2}\preceq 2H^{-1}.$$
It remains to show that $\mvi\pi,\Svi\pi$ is the unique solution satisfying $\klopt{V}$ in the region $\mathcal R_V$. To show this is the case, note that if $(m, S)$ is a solution to $\klopt{V}$ in $\mathcal R_V$, then $\l(H^{1/2}( m-\mhat), (H^{1/2} SH^{1/2})^{1/2}\r)\in  B_{2\sqrt2}(0, 0)\cap\R^d \times S_{0,\sqrt2}$ is a solution to~\eqref{kl-opt-1},~\eqref{kl-opt-2}. Furthermore, the map
$$(m,S)\mapsto \l(H^{1/2}( m-\mhat), (H^{1/2} SH^{1/2})^{1/2}\r)$$ is injective. Since we know solutions to~\eqref{kl-opt-1},~\eqref{kl-opt-2} in $B_{2\sqrt2}(0, 0)\cap\R^d \times S_{0,\sqrt2}$ are unique by Lemma~\ref{lma:m-sig-soln}, it follows that solutions to $\klopt{V}$ in $\mathcal R_V$ must also be unique.
\end{proof}

We now turn to the technical calculations.

\begin{proof}[Proof of Lemma~\ref{lma:gsuff}]
If $g(x)=f(\Svi\pi^{-1/2}(x-\mvi\pi))$ then~\eqref{gcond} is satisfied provided $|f(y)-\gamma(f)|\leq e^{\frac{c_0}{4}\sqrt d\|y\|}$ for all $\|y\|\geq R_g\sqrt d$. We now find $R_g$ to ensure this inequality on $f$ is satisfied. Suppose WLOG that $f(0)=0$, and $|f(y)|\leq \e(C\sqrt d\|y\|^\alpha)$ for all $y\in\R^d$. We first bound $\gamma(|f|)$. We have
\beqs\label{gammabs}
\gamma(|f|)&\leq \max_{\|y\|\leq (4C_f+2\sqrt2)\sqrt d}|f(y)| + (2\pi)^{-d/2}\int_{\|y\|\geq(4C_f+2\sqrt2)\sqrt d}\e\l(C_f\sqrt d\|y\|-\|y\|^2/2\r)dy.\eeqs Now, we bound the two summands separately. For the first summand, we have
\beqs\label{gammabs1}  K_1&:=\max_{\|y\|\leq (4C_f+2\sqrt2)\sqrt d}|f(y)| \\
&\leq \e(C_f(4C_f+3)^\alpha(\sqrt d)^{1+\alpha}) \\
&\leq  \e(C_f(4C_f+3)^\alpha d)=:e^{\bar Cd},\eeqs where $$\bar C:=C_f(4C_f+3)^\alpha.$$
For the second summand in~\eqref{gammabs}, we have
\beqs\label{gammabs2}K_2&:= (2\pi)^{-d/2}\int_{\|y\|\geq(4C_f+2\sqrt2)\sqrt d}e^{-\|y\|^2/4}dy\\
&\leq  2^{\frac d2}(2\pi)^{-d/2}\int_{\|y\|\geq  2\sqrt d}e^{-\|y\|^2/2}dy\\
&\leq 2^{\frac d2}e^{-\frac d2}\leq1.\eeqs In the second line, we used that $C_f\sqrt d\|y\|-\|y\|^2/2 \leq \|y\|^2/4 - \|y\|^2/2= -\|y\|^2/4$ when $\|y\|\geq(4C_f+2\sqrt2)\sqrt d$. Using~\eqref{gammabs1} and~\eqref{gammabs2} in~\eqref{gammabs}, we get
\beq
\gamma(|f|)\leq K_1 + K_2\leq 2e^{\bar Cd}.\eeq Now, if $R_g\geq\frac{4}{c_0}\l(\bar C+\log 4\r)$, then
$$|\gamma(f)|\leq 2e^{\bar Cd}\leq \frac12e^{\frac{c_0}{4}R_gd}\leq \frac12e^{\frac{c_0}{4}\sqrt d\|x\|}\quad\forall \|x\|\geq R_g\sqrt d.$$
Next, we bound $|f(y)|$. Note that if $\|y\|\geq R_g\sqrt d$ and $R_g\geq1$, then
\beqs
C_f\sqrt d\|y\|^\alpha + \log 2 &= \l(\frac{C_f\sqrt d}{\|y\|^{1-\alpha}} + \frac{\log 2}{\|y\|}\r)\|y\|\\
&\leq \l(\frac{C_f\sqrt d}{R_g^{1-\alpha}\sqrt d^{1-\alpha}} + \frac{\log 2}{R_g\sqrt d}\r)\|y\|\\
&\leq \frac{(C_f+\log 2)\sqrt d}{R_g^{1-\alpha}}\|y\| \leq \frac{c_0}{4}\sqrt d\|y\|
\eeqs as long as $R_g\geq \l[\frac{4}{c_0}(C_f+\log 2)\r]^{\frac{1}{1-\alpha}}$. We then have
$$|f(y)|\leq \e(C_f\sqrt d\|y\|^\alpha) \leq \frac12\e\l(\frac{c_0}{4}\sqrt d\|y\|\r)\quad\forall \|y\|\geq R_g\sqrt d.$$ Combining all of the above estimates, we get that
$$
|f(y)-\gamma(f)|\leq |f(y)|+\gamma(|f|) \leq \e\l(\frac{c_0}{4}\sqrt d\|y\|\r)\quad\forall \|y\|\geq R_g\sqrt d
$$ as long as
$$
R_g\geq\max\l( \l[1\vee\frac{4}{c_0}(C_f+\log 2)\r]^{\frac{1}{1-\alpha}}, \frac{4}{c_0}\l(C_f(4C_f+3)^\alpha+\log 4\r)\r).$$ This is a function of $C_f,\alpha,c_0$ which is increasing with $C_f, \alpha,$ and $c_0^{-1}$, as desired.
\end{proof}
\begin{proof}[Proof of Lemma~\ref{lma:tail}]
Lemma E.3 of~\cite{katsTDD} shows that if $Rb> 1+p/d$, then
\beqs
I(R,p,b)\leq (eR)^pd^{\frac p2+1}\e\l(\l[\frac32+\log R-Rb\r]d\r)\eeqs
Massaging this upper bound, we get
$$
I(R,p,b)\les_p R^{p}d^{\frac p2+1}\e\l(\l[\frac32+\log R-Rb\r]d\r)\leq d^{\frac p2+1}\e\l(\l[\frac32+(p+1)\log R-Rb\r]d\r).
$$
Now, we have
$$(p+1)\log R=(2p+2)\log\sqrt R\leq (2p+2)\sqrt R \leq \frac{Rb}{4}$$ provided $R\geq (8p+8)^2b^{-2}$.  Therefore,
$$\frac32+(p+1)\log R-Rb\leq \frac32-\frac34Rb \leq -1-\frac12Rb$$ for $R\geq10b^{-1}$. Thus $I(R,p,b)\les_p d^{\frac p2 +1}e^{-d-Rbd/2}\les_p e^{-Rbd/2}$. The lower bounds on $R$ are met provided
$$R\geq\max\l((1+p)b^{-1}, (8p+8)^2b^{-2},10b^{-1}\r),$$ which can be achieved by taking $R=C_p(1\vee b^{-1})^2$ for some $C_p$ depending only on $p$. \end{proof}

\bibliographystyle{plain}
\bibliography{bibliogr_VI} 
\end{document}